\PassOptionsToPackage{usenames, dvipsnames}{xcolor}
\documentclass[a4paper, 11pt]{preprint}
\pdfoutput=1

\usepackage[usenames, dvipsnames]{xcolor}

\usepackage{tikz}

\usepackage{enumitem}
\usepackage{amsfonts}
\usepackage{amsmath}
\usepackage{hyperref}

\usepackage{natbib} 
\usepackage{caption}
\usepackage{subcaption}

\usepackage{url}

\newcommand{\xx}{\mathbf{x}}
\newcommand{\yy}{\mathbf{y}}

\newcommand{\GG}{\mathcal{G}}
\newcommand{\HH}{\mathcal{H}}

\newcommand{\MM}{\mathcal{M}}

\newcommand{\RR}{\mathcal{R}}
\newcommand{\ci}{\mathcal{I}}
\newcommand{\J}{\mathcal{J}}

\newcommand{\R}{\mathbb{R}}
\newcommand{\C}{\mathbb{C}}

\newcommand{\pa}{\RM{pa}}

\newcommand{\vX}{{\bf X}}
\newcommand{\vx}{{\bf x}}

\DeclareMathOperator{\bernoulli}{\textrm{Ber}}

\newcommand\independent{\protect\mathpalette{\protect\independenT}{\perp}}
\def\independenT#1#2{\mathrel{\rlap{$#1#2$}\mkern2mu{#1#2}}}

\title{Representation of Context-Specific Causal Models with Observational and Interventional Data}
\author{Eliana Duarte and Liam Solus}

\address[L.~Solus]{Department of Mathematics, KTH Royal Institute of Technology, Sweden}
\email{solus@kth.se}

\address[E.~Duarte]{Departmento de Matem\'atica, Faculdade de Ci\^encias, Universidade do Porto, Portugal}
\email{eliana.gelvez@fc.up.pt}
\date{\today}


\begin{document}

\begin{abstract}
We address the problem of representing context-specific causal models based on both observational and experimental data
collected under general (e.g. hard or soft) interventions by introducing a new family of context-specific conditional independence models
called CStrees. This family is defined via a novel factorization criterion that allows for a generalization of the
factorization property defining general interventional DAG models. We derive a graphical characterization of model
equivalence for observational CStrees that extends the Verma and Pearl criterion for DAGs. This characterization is then extended to
CStree models under general, context-specific interventions. To obtain these results, we formalize a notion of context-specific
intervention that can be incorporated into concise graphical representations of CStree models. We relate CStrees to other context-specific models, showing that the families of DAGs, CStrees, labeled DAGs and staged trees form a strict chain of inclusions. We end with an application of interventional CStree models to a real data set, revealing the context-specific nature of the data dependence structure and the soft, interventional perturbations.
\end{abstract}

\maketitle

\section{Introduction}
\label{sec:introduction}

We study the problem of representing causal relations that hold amongst jointly distributed categorical variables when observational data and data from, possibly soft and context-specific, interventions is available. 
To do so, we specify a family of context-specific conditional independence models defined via a factorization criterion that directly allows for general, context-specific interventions extending the interventional DAG models studied by \citet{YKU18}.
Given a vector $\vX = (X_1,\ldots, X_p)$ of jointly distributed categorical variables, we say that $\vX$ is \emph{Markov} to a directed acyclic graph (DAG) $\mathcal{G} = ([p], E)$ with node set $[p] = \{1,\ldots, p\}$ and edges $E$ if 
\begin{equation}
\label{eqn:DAGfactorization}
P(\vx) = \prod_{i=1}^pP(x_i | \vx_{\pa_{\mathcal{G}}(i)}) \qquad \mbox{for all outcomes $\vx=(x_1,\ldots, x_p)$ of $\vX$},
\end{equation}
where $\pa_\GG(i) = \{j\in[p] : j\rightarrow i\in E\}$ denotes the \emph{parents} of $i$ in $\GG$.
The \emph{DAG model} for $\mathcal{G}$, denoted $\mathcal{M}(\mathcal{G})$ is the collection of all $\vX$ that are Markov to $\mathcal{G}$. 

From the perspective of causality, we interpret edges $i\rightarrow j$ in the DAG $\mathcal{G}$ as representing that $X_i$ has a direct causal effect on $X_j$ in the data-generating distribution $\vX$. 
However, it is well-known that the  causal structure $\mathcal{G}$ of the distribution $\vX\in \mathcal{M}(\mathcal{G})$ is generally not identifiable from a random sample alone. 
Specifically, it is possible that two distinct DAGs $\mathcal{G}$ and $\mathcal{H}$ satisfy $\mathcal{M}(\mathcal{G}) = \mathcal{M}(\mathcal{H})$; a phenomenon known as \emph{Markov equivalence}.  
Hence, with only a random sample, one cannot distinguish the data-generating causal structure from the other DAGs within its \emph{Markov equivalence class} (MEC); e.g., the set of all DAGs to which it is Markov equivalent. 

While the entire causal structure cannot be recovered from a random sample alone, characterizations of Markov equivalence show that some causal directions can be identified \citep{VP92}. 
To better learn the complete causal structure, the gold standard approach is to use additional data drawn from interventional distributions; e.g., distributions that arise from augmenting the role played by subsets of variables in the causal system. 
Namely, given a subset $I\subseteq[p]$, called an \emph{intervention target}, we define an \emph{interventional distribution} $\vX^{(I)}$ for $I$ and $\vX\in \mathcal{M}(\mathcal{G})$ as a distribution having probability mass function satisfying 
\begin{equation}
\label{eqn:I-DAGfactorization}
P^{I}(\vx) = \prod_{i\in I}P^{I}(x_i \mid \vx_{\pa_{\mathcal{G}}(i)})\prod_{i\notin I}P(x_i \mid \vx_{\pa_{\mathcal{G}}(i)}) \, \, \,\, \mbox{for all outcomes $\vx = (x_1,\ldots, x_p)$,}
\end{equation}
where $P(x_i \mid \vx_{\pa_{\mathcal{G}}(i)})$ is the conditional factor appearing in~\eqref{eqn:DAGfactorization}. 
Given a sequence of interventional targets $\mathcal{I} = (I_0 :=\emptyset, I_1,\ldots, I_K)$, the $\mathcal{I}$-DAG model for the pair $(\mathcal{G}, \mathcal{I})$ is the collection of sequences of distributions
\begin{equation}
\label{eqn:I-DAGmodel}
\begin{split}
\mathcal{M}(\mathcal{G},\mathcal{I}) = \{(\vX^{I_0},\ldots,\vX^{I_K}) :\, &\mbox{ for all } k \in\{0,\ldots, K\}, \vX^{I_k}\in\mathcal{M}(\mathcal{G}) \mbox{ and for all outcomes $\vx$,} \\
&\,\, P^{I_k}(x_i \mid \vx_{\pa_{\mathcal{G}}(i)}) = P^{I_{k'}}(x_i \mid \vx_{\pa_{\mathcal{G}}(i)}) \mbox{ whenever } i\notin I_k \cup I_{k'}\}.
\end{split}
\end{equation}
Here, we set $I_0 := \emptyset$, corresponding to the \emph{observational distribution} in~\eqref{eqn:DAGfactorization}. 
The invariances in the conditional factors of $P^{I_0}(\vx)$ and $P^{I_k}(\vx)$ for the variables not targeted in the intervention allow for the recovery of additional causal relations, depending on the target sets $\mathcal{I}$. 
Namely, two DAGs are called \emph{$\mathcal{I}$-Markov equivalent} whenever $\mathcal{M}(\mathcal{G},\mathcal{I}) = \mathcal{M}(\mathcal{H}, \mathcal{I})$. 
The $\mathcal{I}$-Markov equivalence classes refine Markov equivalence classes by using the additional invariance information to fix the direction of more edges in the DAGs. 

Characterizations of interventional Markov equivalence were first considered by \citet{HB12}, under the assumption of \emph{perfect interventions}, in which we assume the causal relations in the conditional factors are destroyed under the intervention; e.g., $P^{I_k}(x_i \mid \vx_{\pa_{\mathcal{G}}(i)}) = P^{I_k}(x_i)$. 
More recently, \citet{YKU18} gave a complete characterization of model equivalence for the $\mathcal{I}$-DAG models $\mathcal{M}(\mathcal{G},\mathcal{I})$ defined above.
Their result generalizes the characterization of Hauser and B\"uhlmann to \emph{general interventions}. 
These additionally allow for \emph{soft interventions} \citep{eaton2007exact}, sometimes called \emph{mechanism changes} \citep{tian2013causal}, under which the causal relations need not be destroyed.

As seen in the factorization definition~\eqref{eqn:DAGfactorization} of a DAG model $\mathcal{M}(\mathcal{G})$, a DAG model is defined by a collection of CI relations of the form $X_i \independent \vX_{[i]\setminus \pa_{\mathcal{G}}(i)} \mid \vX_{\pa_{\mathcal{G}}(i)}$, and the corresponding $\mathcal{I}$-DAG model $\mathcal{M}(\mathcal{G},\mathcal{I})$ is defined via a natural extension of this definition. 
In particular, a DAG model $\mathcal{M}(\mathcal{G})$ is a \emph{conditional independence model} \citep{SL14}, i.e., the collection of distributions satisfying a pre-specified set of CI relations, and an $\mathcal{I}$-DAG model is an extension of a conditional independence model that includes a set of conditional invariance relations.
These relations are compactly represented in the DAG structure, allowing one to easily read off causal information. 
On the other hand, the DAG and $\mathcal{I}$-DAG models only capture conditional independence and invariance relations, and thus could overlook important causal relations that only hold for specific outcomes of certain variables in the system.

Given disjoint subsets $A,B,C,S\subseteq[p]$ we say that $\vX_A$ is \emph{conditionally independent of} $\vX_B$ given $\vX_C$ \emph{in the context} $\vX_S = \vx_S$ if 
\begin{equation}
\label{eqn:CSI}
P(\vx_A |\vx_B, \vx_C,\vx_S) = P(\vx_A | \vx_C,\vx_S) 
\end{equation}
for all marginal outcomes $\vx_A,\vx_B$ and $\vx_C$. 
When~\eqref{eqn:CSI} is satisfied, we write $\vX_A \independent \vX_B \mid \vX_C, \vX_S = \vx_S$ and we call this a \emph{context-specific conditional independence relation}, or \emph{CSI relation} for short.
CSI relations arise naturally in a wide variety of modeling problems, as described in \citep{poole2003exploiting}. 
The following is an example where (soft) interventions may also be context-specific.

\begin{example}
    \label{ex:CSI causal model}
    (Adapted from \citep[Example~5]{poole2003exploiting})
    When a child arrives at the hopsital, the staff may want to determine if they are carrying chicken pox. 
    If the child has not recently been exposed, they are likely not a carrier.  
    Hence, carrier status may be independent of all other background factors, given that the child has no recent exposure.
    Given the child has been exposed and has no previous diagnosis then they are likely a carrier regardless of other background factors. 
    Similarly, given that the child has a previous diagnosis, recent exposure may be independent of other background factors, as the child's guardians may be more lax in avoiding further exposure. 
    We let $X_2, X_3$ and $X_4$ denote, respectively, previous diagnosis, recent exposure and carrier.
    Suppose that we have one additional background factor $X_1$, representing family income level (low or high). 
    Given that $X_1$ and $X_2$ may covary due to healthcare costs, we consider the resulting context-specific conditional independence model:
    
    \[
    \mathcal{D} = \{ 
    X_4\independent \vX_{1,2} \mid X_3 = \mbox{no},\, X_4 \independent X_1 \mid  \vX_{2,3} = (\mbox{no}, \mbox{yes}),\, X_3 \independent X_1 \mid X_2 = \mbox{yes}
    \}.
    \]
    Note that the CSI relations in $\mathcal{D}$ cannot be exactly represented by a DAG.
    An example of a context-specific, general (in this case, soft) intervention in this system could be the result of a certain local government program subsidizing healthcare costs for low income families, resulting in a mechanism change replacing $P^{I_0}(X_2 \mid X_1 = \mbox{low})$ with $P^{I_1}(X_2 \mid X_1 = \mbox{low})$. 
    Similarly, another municipality may install public school programs aimed at reducing close-contact with diagnosed children, inducing a mechanism change $P^{I_2}(X_3 \mid X_2 = \mbox{yes})$.
\end{example}

While CSI relations may be important to a given modeling problem, encoding such relations complicates the task of giving concise representations of the causal model. 
In this paper, we present a family of models for which one can produce concise graphical representations of these causal relations in the presence of context-specific information. 
Our models are based on a context-specific generalization of the factorization definition of a DAG model given in~\eqref{eqn:DAGfactorization}. 
Hence, these models admit a straightforward generalization of~\eqref{eqn:I-DAGfactorization} to models for general, context-specific interventions, as in Example~\eqref{eqn:CSI}.
We provide generalizations of the characterization of model equivalence for DAGs due to \citet{VP92} as well as the characterization of $\mathcal{I}$-Markov equivalence of \citet{YKU18} to the context-specific setting.
We also relate our models and results to previously studied families of context-specific models, including the staged tree models of \citet{SA08} and the labeled DAG models of \citet{PNKC15}.
Particularly, we show that our models are the subfamily of the LDAG models of \citet{PNKC15} that admit a factorization which easily generalizes~\eqref{eqn:I-DAGfactorization} when provided with data from soft, context-specific interventions.
Finally, we apply these models to a real data set, demonstrating how the model equivalence characterization may capture the context-specific nature of soft interventional perturbations while also revealing when the causal structure is fully identifiable.

\section{Related Work}
\label{sec:relatedwork}

A variety of different models for context-specific conditional independence have been proposed. 
These models range from including somewhat limited context-specific information, yielding close generalizations of DAG models, to models capable of encoding a multitude of context-specific relations.
Examples of the former models include \emph{Bayesian multinets} \citep{GH96} and \emph{similarity networks} \citep{H91}, for which the contexts $\vx_S$ considered in~\eqref{eqn:CSI} are limited to outcomes of a single variable called the \emph{hypothesis variable}. 
Since the CSI relations in these models are relatively controlled, they tend to directly inherit many of the niceties of DAG models at the cost of limited context-specific information. 

At the other end of the spectrum are the \emph{staged tree models} of \citet{SA08}.
Staged trees are perhaps the broadest model family for encoding context-specific information.
They amount to a colored probability trees in which the colors are used to encode equalities of conditional distributions.
Since the graphical representation of a staged tree model is based on a probability tree, the staged tree representation may be difficult to interpret.  
This is, in part, due to the fact that the number of edges and vertices in the graph grows on the order of $2^{p+1}$ for even $p$ binary variables. 
Hence, drawing or storing these graphs for even few variables is difficult, and reading the captured CSI relations from the coloring can be even more challenging. 
The \emph{chain event graphs} described in \citep{CGS18} offer an alternative representation of a staged tree that can reduce complexity, but these may still be complex and difficult to interpret. 

The characterization of model equivalence for staged trees has been studied. 
However, complete characterizations exist for only some special cases \citep{GS18, GBRS18}. 
Moreover, the complexity of staged tree models is naturally reflected in these characterizations, which tend to be much more technical than the simple characterizations of model equivalence known for DAGs. 
In a similar fashion, staged tree representations for context-specific interventions, both hard and soft, have been studied \citep{RS07, T08, TSR10, DS20}. 
However, the problem of characterizing interventional staged tree model equivalence is yet to be studied. 

The family of CStree models introduced in this paper is created by using staged tree representations, but limits model complexity by
restricting the types of CSI relations they encode. 
In doing so, we obtain a family of context-specific models that allow for the incorporation of more CSI relations than Bayesian multinets and similarity networks, while sufficiently reducing the complexity of staged trees so that one may obtain reasonable characterizations of model equivalence in both the observational and context-specific, interventional settings.

A more moderately complex family of context-specific models called \emph{labeled DAG models}, or LDAGs, were introduced by \citet{PNKC15}. 
An LDAG is a pair $(\mathcal{G}, \mathcal{L})$ where $\mathcal{G} = ([p], E)$ is a DAG and $\mathcal{L}$ is a set of \emph{labels}, one for each edge $j\rightarrow i$ of $\mathcal{G}$. 
The label $L_{j,i}$ for the edge $j\rightarrow i$ is a subset of the outcomes of $\vX_{\pa_{\mathcal{G}}(i)\setminus j}$. 
The \emph{LDAG model} for $(\mathcal{G}, \mathcal{L})$ is the collection of distributions
\begin{equation}
\begin{split}
    \label{eqn:LDAGmodel}
    \mathcal{M}(\mathcal{G},\mathcal{L}) = \{\vX :\, &\vX\in\mathcal{M}(\mathcal{G}) \mbox{ and }\\ 
    &X_i\independent X_j \mid \vX_{\pa_{\mathcal{G}}(i)\setminus j} = \vx_{\pa_{\mathcal{G}}(i)\setminus j} \mbox{ for all } \vx_{\pa_{\mathcal{G}}(i)\setminus j}\in L_{j,i}, \mbox{ for all $j\rightarrow i\in E$}\}.
\end{split}
\end{equation}
Hence, LDAGs are a context-specific generalization of DAG models in which the model-defining context-specific relations are pairwise CSI relations of the form $X_i\independent X_j \mid \vX_{\pa_{\mathcal{G}}(i)\setminus j} = \vx_{\pa_{\mathcal{G}}(i)\setminus j}$. 
One advantage of LDAGs is that they immediately admit a concise graphical representation. 
Specifically, the LDAG representation of $(\mathcal{G}, \mathcal{L})$ simply amounts to drawing the DAG $\mathcal{G}$ with the set $L_{j,i}$ as a label on the edge $j\rightarrow i$, where one omits this label whenever $L_{j,i} = \emptyset$. 
Certain CSI relations of the form~\eqref{eqn:CSI} are then easily read from this representation by deleting the edges $j\rightarrow i$ for which $\vx_S$ restricts to an element of $L_{j,i}$ and then considering d-separation as for standard DAG models. 

As noted by \citet{tikka2019identifying}, LDAGs also encode \emph{hard interventions}, e.g., interventions in which $P^{I_k}(x_i | \vx_{\pa_{\mathcal{G}}(i)}) = P^{I_k}(x_i)$ for all outcomes $x_i$ and $\vx_{\pa_{\mathcal{G}}(i)}$ by analogously deleting the edges pointing into node $i$.  
One could then apply this reasoning to a characterization of LDAG Markov equivalence in the observational case \citep[Theorem~4]{PNKC15} to obtain a direct extension of the model for hard interventions in DAGs given by  \citet{HB12}.
On the other hand, the characterization of \citet{YKU18} extends the Hauser and B\"uhlmann result to a characterization of model equivalence under general (e.g., hard and soft) interventions by way of describing the interventional model via the extension~\eqref{eqn:I-DAGfactorization} of the DAG factorization~\eqref{eqn:DAGfactorization}. 
However, LDAG models are defined via the \emph{pairwise} CSI relations $X_j\independent X_i \mid \vX_{\pa_{\mathcal{G}}(j)\setminus i} = \vx_{\pa_{\mathcal{G}}(j)\setminus i}$ as opposed to a direct, context-specific generalization of the CI relations $X_i \independent \vX_{[i]\setminus \pa_{\mathcal{G}}(i)} \mid \vX_{\pa_{\mathcal{G}}(i)}$ corresponding to the factorization~\eqref{eqn:DAGfactorization}. 
Hence, a characterization of interventional Markov equivalence via LDAGs is not immediately apparent for general, context-specific interventions. 
Specifically, it is not known that the context-specific analogue of the conjecture of Hauser and B\"uhlmann proven in \citep[Corollary 3.12]{YKU18} holds for LDAG representations. 
To do so, one would first have to provide a definition of general interventional LDAG models extending~\eqref{eqn:I-DAGfactorization} and then derive a characterization of model equivalence that generalizes the result of \citet{PNKC15}. 
This is done for an appropriate family of LDAG models in this paper.

The CStree models we introduce are a subfamily of LDAG models whose definition is based on a direct, context-specific generalization of the factorization definition~\eqref{eqn:DAGfactorization} of a DAG model.
For these models, a definition of general interventional LDAG models extending~\eqref{eqn:I-DAGfactorization} is straightforward, allowing one to address the conjecture of \citet{HB12} in a context-specific setting.
%
Using the special structure of these models, we obtain characterizations of Markov equivalence that naturally extend to generalizations of interventional Markov equivalence under general, context-specific interventions. 
The result is that CStree models  are more general than DAG models, possess the desirable representations of LDAG models, and additionally admit general interventional model equivalence characterizations extending those of \citet{YKU18}.

\section{CStrees}
\label{sec:cstrees}
The models studied in this paper are a subfamily of LDAG models called CStrees. 
Their realization as LDAGs may be used to provide concise graphical representations.
However, to allow for extensions that incorporate general, context-specific interventions, CStrees are defined according to a factorization criterion akin to~\eqref{eqn:DAGfactorization}, as opposed to the standard LDAG definition via pairwise CSI relations in~\eqref{eqn:LDAGmodel}. 

\subsection{CStree models}
\label{subsec:cstrees}
A CStree model is a collection of joint categorical distributions $\vX = (X_1,\ldots, X_p)$ assigned to an ordered pair $(\pi, \mathbf{s})$ where $\pi$ is a variable ordering, called the \emph{causal order} and $\mathbf{s}$ is a collection of sets. The set $\mathbf{s}$ indexes a collection of CSI relations
and is defined as follows.

Suppose that $X_i$ has state space $\mathcal{X}_{i} = [d_i]$ for positive integers $d_1,\ldots, d_p > 1$ and $\vX$ has state space $\mathcal{X} = \prod_{i=1}^p\mathcal{X}_i$. 
For $S\subseteq[p]$, we let $\vX_S = (X_i : i\in S)$ denote the marginal distribution for the variables with indices in $S$, and we denote its state space by $\mathcal{X}_S$. 
Given a causal order $\pi = \pi_1\ldots\pi_p$ of the indices in $[p]$, we consider CSI relations of the form
\begin{equation}
\label{eqn:CStree relation}
    X_{\pi_i} \independent \vX_{[\pi_1:\pi_{i-1}]\setminus S} \mid \vX_S = \vx_S
\end{equation}
for some outcome $\vx_S\in\mathcal{X}_S$ and $S\subseteq [\pi_1:\pi_{i-1}] := \{\pi_1,\ldots, \pi_{i-1}\}$. 
To such a relation, we associate the set $\mathcal{S}_{\pi, i}(\vx_S)$ of all marginal outcomes $\vx_{\pi_1:\pi_{i-1}}\in\mathcal{X}_{\pi_1:\pi_{i-1}}$ that agree with $\vx_S$ in the indices $S$. 
Note that we allow $S = [\pi_1:\pi_{i-1}]$, in which case $\mathcal{S}_{\pi,i}(\vx_{\pi_1:\pi_{i-1}}) = \{\vx_{\pi_1:\pi_{i-1}}\}$ is a singleton corresponding to the vacuously satisfied CSI relation $X_{\pi_i} \independent \emptyset \mid \vX_{\pi_1:\pi_{i-1}} = \vx_{\pi_1:\pi_{i-1}}$.

For a given causal order $\pi$, and each $i\in[p]$, we let $\mathcal{D}_{\pi, i}$ be a set of CSI relations as in~\eqref{eqn:CStree relation}. 
Letting $\mathcal{D}= \mathcal{D}_{\pi, 1}\cup \cdots \cup \mathcal{D}_{\pi, p}$, we obtain a collection of distributions
\begin{equation}
\label{eqn:CSI model}
\mathcal{M}(\mathcal{D}) = \{ \vX : \vX \mbox{ satisfies all CSI relations in } \mathcal{D}\}.
\end{equation}
For each $\mathcal{D}_{\pi,i}$, we define the sets
\[
\mathbf{s}_i = \{ \mathcal{S}_{\pi, i}(\vx_S) : X_{\pi_i} \independent \vX_{[\pi_1:\pi_{i-1}]\setminus S} \mid \vX_S = \vx_S\in \mathcal{D}_{\pi, i}\} \mbox{ and } \mathbf{s} = \mathbf{s}_1\cup \cdots \cup \mathbf{s}_p.
\]
The set $\mathcal{M}(\mathcal{D})$ is a context-specific conditional independence model defined by the pair $(\pi, \mathbf{s})$.

\begin{definition}
    \label{def:CStree}
    The pair $(\pi, \mathbf{s})$ is a \emph{CStree} if for all $i\in[p]$, the set $\mathbf{s}_i$ is a partition of $\mathcal{X}_{\pi_i:\pi_{i-1}}$. Given a CStree $\mathcal{T} = (\pi, \mathbf{s})$, the CStree model $\mathcal{M}(\mathcal{T})$ is the model $\mathcal{M}(\mathcal{D})$ for the pair $(\pi, \mathbf{s})$. 
\end{definition}

We say that $\vX$ is \emph{Markov} to the CStree $\mathcal{T} = (\pi,\mathbf{s})$ if $\vX\in\mathcal{M}(\mathcal{T})$. 
Just as for the definition of ``Markov'' for DAG models, a distribution $\vX$ is Markov to a CStree $\mathcal{T}$ if and only if it satisfies a factorization analogous to~\eqref{eqn:DAGfactorization}. 
Specifically, since $\mathbf{s}_i$ partitions $\mathcal{X}_{\pi_1:\pi_{i-1}}$ for every $i\in[p]$, then each outcome $\vx_{\pi_1:\pi_{i-1}}$ may be mapped to the set $S$ for which $\vx_{\pi_1:\pi_{i-1}}\in \mathcal{S}_{\pi, i}(\vx_S)$:
\[
\pa_\mathcal{T}: \bigcup_{i\in[p]}\mathcal{X}_{\pi_1:\pi_{i-1}} \longrightarrow \{S: S\subseteq[p]\}; \qquad \pa_\mathcal{T}(\vx_{\pi_1:\pi_{i-1}}) = S, \mbox{ where } \vx_{\pi_1:\pi_{i-1}}\in\mathcal{S}_{\pi,i}(\vx_S).
\]
It follows that $\vX$ is Markov to $\mathcal{T}$ if and only if
\begin{equation}
    \label{eqn:CStreefactorization}
    P(\vx) = \prod_{i=1}^pP(x_i \mid \vx_{\pa_{\mathcal{T}}(\vx_{\pi_1:\pi_{i-1}})}) \qquad \mbox{for all outcomes $\vx = (x_1,\ldots, x_p)$.}
\end{equation}

\begin{remark}
    \label{rem: DAGs are CStrees}
If $\mathbf{s}_i = \{\mathcal{S}_{\pi, i}(\vx_{P_i}) : \vx_{P_i}\in\mathcal{X}_{P_i}\}$ for some $P_i$ for each $i\in[p]$, then the factorization~\eqref{eqn:CStreefactorization} reduces to the factorization~\eqref{eqn:DAGfactorization} for a DAG $\mathcal{G}$ where $\pa_{\mathcal{G}}(i) = P_i$, for all $i$. 
Hence, CStrees are a generalization of DAG models via a context-specific generalization of~\eqref{eqn:DAGfactorization}. 
In particular, the CSI relations in~\eqref{eqn:CStree relation} are a context-specific relaxation of the CI relations $X_{\pi_i}\independent \vX_{[\pi_1:\pi_{i-1}]\setminus\pa_{\mathcal{G}}(i)} \mid \vX_{\pa_{\mathcal{G}}(i)}$ that define the DAG model $\mathcal{M}(\mathcal{G})$. 
\end{remark}

As shown in Appendix~\ref{appsec: LDAG construction}, by repeated application of the \emph{context-specific decomposition property} \citep{CHKPV19} to the relations in $\mathcal{D}_\mathcal{T}$, one sees that all distributions in $\mathcal{M}(\mathcal{T})$ are also in $\mathcal{M}(\mathcal{G},\mathcal{L})$ for an appropriately defined LDAG $(\mathcal{G},\mathcal{L})$. 
More completely, we obtain the following result. 

\begin{theorem}
    \label{thm: containment}
    Let $\mathbb{D}$, $\mathbb{C}$, $\mathbb{L}$ and $\mathbb{S}$ denote the collections of DAG models, CStree models, LDAG models and staged tree models, respectively. 
    Then
    \[
    \mathbb{D}\subsetneq\mathbb{C}\subsetneq\mathbb{L}\subsetneq\mathbb{S}.
    \]
\end{theorem}

As shown in the details in Appendix~\ref{appsec: observational proofs}, the containments in Theorem~\ref{thm: containment} are in fact strict. 
Figure~\ref{fig:LDAGnotCStree} shows examples for each strict inequality. 
The interpretation of the staged tree graphs is described below, with further details in Remark~\ref{rem: why not general LDAGs} and Appendix~\ref{appsec: observational proofs}.

\begin{figure}[t]
    \begin{subfigure}[b]{0.3\textwidth}
    \centering
    \begin{tikzpicture}[thick,scale=0.4]
		\node[circle, draw, fill=black!0, inner sep=1pt, minimum width=1pt] (1) at (4,4) {\large$1$};
		\node[circle, draw, fill=black!0, inner sep=1pt, minimum width=1pt] (2) at (-4,4) {\large$2$};
		\node[circle, draw, fill=black!0, inner sep=1pt, minimum width=1pt] (3) at (0,0) {\large$3$};

		\draw[->]   (1) -- node[midway,sloped,above]{${\{(0)\}}$} (3) ;
            \draw[->]   (2) -- node[midway,sloped,above]{${\{(1)\}}$} (3) ;
	\end{tikzpicture}
	\caption{An LDAG encoding the relations $X_1 \independent X_2, X_3 \independent X_1 \mid X_2 = 0$ and $X_3 \independent X_1 \mid X_2 = 1$.}\label{fig:LDAGofCStreeNOTLDAG}
    \end{subfigure}
    \hfill
    \begin{subfigure}[b]{0.3\textwidth}
    \centering
    \begin{tikzpicture}[thick,scale=0.15]

		\node[draw, fill=blue!0, inner sep=2pt, rounded corners, minimum width=2pt] (w1) at (-2,14.25) {\scriptsize 111};
		\node[draw, fill=cyan!0, inner sep=2pt, rounded corners, minimum width=2pt] (w2) at (-2,11.25) {\scriptsize 110};
		\node[draw, fill=orange!0, inner sep=2pt, rounded corners, minimum width=2pt] (v1) at (-2,8.25) {\scriptsize 101};
		\node[draw, fill=cyan!0, inner sep=2pt, rounded corners, minimum width=2pt] (v2) at (-2,5.25) {\scriptsize 100};
		\node[draw, fill=red!0, inner sep=2pt, rounded corners, minimum width=2pt] (w1i) at (-2,2.25) {\scriptsize 011};
		\node[draw, fill=cyan!0, inner sep=2pt, rounded corners, minimum width=2pt] (w2i) at (-2,-0.75) {\scriptsize 010};
		\node[draw, fill=orange!0, inner sep=2pt, rounded corners, minimum width=2pt] (v1i) at (-2,-3.75) {\scriptsize 001};
		\node[draw, fill=cyan!0, inner sep=2pt, rounded corners, minimum width=2pt] (v2i) at (-2,-6.75) {\scriptsize 000};

		\node[draw, fill=green!60, inner sep=2pt, rounded corners, minimum width=2pt] (w) at (-8,12.75) {\scriptsize 11};
		\node[draw, fill=cyan!0, inner sep=2pt, rounded corners, minimum width=2pt] (v) at (-8,6.75) {\scriptsize 10};
		\node[draw, fill=green!60, inner sep=2pt, rounded corners, minimum width=2pt] (wi) at (-8,0.75) {\scriptsize 01};
		\node[draw, fill=green!60, inner sep=2pt, rounded corners, minimum width=2pt] (vi) at (-8,-5.25) {\scriptsize 00};

		\node[draw, fill=yellow!60, inner sep=2pt, rounded corners, minimum width=2pt] (r) at (-14,9.75) {\scriptsize 1};
		\node[draw, fill=yellow!60, inner sep=2pt, rounded corners, minimum width=2pt] (ri) at (-14,-1.75) {\scriptsize 0};

		\node[draw, fill=black!0, inner sep=2pt, rounded corners, minimum width=2pt] (I) at (-20,3) {\scriptsize r};

		\draw[->]   (I) --    (r) ;
		\draw[->]   (I) --   (ri) ;

		\draw[->]   (r) --   (w) ;
		\draw[->]   (r) --   (v) ;

		\draw[->]   (w) --  (w1) ;
		\draw[->]   (w) --  (w2) ;


		\draw[->]   (v) --  (v1) ;
		\draw[->]   (v) --  (v2) ;


		\draw[->]   (ri) --   (wi) ;
		\draw[->]   (ri) -- (vi) ;

		\draw[->]   (wi) --  (w1i) ;
		\draw[->]   (wi) --  (w2i) ;


		\draw[->]   (vi) --  (v1i) ;
		\draw[->]   (vi) --  (v2i) ;


		\node at (-17.5,-9) {$X_1$} ;
		\node at (-11.5,-9) {$X_2$} ;
		\node at (-5,-9) {$X_3$} ;

	\end{tikzpicture}
	\caption{Staged tree representation of the LDAG in Figure~\ref{fig:LDAGofCStreeNOTLDAG}.}\label{fig:CStreeofLDAGnotCStree}
    \end{subfigure}
    \hfill
    \begin{subfigure}[b]{0.3\textwidth}
    \centering
    \begin{tikzpicture}[thick,scale=0.15]
		\node[draw, fill=blue!0, inner sep=2pt, rounded corners, minimum width=2pt] (w1) at (-2,14.25) {\scriptsize 111};
		\node[draw, fill=cyan!0, inner sep=2pt, rounded corners, minimum width=2pt] (w2) at (-2,11.25) {\scriptsize 110};
		\node[draw, fill=orange!0, inner sep=2pt, rounded corners, minimum width=2pt] (v1) at (-2,8.25) {\scriptsize 101};
		\node[draw, fill=cyan!0, inner sep=2pt, rounded corners, minimum width=2pt] (v2) at (-2,5.25) {\scriptsize 100};
		\node[draw, fill=red!0, inner sep=2pt, rounded corners, minimum width=2pt] (w1i) at (-2,2.25) {\scriptsize 011};
		\node[draw, fill=cyan!0, inner sep=2pt, rounded corners, minimum width=2pt] (w2i) at (-2,-0.75) {\scriptsize 010};
		\node[draw, fill=orange!0, inner sep=2pt, rounded corners, minimum width=2pt] (v1i) at (-2,-3.75) {\scriptsize 001};
		\node[draw, fill=cyan!0, inner sep=2pt, rounded corners, minimum width=2pt] (v2i) at (-2,-6.75) {\scriptsize 000};

		\node[draw, fill=green!60, inner sep=2pt, rounded corners, minimum width=2pt] (w) at (-8,12.75) {\scriptsize 11};
		\node[draw, fill=cyan!0, inner sep=2pt, rounded corners, minimum width=2pt] (v) at (-8,6.75) {\scriptsize 10};
		\node[draw, fill=green!0, inner sep=2pt, rounded corners, minimum width=2pt] (wi) at (-8,0.75) {\scriptsize 01};
		\node[draw, fill=green!60, inner sep=2pt, rounded corners, minimum width=2pt] (vi) at (-8,-5.25) {\scriptsize 00};

		\node[draw, fill=yellow!60, inner sep=2pt, rounded corners, minimum width=2pt] (r) at (-14,9.75) {\scriptsize 1};
		\node[draw, fill=yellow!60, inner sep=2pt, rounded corners, minimum width=2pt] (ri) at (-14,-1.75) {\scriptsize 0};

		\node[draw, fill=black!0, inner sep=2pt, rounded corners, minimum width=2pt] (I) at (-20,3) {\scriptsize r};

		\draw[->]   (I) --    (r) ;
		\draw[->]   (I) --   (ri) ;

		\draw[->]   (r) --   (w) ;
		\draw[->]   (r) --   (v) ;

		\draw[->]   (w) --  (w1) ;
		\draw[->]   (w) --  (w2) ;

		\draw[->]   (v) --  (v1) ;
		\draw[->]   (v) --  (v2) ;

		\draw[->]   (ri) --   (wi) ;
		\draw[->]   (ri) -- (vi) ;

		\draw[->]   (wi) --  (w1i) ;
		\draw[->]   (wi) --  (w2i) ;

		\draw[->]   (vi) --  (v1i) ;
		\draw[->]   (vi) --  (v2i) ;

		\node at (-17.5,-9) {$X_1$} ;
		\node at (-11.5,-9) {$X_2$} ;
		\node at (-5,-9) {$X_3$} ;

	\end{tikzpicture}
	\caption{A staged tree that is not an LDAG.}\label{fig:notLDAG}
    \end{subfigure}
    
\caption{An LDAG that is not a CStree and a staged tree that is not an LDAG}
\label{fig:LDAGnotCStree}.
\end{figure}

Theorem~\ref{thm: containment} shows that every CStree may be represented with an LDAG, which is perhaps the most compact and interpretable representation of the model. 
However, to provide the desired extensions of CStree models to models for general, context-specific interventions, we will also use a more comprehensive representation of the CSI relations defining a CStree $\mathcal{T} = (\pi,\mathbf{s})$. 
Specifically, consider the rooted tree, also denoted by $\mathcal{T}$ for convenience, on vertex set $\{r\}\cup \bigcup_{i\in[p]}\mathcal{X}_{\pi_1:\pi_{i-1}}$ and edges $\vx_{\pi_1:\pi_{i-1}}\rightarrow \vx_{\pi_1:\pi_{i-1}x_{\pi_i}}$ for all $\vx_{\pi_1:\pi_{i-1}}\in\mathcal{X}_{\pi_1:\pi_{i-1}}$ for all $i\in\{2,\ldots, p\}$, and $r\rightarrow x_{\pi_1}$ for all $x_{\pi_1}\in\mathcal{X}_{\pi_1}$. 
We then color the nodes of $\mathcal{T}$ such that two nodes are the same color if and only if they belong to the same set $\mathcal{S}_{\pi,i}(\vx_S)\in\mathbf{s}$. 
We use the convention that all nodes contained in singleton sets $\mathcal{S}_{\pi,i}(\vx_S)$ are white; e.g., all white nodes are assumed to be distinctly colored. 

\begin{example}
    \label{ex: cstree}
    Consider the model $\MM$ consisting of all joint distributions $\vX = (X_1,X_2,X_3,X_4)$ for binary $X_i$ satisfying the CSI relations in $\mathcal{D}$ from Example~\ref{ex:CSI causal model}. 
    Here we let the outcomes ``no'' and ``yes'' of the variables $X_2,X_3,X_4$ correspond to $0$ and $1$, respectively. 
    Similarly, the outcomes ``low'' and ``high'' of variable $X_1$ correspond to $0$ and $1$, respectively.

    We use the causal order $\pi = 1234$, which is captured by the ordering of the variables under the binary rooted tree in Figure~\ref{fig:cstree}.
    Given this ordering the set $\mathcal{D}$ is the union of the sets 
    \[
    \mathcal{D}_{\pi,3} = \{X_3 \independent X_1 \mid X_2 = 1\} \, \mbox{ and } \, \mathcal{D}_{\pi,4} = \{X_4\independent \vX_{1,2} \mid X_3 = 0,\, X_4 \independent X_1 \mid  \vX_{2,3} = (0,1)\}.
    \]  
    The single relation in $\mathcal{D}_{\pi,3}$ defines the set of outcomes $\mathcal{S}_{\pi,3}(\vx_2 = 1) = \{(0,1),(1,1)\}$.
    These two nodes are colored green in Figure~\ref{fig:cstree} to represent the equality of conditional probabilities $P(X_3 \mid \vX_{1,2} = (0,1)) = P(X_3 \mid \vX_{1,2} = (1,1))$ corresponding to the CSI relation $X_3 \independent X_1 \mid X_2 = 1$. 

    Analogously, the sets $\{(0,0,1),(0,1,0),(1,0,0),(1,1,0)\}$ (blue) and $\{(0,0,1),(1,0,1)\}$ (orange) are colored to represent the first and second relations in $\mathcal{D}_{\pi,4}$, respectively.
    All remaining nodes $\vx_{\pi_1:\pi_{i-1}}$ and $\vx_{\pi_1:\pi_{i-1}}^\prime$ are colored white in Figure~\ref{fig:cstree} to indicate that the model does not assume any further relations of the form $P(X_{\pi_{i}} \mid \vx_{\pi_1:\pi_{i-1}}) = P(X_{\pi_i} \mid \vx_{\pi_1:\pi_{i-1}}^\prime)$. 
    These nodes correspond to singleton sets $\mathcal{S}_{\pi,i}(\vx_{\pi_1:\pi_{i-1}}) = \{\vx_{\pi_1:\pi_{i-1}}\}$. 
    Collecting these singleton sets with the sets of nodes above, we obtain a set $\mathbf{s}$. 

    From the tree in Figure~\ref{fig:cstree}, we see that the sets $\mathbf{s}_i$, for this model, partition $\mathcal{X}_{\pi_1:\pi_{i-1}}$ for all $i$, since none of the sets of colored nodes overlap. 
    Hence, the model specified by the CSI relations $\mathcal{D}$ from Example~\ref{ex:CSI causal model} is in fact the CStree model $\mathcal{T} = (\pi,\mathbf{s})$.
    Its more compact LDAG representation is depicted in Figure~\ref{fig:LDAG}, where the notation $\ast$ is used to indicate that any outcome of the corresponding variable may substituted to give an outcome contained in the edge label.
\end{example}

\begin{figure}[t]
    \begin{subfigure}[b]{0.45\textwidth}
    \centering
    \begin{tikzpicture}[thick,scale=0.15]
		\node[draw, fill=black!0, inner sep=2pt, rounded corners, minimum width=2pt] (w3) at (6,15)  {\scriptsize 1111};
		\node[draw, fill=black!0, inner sep=2pt, rounded corners, minimum width=2pt] (w4) at (6,13.5) {\scriptsize 1110};
		\node[draw, fill=black!0, inner sep=2pt, rounded corners, minimum width=2pt] (w5) at (6,12) {\scriptsize 1101};
		\node[draw, fill=black!0, inner sep=2pt, rounded corners, minimum width=2pt] (w6) at (6,10.5) {\scriptsize 1100};
		\node[draw, fill=black!0, inner sep=2pt, rounded corners, minimum width=2pt] (v3) at (6,9)  {\scriptsize 1011};
		\node[draw, fill=black!0, inner sep=2pt, rounded corners, minimum width=2pt] (v4) at (6,7.5) {\scriptsize 1010};
		\node[draw, fill=black!0, inner sep=2pt, rounded corners, minimum width=2pt] (v5) at (6,6) {\scriptsize 1001};
		\node[draw, fill=black!0, inner sep=2pt, rounded corners, minimum width=2pt] (v6) at (6,4.5) {\scriptsize 1000};
		\node[draw, fill=black!0, inner sep=2pt, rounded corners, minimum width=2pt] (w3i) at (6,3)  {\scriptsize 0111};
		\node[draw, fill=black!0, inner sep=2pt, rounded corners, minimum width=2pt] (w4i) at (6,1.5) {\scriptsize 0110};
		\node[draw, fill=black!0, inner sep=2pt, rounded corners, minimum width=2pt] (w5i) at (6,0) {\scriptsize 0101};
		\node[draw, fill=black!0, inner sep=2pt, rounded corners, minimum width=2pt] (w6i) at (6,-1.5) {\scriptsize 0100};
		\node[draw, fill=black!0, inner sep=2pt, rounded corners, minimum width=2pt] (v3i) at (6,-3)  {\scriptsize 0011};
		\node[draw, fill=black!0, inner sep=2pt, rounded corners, minimum width=2pt] (v4i) at (6,-4.5) {\scriptsize 0010};
		\node[draw, fill=black!0, inner sep=2pt, rounded corners, minimum width=2pt] (v5i) at (6,-6) {\scriptsize 0001};
		\node[draw, fill=black!0, inner sep=2pt, rounded corners, minimum width=2pt] (v6i) at (6,-7.5) {\scriptsize 0000};

		\node[draw, fill=blue!0, inner sep=2pt, rounded corners, minimum width=2pt] (w1) at (-2,14.25) {\scriptsize 111};
		\node[draw, fill=cyan!60, inner sep=2pt, rounded corners, minimum width=2pt] (w2) at (-2,11.25) {\scriptsize 110};
		\node[draw, fill=orange!60, inner sep=2pt, rounded corners, minimum width=2pt] (v1) at (-2,8.25) {\scriptsize 101};
		\node[draw, fill=cyan!60, inner sep=2pt, rounded corners, minimum width=2pt] (v2) at (-2,5.25) {\scriptsize 100};
		\node[draw, fill=red!0, inner sep=2pt, rounded corners, minimum width=2pt] (w1i) at (-2,2.25) {\scriptsize 011};
		\node[draw, fill=cyan!60, inner sep=2pt, rounded corners, minimum width=2pt] (w2i) at (-2,-0.75) {\scriptsize 010};
		\node[draw, fill=orange!60, inner sep=2pt, rounded corners, minimum width=2pt] (v1i) at (-2,-3.75) {\scriptsize 001};
		\node[draw, fill=cyan!60, inner sep=2pt, rounded corners, minimum width=2pt] (v2i) at (-2,-6.75) {\scriptsize 000};

		\node[draw, fill=green!60, inner sep=2pt, rounded corners, minimum width=2pt] (w) at (-8,12.75) {\scriptsize 11};
		\node[draw, fill=cyan!0, inner sep=2pt, rounded corners, minimum width=2pt] (v) at (-8,6.75) {\scriptsize 10};
		\node[draw, fill=green!60, inner sep=2pt, rounded corners, minimum width=2pt] (wi) at (-8,0.75) {\scriptsize 01};
		\node[draw, fill=cyan!0, inner sep=2pt, rounded corners, minimum width=2pt] (vi) at (-8,-5.25) {\scriptsize 00};

		\node[draw, fill=yellow!0, inner sep=2pt, rounded corners, minimum width=2pt] (r) at (-14,9.75) {\scriptsize 1};
		\node[draw, fill=yellow!0, inner sep=2pt, rounded corners, minimum width=2pt] (ri) at (-14,-1.75) {\scriptsize 0};

		\node[draw, fill=black!0, inner sep=2pt, rounded corners, minimum width=2pt] (I) at (-20,3) {\scriptsize r};

		\draw[->]   (I) --    (r) ;
		\draw[->]   (I) --   (ri) ;

		\draw[->]   (r) --   (w) ;
		\draw[->]   (r) --   (v) ;

		\draw[->]   (w) --  (w1) ;
		\draw[->]   (w) --  (w2) ;

		\draw[->]   (w1) --   (w3) ;
		\draw[->]   (w1) --   (w4) ;
		\draw[->]   (w2) --  (w5) ;
		\draw[->]   (w2) --  (w6) ;

		\draw[->]   (v) --  (v1) ;
		\draw[->]   (v) --  (v2) ;

		\draw[->]   (v1) --  (v3) ;
		\draw[->]   (v1) --  (v4) ;
		\draw[->]   (v2) --  (v5) ;
		\draw[->]   (v2) --  (v6) ;

		\draw[->]   (ri) --   (wi) ;
		\draw[->]   (ri) -- (vi) ;

		\draw[->]   (wi) --  (w1i) ;
		\draw[->]   (wi) --  (w2i) ;

		\draw[->]   (w1i) --  (w3i) ;
		\draw[->]   (w1i) -- (w4i) ;
		\draw[->]   (w2i) --  (w5i) ;
		\draw[->]   (w2i) --  (w6i) ;

		\draw[->]   (vi) --  (v1i) ;
		\draw[->]   (vi) --  (v2i) ;

		\draw[->]   (v1i) --  (v3i) ;
		\draw[->]   (v1i) -- (v4i) ;
		\draw[->]   (v2i) -- (v5i) ;
		\draw[->]   (v2i) --  (v6i) ;

		\node at (-17.5,-9) {$X_1$} ;
		\node at (-11.5,-9) {$X_2$} ;
		\node at (-5,-9) {$X_3$} ;
		\node at (2,-9) {$X_4$} ;

	\end{tikzpicture}
	\caption{A CStree $\mathcal{T}$ for variable ordering $\pi = 1234$.}\label{fig:cstree}
    \end{subfigure}
    \hfill
    \begin{subfigure}[b]{0.45\textwidth}
    \centering
    \begin{tikzpicture}[thick,scale=0.4]
		\node[circle, draw, fill=black!0, inner sep=1pt, minimum width=1pt] (H1) at (3.25,8) {\large$1$};
		\node[circle, draw, fill=black!0, inner sep=1pt, minimum width=1pt] (B1) at (-2.25,4) {\large$2$};
		\node[circle, draw, fill=black!0, inner sep=1pt, minimum width=1pt] (G1) at (8.25,4) {\large$3$};
		\node[circle, draw, fill=black!0, inner sep=1pt, minimum width=1pt] (B2) at (3.25,0) {\large$4$};

		\draw[->]   (H1) -- (B1) ;
		\draw[->]   (H1) -- node[midway,sloped,above]{${\{1\}}$}(G1) ;
		\draw[->]   (H1) -- node[align=center,below, rotate=-90]{{$\{(0,1),\, (\ast,0)\}$}} (B2) ;
		\draw[->]   (B1) -- node[align=center,below, rotate=-40]{{$\{(\ast,0)\}$}} (B2) ;
		\draw[->]   (G1) -- (B2) ;
		\draw[->]   (B1) -- (G1) ;
	\end{tikzpicture}
	\caption{The LDAG of the CStree in Figure~\ref{fig:cstree}.}\label{fig:LDAG}
    \end{subfigure}
    
\caption{A CStree representation of the context-specific conditional independence model $\mathcal{D}$ on four binary variables from Example~\ref{ex:CSI causal model}.}
\label{fig:CStreeExample}
\end{figure}

The colored tree representation of a CStree is its \emph{staged tree representation}. 
A general staged tree model corresponds to an arbitrary coloring of the rooted tree.
While this reveals why it may be difficult to extract causal information easily from a staged tree representation, the staged tree provides a complete representation of how a distribution factorizes according to a given causal ordering. 
Hence, we will, at times, make use of these representations to derive the desired factorization-based generalizations of interventional DAG models to CStrees. 

Given a CStree $\mathcal{T} = (\pi, \mathbf{s})$, the set $\mathcal{S}_{\pi,i}(\vx_S)$ is a \emph{stage}, $\vx_S$ is its \emph{stage-defining context}, and $S$ will be its set of \emph{context variables}. 
The set of outcomes $\mathcal{X}_{\pi_1:\pi_{i-1}}$ is referred to as \emph{level $i$} of $\mathcal{T} = (\pi, \mathbf{s})$ and $\mathbf{s}_i$ is a \emph{staging} of level $i$. 
In CStrees, each stage $s = \mathcal{S}_{\pi,i}(\vx_S)$ corresponds to a conditional distribution $X_{\pi_i} \mid \vX_S = \vx_S$ used in the factorization~\eqref{eqn:CStreefactorization}. 
This conditional distribution may be parameterized via $\theta_{\pi_i, s} = [\theta_{\pi_i, s, 1},\ldots,\theta_{\pi_i, s, d_{\pi_i}}]$ satisfying $\sum_{t=1}^{d_{\pi_i}}\theta_{\pi_i, s, t} = 1$ and $\theta_{\pi_i, s, t} > 0$ for all $t\in [d_{\pi_i}]$. 
We let $\theta_{\pi_i, \mathbf{s}_i} = [\theta_{\pi_i, s} : s\in\mathbf{s}_i]$, and we let $\theta_{\pi, \mathbf{s}} = [\theta_{\pi_i,\mathbf{s}_i} : i\in[p]]$. 
Then the triple $\mathcal{T} = (\pi, \mathbf{s}, \theta_{\pi,s})$ is a parameterized CStree. 
If one labels the edge $\vx_{\pi_1:\pi_{i-1}} \rightarrow \vx_{\pi_1:\pi_{i-1}}t$ with $\theta_{\pi_i, s, t}$, then multiplying the edge labels along all root-to-leaf paths in $\mathcal{T}$ yields the joint distribution of $\vX\in \mathcal{M}(\mathcal{T})$ with parameters $\theta_{\pi, \mathbf{s}}$. 
Specifically, since $\theta_{\pi_i,s,t} = P(X_{\pi_i} = t \mid \vX_S = \vx_S)$, this parameterization of the CStree model $\mathcal{M}(\mathcal{T})$ corresponds exactly to the factorization in~\eqref{eqn:CStreefactorization}.

\begin{remark}
    \label{rem: why not general LDAGs}
    As noted directly above, the stage-defining contexts $\vx_S$ correspond to the conditional factors $X_{\pi_i} \mid \vX_S = \vx_S$ used in the factorization~\eqref{eqn:CStreefactorization} of a CStree model. 
    The fact that the relevant conditional distributions in this factorization are specified by pairs $(\pi_i, \vx_S)$ yields a factorization of the distribution that is amenable to soft, context-specific, interventions in analogy to~\eqref{eqn:I-DAGfactorization}. 
    This factorization property will play a fundamental role in our generalization of the interventional DAG model $\MM(\GG, \ci)$. 
    Theorem~\ref{thm: containment} implies that not all LDAGs factor into a product of conditional factors indexed by node-context pairs; e.g., not all LDAGs are CStrees.
    For instance, the relations of the LDAG in Figure~\ref{fig:LDAGofCStreeNOTLDAG} imply the equality of conditional probabilities
    \[
    P(X_3 \mid \vX_{1,2} = (0,0)) = P(X_3 \mid \vX_{1,2} = (1,0)) = P(X_3 \mid \vX_{1,2} = (1,1)), 
    \]
    which cannot be represented by a single conditional distribution $X_{\pi_i} \mid \vX_S = \vx_S$. This equality is encoded by the three green nodes in the staged tree representation of the model in Figure~\ref{fig:CStreeofLDAGnotCStree}.
\end{remark}

\subsection{Markov properties of CStrees}
\label{sec:markov properties}
A CStree model is a \emph{context-specific conditional independence model} according to~\eqref{eqn:CSI model} and Definition~\ref{def:CStree}. 
Thus, one may ask for a \emph{global Markov property} for $\mathcal{M}(\mathcal{T})$; i.e., the complete set of CSI relations satisfied by all distributions in $\mathcal{M}(\mathcal{T})$. 
In analogy to DAGs, one would naturally like to have a graphical representation of $\mathcal{T}$ that allows one to easily read-off the more general CSI relations in the Markov property for $\mathcal{T}$. 
To provide such a global Markov property and corresponding graphical representation, we first define context-specific conditional independence models. 

\subsubsection{Conditional independence models}
\label{subsec: conditional independence axioms}
A \emph{conditional independence model} $\J$ over a set of variables $V$ is a collection of triples $\langle A, B, \mid C \rangle$ where $A, B, C \subseteq V$. 
By assumption $\J$ always contains the triples $\langle A, \emptyset \mid C\rangle$ and $\langle \emptyset, B \mid C\rangle$. 
A DAG $\GG = ([p], E)$ encodes the conditional independence model
\[
\J(\GG) = \{ \langle A, B \mid C\rangle : \mbox{ $A$ and $B$ are d-separated given $C$ in $\GG$}\}. 
\]
The model $\J(\GG)$ is called a \emph{graphoid} since it is closed under the conditional independence axioms, including the intersection axiom \citep{SL14}. 
When the triples $\langle A, B \mid C\rangle$ are interpreted as CI relations $\vX_A \independent \vX_B \mid \vX_C$ then the DAG model $\MM(\GG)$ is the set of all distributions satisfying all relations in $\J(\GG)$. 

If $\pi$ is a topological ordering of the DAG $\GG$, one may define a second conditional independence model for the pair $(\GG, \pi)$ reflecting the factorization definition~\eqref{eqn:DAGfactorization}
\[
\J(\GG, \pi) = \{ \langle \pi_i, [\pi_1:\pi_{i-1}]\setminus \pa_\GG(i) \mid \pa_\GG(i)\rangle : i\in[p]\}.
\]
The distributions satisfying all relations in $\J(\GG, \pi)$ in fact satisfy additional relations implied by the conditional independence axioms. 
Given a conditional independence model $\J$, the \emph{closure} $\overline{\J}$ of $\J$ is the conditional independence model produced by iteratively applying the conditional independence axioms to $\J$ and adding the resulting relations. 
A classic result (see \citep[Theorem~3.27]{L96}) states that $\overline{\J(\GG, \pi)} = \J(\GG)$.  
Hence, $\J(\GG)$ is called the \emph{global Markov property of $\GG$}, and $\vX$ satisfies the global Markov property of $\GG$ if it satisfies all relations in $\J(\GG)$. 

Since Definition~\ref{def:CStree} directly generalizes the factorization definition~\eqref{eqn:DAGfactorization} of the DAG model $\MM(\GG)$, in the following, we construct an analogous global Markov property for a CStree $\mathcal{T}$ by taking the context-specific closure of the set of relations $\MM(\mathcal{T})$. 

\subsubsection{Context-specific conditional independence models}
\label{subsec: context-specific conditional independence}

A \emph{context-specific conditional independence model} $\J$ over a set of categorical variables $V$ is a collection of quadruples $\langle A, B \mid C, \vx_S\rangle$ where $A, B, C, S\subseteq V$ are disjoint and $\vx_S$ is an outcome of the variables in $S$. 
By definition, $\J$ always contains the quadruples $\langle A, \emptyset \mid C, \vx_S\rangle$ and $\langle \emptyset, B \mid C, \vx_S\rangle$. 
The model $\J$ is called a \emph{graphoid} if it is closed under the following \emph{context-specific conditional independence axioms}:

\begin{enumerate}
	\item \emph{symmetry.} If $\langle A,B \mid C, \vx_S\rangle\in\J$ then $\langle B,A \mid C, \vx_S\rangle\in\J$.
	\item \emph{decomposition.} If $\langle A,B\cup D \mid C, \vx_S\rangle\in\J$ then $\langle A,B \mid C, \vx_S\rangle\in\J$.
	\item \emph{weak union.} If $\langle A,B\cup D \mid C, \vx_S\rangle\in\J$ then $\langle A,B \mid C\cup D, \vx_S\rangle\in\J$. 
	\item \emph{contraction.} If $\langle A,B \mid C\cup D, \vx_S\rangle\in\J$ and $\langle A,D \mid C, \vx_S\rangle\in\J$ then $\langle A,B\cup D \mid C, \vx_S\rangle\in\J$. 
	\item \emph{intersection.} If $\langle A,B \mid C\cup D, \vx_S\rangle\in\J$ and $\langle A,C \mid B\cup D, \vx_S\rangle\in\J$ then $\langle A, B\cup C \mid D, \vx_S\rangle\in\J$.  
	\item \emph{specialization.} If $\langle A,B \mid C, \vx_S\rangle\in\J$, $T\subseteq C$ and $\xx_T\in\mathcal{X}_T$, then $\langle A,B \mid C\setminus T, \xx_{T\cup S}\rangle\in\J$.
	\item \emph{absorption.} If $\langle A,B \mid C, \vx_S\rangle\in\J$, $T\subseteq S$ for which $\langle A,B \mid C, \xx_{S\setminus T}\xx_T\rangle\in\J$ for all $\xx_T \in \mathcal{X}_T$, then $\langle A,B \mid C\cup T, \xx_{S\setminus T}\rangle\in\J$. 
\end{enumerate}

The conditional independence axioms of Subsection~\ref{subsec: conditional independence axioms} correspond to (1) -- (5) with $S = \emptyset$, and the closure $\overline{\J}$ of a context-specific conditional independence model $\J$ is defined in the analogous way. The quadruples $\langle A, B \mid C, \vx_S\rangle$ are typically interpreted as CSI relations.

Given a CStree $\mathcal{T} = (\pi,\mathbf{s})$, it follows from Definition~\ref{def:CStree} that $\mathcal{M}(\mathcal{T})$ is the set of all distributions satisfying the CSI relations in the context-specific conditional independence model 
\[
\J(\mathcal{T},\pi) = \{\langle \pi_i, [\pi_1:\pi_{i-1}]\setminus S \mid \emptyset, \vX_S = \vx_S\rangle : X_{\pi_i} \independent \vX_{[\pi_1:\pi_{i-1}]\setminus S}\mid \vX_S = \vx_S \in \mathcal{C}_{\mathcal{T}}\}.
\]
Let $\J(\mathcal{T}) = \overline{\J(\mathcal{T},\pi)}$ denote the closure of $\J(\mathcal{T},\pi)$.
To obtain a global Markov property for $\mathcal{T}$ in analogy to $\J(\GG)$ for a DAG $\GG$, we wish to compute the closure $\J(\mathcal{T})$ of the model $\J(\mathcal{T},\pi)$ and provide some graphical interpretation of relations therein. 
To do so, we identify a special set of contexts associated to a CStree model. 
Specifically, by the absorption axiom, there exists a (finite) collection of contexts $\vx_M$ such that for any CSI relation
\[
\langle\vX_A, \vX_B \mid \vX_C, \vx_M\rangle\in \J(\mathcal{T}),
\]
there is no subset $T\subseteq M$ for which 
\[
\langle \vX_A, \vX_B \mid \vX_{C\cup T}, \vx_{M \setminus T}\rangle\in\J(\mathcal{T}).
\]
We call each such $\xx_M$ a \emph{minimal context} for $\mathcal{T}$, and we let 
\[
\mathcal{C}_\mathcal{T} = \{\vx_M: \vx_M \mbox{ is minimal for $\mathcal{T}$}\} \cup \{\vx_\emptyset\}
\]
denote the collection of minimal contexts for the CStree $\mathcal{T}$ with the \emph{empty context} $\vx_\emptyset$ added in. 
Since $M = \emptyset$ in the empty context, it corresponds to no variable having a fixed outcome. 
Note that in some cases, repeated use of the absorption axiom can lead to a CI relation $\vX_A \independent \vX_B \mid \vX_D$ being in $\J(\mathcal{T})$. 
These relations have minimal context $\vx_\emptyset$. 
It follows from \citep[Proposition~2.2]{DS21} that $\MM(\mathcal{T})$ is a DAG model if and only if $\mathcal{C}_\mathcal{T} = \{\vx_\emptyset\}$. 

To extract a graphical representation of relations in $\J(\mathcal{T})$ we will use the following lemma.
\begin{lemma}
\label{lem: subcontexts}
Suppose that $\langle\vX_A, \vX_B \mid \vX_C, \vx_M\rangle\in\J(\mathcal{T})$. 
Then either
\begin{enumerate}
\item $\vx_M\in\mathcal{C}_\mathcal{T}$, or
\item $\langle\vX_A, \vX_B \mid \vX_C, \vx_M\rangle$ is implied by the specialization of some $\langle\vX_A, \vX_B \mid \vX_{C^\prime}, \vx_{M^\prime}\rangle \in \J(\mathcal{T})$, where $\vx_{M^\prime}\in\mathcal{C}_\mathcal{T}$. 
\end{enumerate}
In particular, every stage-defining context contains $\vx_S$ a minimal context.
\end{lemma}

It follows from Lemma~\ref{lem: subcontexts}, and the fact that axioms $(1) - (5)$ commute with
absorption, that $\J(\mathcal{T})$ is equal to the closure under specialization of the
union of context-specific graphiods 
\[
\bigcup_{\xx_M\in\mathcal{C}_\mathcal{T}}\J_{\xx_M},
\]
where $\J_{\xx_M}$ consists of all relations in $\J(\mathcal{T})$ with context $\xx_M$. 
To give a graphical representation of $\mathcal{T}$ whose combinatorics easily encodes relations in the global Markov property $\J(\mathcal{T})$, we will give a DAG representation $\GG_{\vx_M}$ of each model $\J_{\vx_M}$.  
These DAGs are minimal I-MAPs of the relations in $\J_{\vx_M}$ with respect to the causal order $\pi$ defining $\mathcal{T} = (\pi, \mathbf{s})$. 

\begin{definition}\citep{VP90}
    \label{def: minimal I-MAP}
    Let $\J$ be a conditional independence model on variables $[p]$, and let $\pi$ be a total ordering of $[p]$. 
    The DAG $\GG$ on node set $[p]$ with edge set
    \[
    E = \{ \pi_i \rightarrow \pi_j : i < j \mbox{ and } \langle \pi_j, \pi_i | [\pi_1:\pi_{j-1}]\setminus \pi_i\rangle \notin \J\}
    \]
    is called the \emph{minimal I-MAP} of $\J$ with respect to $\pi$. 
\end{definition}

We can then define the following alternative graphical representation of $\mathcal{T}$.

\begin{definition}
    \label{def: minimal context graphs}
    Let $\mathcal{T} = (\pi, \mathbf{s})$ be a CStree with set of minimal contexts $\mathcal{C}_\mathcal{T}$. 
    For $\vx_M\in\mathcal{C}_\mathcal{T}$, we let $\GG_{\vx_M}$ denote the minimal I-MAP of 
    \[
    \{\langle A, B \mid C\rangle : \langle A, B \mid C, \vx_M \rangle \in\J_{\vx_M}\}
    \]
    with respect to $\pi$. 
    The DAG $\GG_{\vx_M}$ is the \emph{minimal context graph} for $\vx_M$, and $\GG_{\mathcal{T}} = \{\GG_{\vx_M}\}_{\vx_M\in\mathcal{C}_\mathcal{T}}$ is called the \emph{minimal context graph representation} of $\mathcal{T}$. 
\end{definition}

It can be checked that the minimal contexts for the CStree $\mathcal{T}$ in Figure~\ref{fig:cstree} from Example~\ref{ex: cstree} are $\mathcal{C}_\mathcal{T} = \{\vx_\emptyset, x_2 = 0, x_2 = 1, x_3 = 0\}$. 
The resulting minimal context graph representation of $\mathcal{T}$ is depicted in Figure~\ref{fig:minimal context graphs}.
\begin{figure}[t]
    \centering
 \begin{tabular}{| c | c | c | c |}\hline
    \begin{tikzpicture}[thick,scale=0.25]
		\node[circle, draw, fill=black!0, inner sep=1pt, minimum width=1pt] (1) at (3.25,8) {\large$1$};
		\node[circle, draw, fill=black!0, inner sep=1pt, minimum width=1pt] (2) at (-2.25,4) {\large$2$};
		\node[circle, draw, fill=black!0, inner sep=1pt, minimum width=1pt] (3) at (8.25,4) {\large$3$};
		\node[circle, draw, fill=black!0, inner sep=1pt, minimum width=1pt] (4) at (3.25,0) {\large$4$};

		\draw[->]   (1) -- (2) ;
		\draw[->]   (1) -- (3) ;
		\draw[->]   (1) -- (4) ;
		\draw[->]   (2) -- (3) ;
		\draw[->]   (2) -- (4) ;
		\draw[->]   (3) -- (4) ;

            \node at (-2.25, 8) {$\GG_{\vx_\emptyset}$} ;
	\end{tikzpicture}
 &
 \begin{tikzpicture}[thick,scale=0.25]
		\node[circle, draw, fill=black!0, inner sep=1pt, minimum width=1pt] (1) at (3.25,8) {\large$1$};
		\node[circle, draw, fill=black!0, inner sep=1pt, minimum width=1pt] (3) at (8.25,4) {\large$3$};
		\node[circle, draw, fill=black!0, inner sep=1pt, minimum width=1pt] (4) at (3.25,0) {\large$4$};

		\draw[->]   (1) -- (3) ;
		\draw[->]   (3) -- (4) ;

            \node at (-2.25, 8) {$\GG_{x_2 = 0}$} ;
    \end{tikzpicture}
 &
 \begin{tikzpicture}[thick,scale=0.25]
		\node[circle, draw, fill=black!0, inner sep=1pt, minimum width=1pt] (1) at (3.25,8) {\large$1$};
		\node[circle, draw, fill=black!0, inner sep=1pt, minimum width=1pt] (3) at (8.25,4) {\large$3$};
		\node[circle, draw, fill=black!0, inner sep=1pt, minimum width=1pt] (4) at (3.25,0) {\large$4$};

		\draw[->]   (1) -- (4) ;
		\draw[->]   (3) -- (4) ;

            \node at (-2.25, 8) {$\GG_{x_2 = 1}$} ;
    \end{tikzpicture}
 & 
 \begin{tikzpicture}[thick,scale=0.25]
		\node[circle, draw, fill=black!0, inner sep=1pt, minimum width=1pt] (1) at (3.25,8) {\large$1$};
		\node[circle, draw, fill=black!0, inner sep=1pt, minimum width=1pt] (2) at (-2.25,4) {\large$2$};
            \node at (8.25,4) {\,} ;
		\node[circle, draw, fill=black!0, inner sep=1pt, minimum width=1pt] (4) at (3.25,0) {\large$4$};

		\draw[->]   (1) -- (2) ;

            \node at (-2.25, 8) {$\GG_{x_3 = 0}$} ;
    \end{tikzpicture}
 \\\hline
 \end{tabular}
\caption{The minimal context graphs of the CStree in Figure~\ref{fig:cstree}.}
\label{fig:minimal context graphs}
\end{figure}

\begin{definition}
    \label{def: global MP}
    We say that a distribution $\vX$ \emph{satisfies the global Markov property} with respect to a CStree $\mathcal{T}$ if, for all $\vx_M\in\mathcal{C}_\mathcal{T}$, $\vX$ entails the CSI relation $\vX_A \independent \vX_B \mid \vX_C, \vX_M = \vx_M$ whenever $A$ and $B$ are d-separated given $C$ in $\GG_{\vx_M}$. 
Let $\MM(\GG_\mathcal{T})$ denote the set of all distributions that satisfy the global Markov property with respect to $\mathcal{T}$.
\end{definition}

We have the following theorem.

\begin{theorem}
\label{thm: markov to context graphs}
Let $\mathcal{T} = (\pi, \mathbf{s})$ be a CStree, and let $\vX = (X_1,\ldots, X_p)$ be a categorical distribution.
The following are equivalent:
\begin{enumerate}
	\item $\vX$ is Markov $\mathcal{T}$, 
	\item $\vX$ satisfies the global Markov property with respect to $\mathcal{T}$, and 
	\item for all $\vx_M\in\mathcal{C}_\mathcal{T}$,
	\[
	P(\vX_{[p]\setminus M} \mid \vX_M = \vx_M) = \prod_{k\in[p]\setminus M}P(X_k \mid \vX_{\pa_{\GG_{\vx_M}}(k)}, \vX_M = \vx_M).
	\]
\end{enumerate}
In particular, $\MM(\GG_\mathcal{T}) = \MM(\mathcal{T})$.
\end{theorem}

\begin{remark}
    \label{rem: incomplete}
    While the equivalence of~(1) and~(2) in Theorem~\ref{thm: markov to context graphs} provides an analogy between equivalence of factorizing according to a DAG $\GG$ and satisfying the global Markov property with respect to $\GG$, it is important to note that the CStree global Markov property differs in one important way: 
    The global Markov property for DAGs is \emph{complete}; i.e., any CI relation in the closure of the model-defining relations $\overline{\J(\GG,\pi)}$ is witnessed as a d-separation in the DAG $\GG$ \citep{VP90}. 
    In particular, there is a combinatorial rule for the graphical representation $\GG$ of the model $\MM(\GG)$ that captures all CI relations in $\J(\GG)$.  
    While the global Markov property for CStrees given above graphically represents many more CSI relations than the model-defining relations $\J(\mathcal{T}, \pi)$, there exist CStree models for which there are relations in the closure $\J(\mathcal{T}) = \overline{\J(\mathcal{T},\pi)}$ that are not explicitly represented by a d-separation in any of the minimal context graphs. 
    More generally, there is no known family of context-specific conditional independence models with a known complete global Markov property. 
\end{remark}

\begin{remark}
    \label{rem: LDAG markov property}
    An alternative global Markov property for CStrees, analogous to that in Theorem~\ref{thm: markov to context graphs}~(2) may be obtained using their LDAG representation and \citep[Theorem~4]{PNKC15}. 
    Specifically, \citep[Theorem~4]{PNKC15} characterizes LDAG model membership as satisfying all CSI relations encoded by d-separations in a collection of context graphs, one for each outcome $\vx\in \mathcal{X}$. 
    In general, the set of minimal contexts $\mathcal{C}_\mathcal{T}$ needed in Theorem~\ref{thm: markov to context graphs}~(2) will be less than the number of joint outcomes $\mathcal{X}$ of the entire distribution.
\end{remark}

\subsection{Model equivalence}
\label{subsec: model equivalence}
We say that two CStrees $\mathcal{T}$ and $\mathcal{T}^\prime$ are \emph{Markov equivalent} if $\MM(\mathcal{T}) = \MM(\mathcal{T}^\prime)$. 
Using the global Markov property of CStrees in Definition~\ref{def: global MP} and the equivalence of~(1) and~(2) in Theorem~\ref{thm: markov to context graphs}, we obtain a characterization of Markov equivalence of CStrees. 
To do so, we first observe the following lemma.

\begin{lemma}
\label{lem: equal minimal contexts}
If $\mathcal{T}$ and $\mathcal{T}^\prime$ are Markov equivalent CStrees then their sets of minimal contexts are
equal; that is, $\mathcal{C}_\mathcal{T}  = \mathcal{C}_{\mathcal{T}^\prime}$.
\end{lemma}

We can then prove the following theorem. 

\begin{theorem}
\label{thm: first characterization}
Two CStrees, $\mathcal{T}$ and $\mathcal{T}^\prime$, are Markov equivalent if and only if they have the same set of minimal contexts and their minimal contexts graphs are pairwise Markov equivalent; that is, $\mathcal{C}_\mathcal{T} = \mathcal{C}_{\mathcal{T}^\prime}$ and $\GG_{\vx_M}\in\GG_\mathcal{T}$ and $\GG^\prime_{\vx_M}\in\GG_{\mathcal{T}^\prime}$ are Markov equivalent for all $\vx_M\in \mathcal{C}_\mathcal{T}$.
\end{theorem}

\citet{VP92} showed that two DAGs $\GG$ and $\HH$ are Markov equivalent if and only if they have the same skeleton and set of v-structures. 
The \emph{skeleton} of a DAG $\GG$ is the undirected graph of adjacencies in $\GG$, and a \emph{v-structure} in $\GG$ is a path $i\rightarrow j \leftarrow k$ where $i$ and $k$ are not adjacent. 
Hence, we have the following corollary to Theorem~\ref{thm: first characterization}, generalizing the result of \citet{VP92} to CStrees. 

\begin{corollary}
\label{cor: VP generalization}
Two CStrees $\mathcal{T}$ and $\mathcal{T}^\prime$ are Markov equivalent if and only if $\mathcal{C}_\mathcal{T} = \mathcal{C}_{\mathcal{T}^\prime}$ and for all $\vx_M\in\mathcal{C}_\mathcal{T}$, the graphs $\GG_{\vx_M}\in\GG_\mathcal{T}$ and $\GG_{\vx_M}^\prime\in\GG_{\mathcal{T}^\prime}$ have the same skeleton and v-structures.  
\end{corollary}

\begin{example}
    \label{ex: equivalent CStrees}
    The two CStrees $\mathcal{T}_A = (\pi_A, \mathbf{s}_A)$ and $\mathcal{T}_B = (\pi_B, \mathbf{s}_B)$ depicted in Figure~\ref{fig: equivalent CStrees} are Markov equivalent to the CStree $\mathcal{T}$ in Figure~\ref{fig:cstree} that represents our chicken pox model $\mathcal{D}$ from Example~\ref{ex:CSI causal model}.
    The staged tree representations of these CStrees are presented in Figure~\ref{fig: equivalent CStrees}, showing how these representations can change between equivalent models.
    The tree $\mathcal{T}_A$ has causal order $\pi_A = 2134$, and $\mathcal{T}_B$ has causal order $\pi_B = 3214$. 
    The minimal context graph representation for $\mathcal{T}$ is depicted in Figure~\ref{fig:minimal context graphs}. 
    With the help of Corollary~\ref{cor: VP generalization}, the model equivalence $\MM(\mathcal{T}) = \MM(\mathcal{T}_A) = \MM(\mathcal{T}_B)$ is much more easily seen in the corresponding minimal context graphs for these two trees, which are obtained by swapping the directions of the edges between $1,2$ and $1,3$, respectively, in all minimal contexts graphs in $\GG_\mathcal{T}$.  
\end{example}

\begin{figure}[t]
    \begin{subfigure}[b]{0.45\textwidth}
    \centering
    \begin{tikzpicture}[thick,scale=0.15]
		\node[draw, fill=black!0, inner sep=2pt, rounded corners, minimum width=2pt] (w3) at (6,15)  {\scriptsize 1111};
		\node[draw, fill=black!0, inner sep=2pt, rounded corners, minimum width=2pt] (w4) at (6,13.5) {\scriptsize 1110};
		\node[draw, fill=black!0, inner sep=2pt, rounded corners, minimum width=2pt] (w5) at (6,12) {\scriptsize 1101};
		\node[draw, fill=black!0, inner sep=2pt, rounded corners, minimum width=2pt] (w6) at (6,10.5) {\scriptsize 1100};
		\node[draw, fill=black!0, inner sep=2pt, rounded corners, minimum width=2pt] (v3) at (6,9)  {\scriptsize 1011};
		\node[draw, fill=black!0, inner sep=2pt, rounded corners, minimum width=2pt] (v4) at (6,7.5) {\scriptsize 1010};
		\node[draw, fill=black!0, inner sep=2pt, rounded corners, minimum width=2pt] (v5) at (6,6) {\scriptsize 1001};
		\node[draw, fill=black!0, inner sep=2pt, rounded corners, minimum width=2pt] (v6) at (6,4.5) {\scriptsize 1000};
		\node[draw, fill=black!0, inner sep=2pt, rounded corners, minimum width=2pt] (w3i) at (6,3)  {\scriptsize 0111};
		\node[draw, fill=black!0, inner sep=2pt, rounded corners, minimum width=2pt] (w4i) at (6,1.5) {\scriptsize 0110};
		\node[draw, fill=black!0, inner sep=2pt, rounded corners, minimum width=2pt] (w5i) at (6,0) {\scriptsize 0101};
		\node[draw, fill=black!0, inner sep=2pt, rounded corners, minimum width=2pt] (w6i) at (6,-1.5) {\scriptsize 0100};
		\node[draw, fill=black!0, inner sep=2pt, rounded corners, minimum width=2pt] (v3i) at (6,-3)  {\scriptsize 0011};
		\node[draw, fill=black!0, inner sep=2pt, rounded corners, minimum width=2pt] (v4i) at (6,-4.5) {\scriptsize 0010};
		\node[draw, fill=black!0, inner sep=2pt, rounded corners, minimum width=2pt] (v5i) at (6,-6) {\scriptsize 0001};
		\node[draw, fill=black!0, inner sep=2pt, rounded corners, minimum width=2pt] (v6i) at (6,-7.5) {\scriptsize 0000};

		\node[draw, fill=blue!0, inner sep=2pt, rounded corners, minimum width=2pt] (w1) at (-2,14.25) {\scriptsize 111};
		\node[draw, fill=cyan!60, inner sep=2pt, rounded corners, minimum width=2pt] (w2) at (-2,11.25) {\scriptsize 110};
		\node[draw, fill=orange!0, inner sep=2pt, rounded corners, minimum width=2pt] (v1) at (-2,8.25) {\scriptsize 101};
		\node[draw, fill=cyan!60, inner sep=2pt, rounded corners, minimum width=2pt] (v2) at (-2,5.25) {\scriptsize 100};
		\node[draw, fill=orange!60, inner sep=2pt, rounded corners, minimum width=2pt] (w1i) at (-2,2.25) {\scriptsize 011};
		\node[draw, fill=cyan!60, inner sep=2pt, rounded corners, minimum width=2pt] (w2i) at (-2,-0.75) {\scriptsize 010};
		\node[draw, fill=orange!60, inner sep=2pt, rounded corners, minimum width=2pt] (v1i) at (-2,-3.75) {\scriptsize 001};
		\node[draw, fill=cyan!60, inner sep=2pt, rounded corners, minimum width=2pt] (v2i) at (-2,-6.75) {\scriptsize 000};

		\node[draw, fill=green!60, inner sep=2pt, rounded corners, minimum width=2pt] (w) at (-8,12.75) {\scriptsize 11};
		\node[draw, fill=green!60, inner sep=2pt, rounded corners, minimum width=2pt] (v) at (-8,6.75) {\scriptsize 10};
		\node[draw, fill=cyan!0, inner sep=2pt, rounded corners, minimum width=2pt] (wi) at (-8,0.75) {\scriptsize 01};
		\node[draw, fill=cyan!0, inner sep=2pt, rounded corners, minimum width=2pt] (vi) at (-8,-5.25) {\scriptsize 00};

		\node[draw, fill=yellow!0, inner sep=2pt, rounded corners, minimum width=2pt] (r) at (-14,9.75) {\scriptsize 1};
		\node[draw, fill=yellow!0, inner sep=2pt, rounded corners, minimum width=2pt] (ri) at (-14,-1.75) {\scriptsize 0};

		\node[draw, fill=black!0, inner sep=2pt, rounded corners, minimum width=2pt] (I) at (-20,3) {\scriptsize r};

		\draw[->]   (I) --    (r) ;
		\draw[->]   (I) --   (ri) ;

		\draw[->]   (r) --   (w) ;
		\draw[->]   (r) --   (v) ;

		\draw[->]   (w) --  (w1) ;
		\draw[->]   (w) --  (w2) ;

		\draw[->]   (w1) --   (w3) ;
		\draw[->]   (w1) --   (w4) ;
		\draw[->]   (w2) --  (w5) ;
		\draw[->]   (w2) --  (w6) ;

		\draw[->]   (v) --  (v1) ;
		\draw[->]   (v) --  (v2) ;

		\draw[->]   (v1) --  (v3) ;
		\draw[->]   (v1) --  (v4) ;
		\draw[->]   (v2) --  (v5) ;
		\draw[->]   (v2) --  (v6) ;

		\draw[->]   (ri) --   (wi) ;
		\draw[->]   (ri) -- (vi) ;

		\draw[->]   (wi) --  (w1i) ;
		\draw[->]   (wi) --  (w2i) ;

		\draw[->]   (w1i) --  (w3i) ;
		\draw[->]   (w1i) -- (w4i) ;
		\draw[->]   (w2i) --  (w5i) ;
		\draw[->]   (w2i) --  (w6i) ;

		\draw[->]   (vi) --  (v1i) ;
		\draw[->]   (vi) --  (v2i) ;

		\draw[->]   (v1i) --  (v3i) ;
		\draw[->]   (v1i) -- (v4i) ;
		\draw[->]   (v2i) -- (v5i) ;
		\draw[->]   (v2i) --  (v6i) ;

		\node at (-17.5,-9) {$X_2$} ;
		\node at (-11.5,-9) {$X_1$} ;
		\node at (-5,-9) {$X_3$} ;
		\node at (2,-9) {$X_4$} ;

	\end{tikzpicture}
	\caption{$\mathcal{T}_A = (\pi_A,\mathbf{s}_A)$.}\label{fig:cstreeA}
    \end{subfigure}
    \hfill
    \begin{subfigure}[b]{0.45\textwidth}
    \centering
    \begin{tikzpicture}[thick,scale=0.15]
		\node[draw, fill=black!0, inner sep=2pt, rounded corners, minimum width=2pt] (w3) at (6,15)  {\scriptsize 1111};
		\node[draw, fill=black!0, inner sep=2pt, rounded corners, minimum width=2pt] (w4) at (6,13.5) {\scriptsize 1110};
		\node[draw, fill=black!0, inner sep=2pt, rounded corners, minimum width=2pt] (w5) at (6,12) {\scriptsize 1101};
		\node[draw, fill=black!0, inner sep=2pt, rounded corners, minimum width=2pt] (w6) at (6,10.5) {\scriptsize 1100};
		\node[draw, fill=black!0, inner sep=2pt, rounded corners, minimum width=2pt] (v3) at (6,9)  {\scriptsize 1011};
		\node[draw, fill=black!0, inner sep=2pt, rounded corners, minimum width=2pt] (v4) at (6,7.5) {\scriptsize 1010};
		\node[draw, fill=black!0, inner sep=2pt, rounded corners, minimum width=2pt] (v5) at (6,6) {\scriptsize 1001};
		\node[draw, fill=black!0, inner sep=2pt, rounded corners, minimum width=2pt] (v6) at (6,4.5) {\scriptsize 1000};
		\node[draw, fill=black!0, inner sep=2pt, rounded corners, minimum width=2pt] (w3i) at (6,3)  {\scriptsize 0111};
		\node[draw, fill=black!0, inner sep=2pt, rounded corners, minimum width=2pt] (w4i) at (6,1.5) {\scriptsize 0110};
		\node[draw, fill=black!0, inner sep=2pt, rounded corners, minimum width=2pt] (w5i) at (6,0) {\scriptsize 0101};
		\node[draw, fill=black!0, inner sep=2pt, rounded corners, minimum width=2pt] (w6i) at (6,-1.5) {\scriptsize 0100};
		\node[draw, fill=black!0, inner sep=2pt, rounded corners, minimum width=2pt] (v3i) at (6,-3)  {\scriptsize 0011};
		\node[draw, fill=black!0, inner sep=2pt, rounded corners, minimum width=2pt] (v4i) at (6,-4.5) {\scriptsize 0010};
		\node[draw, fill=black!0, inner sep=2pt, rounded corners, minimum width=2pt] (v5i) at (6,-6) {\scriptsize 0001};
		\node[draw, fill=black!0, inner sep=2pt, rounded corners, minimum width=2pt] (v6i) at (6,-7.5) {\scriptsize 0000};

		\node[draw, fill=blue!0, inner sep=2pt, rounded corners, minimum width=2pt] (w1) at (-2,14.25) {\scriptsize 111};
		\node[draw, fill=blue!00, inner sep=2pt, rounded corners, minimum width=2pt] (w2) at (-2,11.25) {\scriptsize 110};
		\node[draw, fill=orange!60, inner sep=2pt, rounded corners, minimum width=2pt] (v1) at (-2,8.25) {\scriptsize 101};
		\node[draw, fill=orange!60, inner sep=2pt, rounded corners, minimum width=2pt] (v2) at (-2,5.25) {\scriptsize 100};
		\node[draw, fill=cyan!60, inner sep=2pt, rounded corners, minimum width=2pt] (w1i) at (-2,2.25) {\scriptsize 011};
		\node[draw, fill=cyan!60, inner sep=2pt, rounded corners, minimum width=2pt] (w2i) at (-2,-0.75) {\scriptsize 010};
		\node[draw, fill=cyan!60, inner sep=2pt, rounded corners, minimum width=2pt] (v1i) at (-2,-3.75) {\scriptsize 001};
		\node[draw, fill=cyan!60, inner sep=2pt, rounded corners, minimum width=2pt] (v2i) at (-2,-6.75) {\scriptsize 000};

		\node[draw, fill=green!60, inner sep=2pt, rounded corners, minimum width=2pt] (w) at (-8,12.75) {\scriptsize 11};
		\node[draw, fill=cyan!0, inner sep=2pt, rounded corners, minimum width=2pt] (v) at (-8,6.75) {\scriptsize 10};
		\node[draw, fill=green!60, inner sep=2pt, rounded corners, minimum width=2pt] (wi) at (-8,0.75) {\scriptsize 01};
		\node[draw, fill=cyan!0, inner sep=2pt, rounded corners, minimum width=2pt] (vi) at (-8,-5.25) {\scriptsize 00};

		\node[draw, fill=yellow!0, inner sep=2pt, rounded corners, minimum width=2pt] (r) at (-14,9.75) {\scriptsize 1};
		\node[draw, fill=yellow!0, inner sep=2pt, rounded corners, minimum width=2pt] (ri) at (-14,-1.75) {\scriptsize 0};

		\node[draw, fill=black!0, inner sep=2pt, rounded corners, minimum width=2pt] (I) at (-20,3) {\scriptsize r};

		\draw[->]   (I) --    (r) ;
		\draw[->]   (I) --   (ri) ;

		\draw[->]   (r) --   (w) ;
		\draw[->]   (r) --   (v) ;

		\draw[->]   (w) --  (w1) ;
		\draw[->]   (w) --  (w2) ;

		\draw[->]   (w1) --   (w3) ;
		\draw[->]   (w1) --   (w4) ;
		\draw[->]   (w2) --  (w5) ;
		\draw[->]   (w2) --  (w6) ;

		\draw[->]   (v) --  (v1) ;
		\draw[->]   (v) --  (v2) ;

		\draw[->]   (v1) --  (v3) ;
		\draw[->]   (v1) --  (v4) ;
		\draw[->]   (v2) --  (v5) ;
		\draw[->]   (v2) --  (v6) ;

		\draw[->]   (ri) --   (wi) ;
		\draw[->]   (ri) -- (vi) ;

		\draw[->]   (wi) --  (w1i) ;
		\draw[->]   (wi) --  (w2i) ;

		\draw[->]   (w1i) --  (w3i) ;
		\draw[->]   (w1i) -- (w4i) ;
		\draw[->]   (w2i) --  (w5i) ;
		\draw[->]   (w2i) --  (w6i) ;

		\draw[->]   (vi) --  (v1i) ;
		\draw[->]   (vi) --  (v2i) ;

		\draw[->]   (v1i) --  (v3i) ;
		\draw[->]   (v1i) -- (v4i) ;
		\draw[->]   (v2i) -- (v5i) ;
		\draw[->]   (v2i) --  (v6i) ;

		\node at (-17.5,-9) {$X_3$} ;
		\node at (-11.5,-9) {$X_2$} ;
		\node at (-5,-9) {$X_1$} ;
		\node at (2,-9) {$X_4$} ;

	\end{tikzpicture}
	\caption{$\mathcal{T}_B = (\pi_B,\mathbf{s}_B)$.}\label{fig:cstreeB}
    \end{subfigure}
    
\caption{The staged tree representation of two CStrees that are Markov equivalent to $\mathcal{T}$ from Figure~\ref{fig:cstree}.  The staged tree representation explicitly shows the models satisfy the CStree factorization~\eqref{eqn:CStreefactorization}. The minimal context graphs, described in Example~\ref{ex: equivalent CStrees} show model equivalence much more easily via Corollary~\ref{cor: VP generalization}.}
\label{fig: equivalent CStrees}
\end{figure}

\section{Interventional CStrees}
\label{sec:interventionalcstrees}
We now consider a context-specific generalization of the interventional DAG model $\MM(\GG, \ci)$ for a DAG $\GG = ([p], E)$ and set of intervention targets $\ci = (I_0= \emptyset, I_1,\ldots, I_K)$, as defined in~\eqref{eqn:I-DAGmodel}. 
To graphically represent the model $\MM(\GG, \ci)$, \citet{YKU18} introduced the $\ci$-DAG $\GG^\ci = ([p]\cup W_\ci, E_\ci)$, where
\[
W_\ci = \{\omega_k : k\in[K]\} \qquad \mbox{and} \qquad E_\ci = \{\omega_k \rightarrow i : i \in I_k, \mbox{ for all } k\in[K]\}.
\]
They then generalized the characterization of Markov equivalence for DAG models $\MM(\GG)$ of \citet{VP90} as follows. 

\begin{theorem}\cite[Theorem~3.9]{YKU18}
    \label{thm: Yang characterization}
    $\MM(\GG,\ci) = \MM(\HH, \ci)$ if and only if $\GG^\ci$ and $\HH^\ci$ have the same skeleton and v-structures. 
\end{theorem}

Provided with our generalization of the characterization of \citet{VP90} to CStree models $\MM(\mathcal{T})$ in Corollary~\ref{cor: VP generalization}, our goal in this section is to extend our result so as to obtain a generalization of Theorem~\ref{thm: Yang characterization} to general, context-specific interventions in CStrees. 

To accomplish this goal, we will first recall the global $\ci$-Markov property \citet{YKU18} used to obtain the result in Theorem~\ref{thm: Yang characterization}. 
We then define $\ci$-CStrees and make precise the notion of a general, context-specific intervention.  
This allows us to extend our global Markov property of CStrees in Definition~\ref{def: global MP} to a global $\ci$-Markov property of CStrees. 
From there we may obtain the desired result. 
The global $\ci$-Markov property for DAGs is the following.

\begin{definition}
\label{def: I-Markov property}
Let $\ci = (I_0 = \emptyset, I_1,\ldots,I_K)$ be a sequence of intervention targets.  
Let $(\vX^{I})_{I\in\ci}$ be a set of (strictly positive) distributions.  
Then $(\vX^{I})_{I\in\ci}$ satisfies the \emph{$\ci$-Markov property} with respect to $\GG$ and $\ci$ if 
\begin{enumerate}
	\item $\vX_A \independent \vX_B \mid \vX_C$ for any $I\in \ci$ and any disjoint $A,B,C\subseteq[p]$ for which $C$ d-separates $A$ and $B$ in $\GG$.  
	\item $P^{I}(\vX_A \mid \vX_C) = P^{\emptyset}(\vX_A \mid \vX_C)$ for any $I\in \ci$ and any disjoint $A,C\subseteq [p]$ for which $C\cup W_{\ci}\setminus\omega_I$ d-separates $A$ and $w_I$ in $\GG^\ci$. 
\end{enumerate}
\end{definition}

\subsection{Interventional CStrees}
\label{subsec: interventional staged tree models and CStrees}
General interventions in DAG models, as defined in~\eqref{eqn:I-DAGfactorization}, amount to replacing the conditional factor $P(x_i \mid \vx_{\pa_\GG(i)})$ in~\eqref{eqn:DAGfactorization} with an new conditional distribution $P^{I}(x_i \mid \vx_{\pa_\GG(i)})$ for all $i$ in the intervention target $I$, for all outcomes $\vx_{\pa_\GG(i)}\in\mathcal{X}_{\pa_\GG(i)}$. 
In the context-specific generalization~\eqref{eqn:CStreefactorization} of the DAG factorization~\eqref{eqn:DAGfactorization}, the key difference is that the conditional factors $P(x_i | \vx_S)$ may have different sets $S$, specifying more diverse, context-specific relations; namely, these conditional factors need not satisfy $S = P_i$ for a subset $P_i$ indexed by $i$. 

As discussed in Remark~\ref{rem: why not general LDAGs}, it naturally follows that, for CStree models, an intervention may only perturb the conditional factors $P(x_i | \vx_S)$ for specific stages $\mathcal{S}_{\pi,i}(\vx_S)\in \mathbf{s}_i$; e.g., for specific choices of stage-defining contexts $\vx_S$. 
In this regard, we may define a \emph{context-specific intervention target} to be a subset $I$ of the stages $\mathbf{s}$ in the CStree $\mathcal{T} = (\pi, \mathbf{s})$.
An \emph{interventional distribution} $\vX^{I}$ for $I\subseteq\mathbf{s}$ and $\vX\in \mathcal{M}(\mathcal{G})$ is a distribution having probability mass function satisfying 
\begin{equation}
    \label{eqn:I-CStreefactorization}
    P^I(\vx) = \prod_{i\, :\,  \vx_{\pa_{\mathcal{T}}(\vx_{\pi_1:\pi_{i-1}})}\in I}P^I(x_i \mid \vx_{\pa_{\mathcal{T}}(\vx_{\pi_1:\pi_{i-1}})})\prod_{i \, : \, \vx_{\pa_{\mathcal{T}}(\vx_{\pi_1:\pi_{i-1}})}\notin I}P(x_i \mid \vx_{\pa_{\mathcal{T}}(\vx_{\pi_1:\pi_{i-1}})}) 
\end{equation}
for all outcomes $\vx = (x_1,\ldots, x_p)$.
For notational convenience, we typically denote the elements $\mathcal{S}_{\pi,i}(\vx_S)$ of the target $I$ simply by their stage-defining contexts $\vx_S$, as in~\eqref{eqn:I-CStreefactorization}. 
Given a CStree $\mathcal{T} = (\pi, \mathbf{s})$ and sequence of intervention targets $\ci = (I_0 = \emptyset, I_1,\ldots, I_K)$, we may then define the \emph{interventional CStree model}
\begin{equation}
    \label{eqn:I-CStreemodel}
    \begin{split}
    \MM(\mathcal{T},\ci) = \{(\vX^0,\ldots, \vX^K) :\, &\mbox{for all $k\in\{0,\ldots, K\}$, } \vX^k\in \MM(\mathcal{T}) \mbox{ and for all outcomes $\vx$}\\
    &\mbox{  and stages $\mathcal{S}_{\pi,i}(\vx_S)\in\mathbf{s}$, } \\
    &P^{I_k}(x_i \mid \vx_S) = P^{I_{k^\prime}}(x_i \mid \vx_S) \mbox{ whenever $\vx_S\notin I_k\cup I_{k^\prime}$}\}. \\
    \end{split}
\end{equation}

Similar to the CStree models $\MM(\mathcal{T})$ in Subsection~\ref{subsec:cstrees}, one can graphically represent the interventional CStree model $\MM(\mathcal{T},\ci)$ using a staged tree. 
In particular, for each $I_k\in \ci$, we take a copy $\mathcal{T}^k$ of the staged tree representation for the model $\MM(\mathcal{T})$ and denote its root node by $r^k$. 
We then connect these graphs to an additional root node $r^\ci$ with edges $r^\ci\rightarrow r^k$ for all $k \in\{0,\ldots, K\}$. 
For convenience, we often write the elements of the set $I_k$ along the edge $r^\ci\rightarrow r^k$. 
The resulting colored tree is denoted $\mathcal{T}^\ci$.
Note that, by the invariances defining the model $\MM(\mathcal{T},\ci)$, the stages $\mathcal{S}_{\pi,i}(\vx_S)$ in $\mathcal{T}^k$ and its corresponding copy in the tree $\mathcal{T}^{k^\prime}$ will be the same color whenever $\vx_S\notin I_k\cup I_{k^\prime}$. 
In particular, $\mathcal{T}^\ci$ is a staged tree in which any stage not targeted for intervention in $I_k$ or $I_{k^\prime}$ are unioned into a single stage. 
This union graphically captures the context-specific invariances defining $\MM(\mathcal{T},\ci)$. 
Moreover, it reveals that the model $\MM(\mathcal{T},\ci)$ is always an \emph{interventional staged tree model} as defined by \citet{DS20}.
As these trees are cumbersome to draw, for explicit examples, we refer the reader to \citep{DS20} and to Appendix~\ref{appsec: additional results} where interventional staged tree representations for the real data analysis in Section~\ref{sec:real examples} are presented.


For modeling purposes, it is preferable to have a more compact graphical representation of the model $\MM(\mathcal{T},\ci)$ than its staged tree representation. 
Since our characterization of Markov equivalence for $\MM(\mathcal{T})$ uses the minimal context DAG representation introduced in Definition~\ref{def: minimal context graphs}, this is the natural candidate for generalization to a graphical characterization of $\ci$-Markov equivalence (as Theorem~\ref{thm: Yang characterization} does for Verma and Pearl).
To do so, we first extract a global $\ci$-Markov property for CStrees, generalizing Definition~\ref{def: I-Markov property}. 
This will rely on the following lemma.

\begin{lemma}
\label{lem: characterizing interventional settings}
Let $\mathcal{T}$ be a CStree and $\ci$ a sequence of targets. 
Then $(\vX^{I})_{I\in\ci}\in\MM(\mathcal{T}, \ci)$ if and only if there exists $\vX^{0}\in\MM(\mathcal{T})$ such that $\vX^{I}$ factorizes as in~\eqref{eqn:I-CStreefactorization} with respect to $\vX^{0}$ for all $I\in\ci$.  
\end{lemma}

\subsection{$\ci$-Markov properties of CStrees}
We first construct a generalization of minimal context graphs for $\mathcal{T}$ that include interventions. 
This is done in analogy to the construction of the $\ci$-DAG $\GG^\ci$. 
By Lemma~\ref{lem: subcontexts}, every stage-defining context $\vx_S$ contains (at least one) minimal context $\vx_M\in\mathcal{C}_\mathcal{T}$ as a \emph{subcontext}; e.g., $M\subseteq S$ and $\vx_{S}$ restricted to the indices in $M$ is equal to $\vx_M$. 

Let $\GG_{\vx_M}$ be a minimal context graph of $\mathcal{T}$, and $\ci = (I_0,\ldots, I_K)$ a sequence of intervention targets. 
For the target $I_k\subseteq \mathbf{s}$, let 
\[
I_{k,\vx_M} = \{\vx_S : \vx_S\in I_k \mbox{ and } \vx_M \mbox{ is a subcontext of } \vx_S\}.
\]
The set $I_{k,\vx_M}$ isolates the elements of the intervention target acting in the minimal context $\vx_M$. 
We may then define the minimal context $\ci$-graph $\GG^\ci_{\vx_M}$ with node set
\[
W_{\vx_M}^\ci = \{\omega_k : k\in [K]\} \cup [p]\setminus M
\]
and edge set
\[
E_{\vx_M}^\ci = \{\omega_{k}\rightarrow \pi_i : \pi_i\in[p]\setminus M \mbox{ and } \vx_S\in \mathbf{s}_i\cap I_{k,\vx_M}\}.
\]
That is, $\GG^\ci_{\vx_M}$ is the DAG $\GG_{\vx_M}$ with an additional node $\omega_k$ for each nonempty intervention target $I_k$, with edges pointing from $\omega_k$ to the variables whose conditional distributions $P(x_{\pi_i} \mid \vx_S)$ are augmented by $I_k$. 
The minimal context graphs for the interventional CStree $\mathcal{T}^\ci$ are $\{\GG^\ci_{\vx_M}\}_{\vx_M\in\mathcal{C}_\mathcal{T}}$.

\begin{example}
    \label{ex: minimal context I-graphs}
    Consider the context-specific interventions $I_1 = \{\mathcal{S}_{\pi,2}(x_1 = 0)\}$ and $I_2 = \{\mathcal{S}_{\pi,3}(x_2 = 1)\}$ corresponding to the context-specific mechanism changes $P^{I_1}(X_2 \mid X_1 = 0)$ and $P^{I_2}(X_3 \mid X_2 = 1)$ in the chicken pox model from Example~\ref{ex:CSI causal model}.
    Let $\mathcal{T}$ be the CStree for this model constructed in Example~\ref{ex: cstree} as depicted in Figure~\ref{fig:CStreeExample}.
    The CStree $\mathcal{T}$ was shown to have the minimal context graphs $\GG_\mathcal{T}$ depicted in Figure~\ref{fig:minimal context graphs}. 
    It can be checked that the minimal context $\ci$-DAGs for this model are those in Figure~\ref{fig:minimal context I-graphs}. 
    By construction, all interventions are depicted in the minimal context graph $\GG_{\vx_\emptyset}$, as the empty context is always a subcontext of any $\vx_S\in I$.
    Comparing the $\ci$-graphs $\GG_{x_2 = 0}^\ci$ and $\GG_{x_2 = 1}^\ci$, we see that the interventional edge $\omega_2\rightarrow 3$ in $\GG_{\vx_\emptyset}$ actually corresponds to the perturbance of the factor $P^{I_0}(X_3 \mid X_2 = 1)$, while the factor $P^{I_2}(X_3 \mid X_2 = 0) = P^{I_0}(X_3 \mid X_2 = 0)$ remains invariant. 

    On the other hand, the edge $\omega_1 \rightarrow 2$ only occurs in the minimal context graph $\GG_{\vx_\emptyset}$. 
    Hence, without further assumptions or information, we can only deduce from the minimal context $\ci$-DAGs $\GG_{\mathcal{T}}^\ci$ that at least one of the mechanisms $P^{I_0}(X_2 \mid X_1 = 0)$ and $P^{I_0}(X_2 \mid X_1 = 1)$ is perturbed. 
    Specifically, this graphical representation obscures the invarance $P^{I_1}(X_2 \mid X_1 = 1) = P^{I_0}(X_2 \mid X_1 = 1)$. 
    In the following, we address this obfuscation so that we can make proper use of $\GG_\mathcal{T}^\ci$ as a compact representation of context-specific, general interventional models.
\end{example}

\begin{figure}[t]
    \centering
 \begin{tabular}{| c | c | c | c |}\hline
    \begin{tikzpicture}[thick,scale=0.25]
		\node[circle, draw, fill=black!0, inner sep=1pt, minimum width=1pt] (1) at (3.25,8) {\large$1$};
		\node[circle, draw, fill=black!0, inner sep=1pt, minimum width=1pt] (2) at (-2.25,4) {\large$2$};
		\node[circle, draw, fill=black!0, inner sep=1pt, minimum width=1pt] (3) at (8.25,4) {\large$3$};
		\node[circle, draw, fill=black!0, inner sep=1pt, minimum width=1pt] (4) at (3.25,0) {\large$4$};

            \node[circle, fill=black!0, inner sep=1pt, minimum width=1pt] (i1) at (-1.5,0.5) {$\omega_1$} ; 
            \node[circle, fill=black!0, inner sep=1pt, minimum width=1pt] (i2) at (7.5,0.5) {$\omega_2$} ; 

		\draw[->]   (1) -- (2) ;
		\draw[->]   (1) -- (3) ;
		\draw[->]   (1) -- (4) ;
		\draw[->]   (2) -- (3) ;
		\draw[->]   (2) -- (4) ;
		\draw[->]   (3) -- (4) ;

            \draw[->]   (i1) -- (2) ;
            \draw[->]   (i2) -- (3) ;

            \node at (-2.25, 8) {$\GG_{\vx_\emptyset}$} ;
	\end{tikzpicture}
 &
 \begin{tikzpicture}[thick,scale=0.25]
		\node[circle, draw, fill=black!0, inner sep=1pt, minimum width=1pt] (1) at (3.25,8) {\large$1$};
		\node[circle, draw, fill=black!0, inner sep=1pt, minimum width=1pt] (3) at (8.25,4) {\large$3$};
		\node[circle, draw, fill=black!0, inner sep=1pt, minimum width=1pt] (4) at (3.25,0) {\large$4$};

            \node[circle, fill=black!0, inner sep=1pt, minimum width=1pt] (i1) at (-1.5,0.5) {$\omega_1$} ; 
            \node[circle, fill=black!0, inner sep=1pt, minimum width=1pt] (i2) at (7.5,0.5) {$\omega_2$} ; 

		\draw[->]   (1) -- (3) ;
		\draw[->]   (3) -- (4) ;

            \node at (-2.25, 8) {$\GG_{x_2 = 0}$} ;
    \end{tikzpicture}
 &
 \begin{tikzpicture}[thick,scale=0.25]
		\node[circle, draw, fill=black!0, inner sep=1pt, minimum width=1pt] (1) at (3.25,8) {\large$1$};
		\node[circle, draw, fill=black!0, inner sep=1pt, minimum width=1pt] (3) at (8.25,4) {\large$3$};
		\node[circle, draw, fill=black!0, inner sep=1pt, minimum width=1pt] (4) at (3.25,0) {\large$4$};

            \node[circle, fill=black!0, inner sep=1pt, minimum width=1pt] (i1) at (-1.5,0.5) {$\omega_1$} ; 
            \node[circle, fill=black!0, inner sep=1pt, minimum width=1pt] (i2) at (7.5,0.5) {$\omega_2$} ; 

		\draw[->]   (1) -- (4) ;
		\draw[->]   (3) -- (4) ;

		\draw[->]   (i2) -- (3) ;

            \node at (-2.25, 8) {$\GG_{x_2 = 1}$} ;
    \end{tikzpicture}
 & 
 \begin{tikzpicture}[thick,scale=0.25]
		\node[circle, draw, fill=black!0, inner sep=1pt, minimum width=1pt] (1) at (3.25,8) {\large$1$};
		\node[circle, draw, fill=black!0, inner sep=1pt, minimum width=1pt] (2) at (-2.25,4) {\large$2$};
            \node at (8.25,4) {\,} ;
		\node[circle, draw, fill=black!0, inner sep=1pt, minimum width=1pt] (4) at (3.25,0) {\large$4$};

            \node[circle, fill=black!0, inner sep=1pt, minimum width=1pt] (i1) at (-1.5,0.5) {$\omega_1$} ; 
            \node[circle, fill=black!0, inner sep=1pt, minimum width=1pt] (i2) at (7.5,0.5) {$\omega_2$} ; 

		\draw[->]   (1) -- (2) ;

            \node at (-2.25, 8) {$\GG_{x_3 = 0}$} ;
    \end{tikzpicture}
 \\\hline
 \end{tabular}
\caption{The minimal context graphs of the CStree in Figure~\ref{fig:cstree}.}
\label{fig:minimal context I-graphs}
\end{figure}

Using the minimal context graphs $\GG_\mathcal{T}^\ci$, we may define the following context-specific $\ci$-Markov property generalizing Definition~\ref{def: I-Markov property}.

\begin{definition}
\label{def: context-specific I-Markov property}
Let $\mathcal{T}^\ci$ be an interventional CStree where $\ci = (I_0 = \emptyset, I_1,\ldots, I_K)$.
Suppose that $(\vX^{I})_{I\in\ci}$ is a sequence of strictly positive distributions.  
We say that $(\vX^{I})_{I\in\ci}$ satisfies the \emph{context-specific $\ci$-Markov property} with respect to $\mathcal{T}^\ci$ if for any $\vx_M\in\mathcal{C}_\mathcal{T}$:
\begin{enumerate}
	\item $\vX_A\independent \vX_B \mid \vX_C, \vX_M = \vx_M$ in $\vX^{I}$ for any $I\in\ci$ and any disjoint $A,B,C\subseteq[p]\setminus M$ whenever $A$ and $B$ are d-separated given $C$ in $\GG_{\vx_M}$, and
	\item $P^{I_k}(\vX_A \mid \vX_C, \vX_M = \vx_M) = P^{I_0}(\vX_A \mid \vX_C, \vX_M = \vx_M)$ for any $I_k\in\ci$ and any disjoint $A,C\subseteq[p]\setminus M$ for which $C\cup W_{\ci}\setminus\{\omega_k\}$ d-separates $A$ and $\omega_k$ in $\GG_{\vx_M}^\ci$. 
\end{enumerate}
We let $\MM(\GG_{\mathcal{T}}^\ci)$ denote the collection of all $(\vX^{I})_{I\in\ci}$ satisfying the context-specific $\ci$-Markov property with respect to $\GG_{\mathcal{T}}^\ci$.
\end{definition}

In analogy to Theorem~\ref{thm: markov to context graphs}, we would like to observe that $\MM(\mathcal{T},\ci) = \MM(\GG_{\mathcal{T}}^\ci)$. 
However, due to the fact that our interventions are context-specific, e.g., augmenting factors $P(x_i \mid \vx_S)$ for specific outcomes $\vx_S$, there is a subtlety to ensuring that $\mathcal{T}^\ci$ and $\mathcal{G}_{\mathcal{T}}^\ci$ indeed represent the same sequences of distributions, as noted in Example~\ref{ex: minimal context I-graphs}.

More precisely, in DAG models, the intervention $i\in I_k$ indicates that the conditional factors $P(x_i \mid \vx_{\pa_\GG(i)})$ are augmented for all outcomes $\vx_{\pa_\GG(i)}$, and this is captured by the single arrow $\omega_k\rightarrow i$ in $\GG^\ci$. 
However, as noted in Section~\ref{subsec:cstrees}, each outcome $\vx_{\pa_\GG(i)}$ corresponds to a different stage in the CStree interpretation of $\MM(\GG)$. 
Hence, our context-specific interventions, which target individual stages in the CStree, need not target all outcomes $\vx_{\pa_\GG(i)}$. 
In this case, the arrow $\omega_k \rightarrow i$ in the $\ci$-DAG $\GG^\ci$ would not capture the invariances $P^{I_k}(x_i \mid \vx_{\pa_\GG(i)}) = P(x_i \mid \vx_{\pa_\GG(i)})$ for the outcomes $\vx_{\pa_\GG(i)}$ not targeted by the intervention.

To accommodate for this subtlety in invariances at the context-specific level, we may impose the following assumption on the sequence of intervention targets $\ci$.

\begin{assumption}
    \label{ass:complete interventions}
    Let $\mathcal{T} = (\pi,\mathbf{s})$ be a CStree, and $I_k\subseteq \mathbf{s}$ an intervention target. 
    The target $I_k$ is \emph{complete with respect to $\vx_M\in\mathcal{C}_\mathcal{T}$} if whenever $\vx_S\in \mathbf{s}_i\cap I_{k,\vx_M}$ for a minimal context $\vx_M$ of $\mathcal{T}$ then for all $\mathcal{S}_{\pi,i}(\vx^\prime_{S^\prime})\in \mathbf{s}_i$ for which $\vx_M$ is a subcontext of $\vx^\prime_{S^\prime}$, we have that $\vx^\prime_{S^\prime}\in \mathbf{s}_i\cap I_{k,\vx_M}$.
\end{assumption}

\begin{remark}
\label{rem: DAG interventions are complete}
We note that that an intervention target $I_k$ in a DAG model $\MM(\GG)$ is always complete. 
If $\mathcal{T} = (\pi, \mathbf{s})$ is a CStree representation of $\MM(\GG)$ then $\mathcal{C}_\mathcal{T} = \{\vx_\emptyset\}$ and $\mathbf{s}_i = \{\mathcal{S}_{\pi,i}(\vx_{\pa_\GG(\pi_i)}) : \vx_{\pa_\GG(\pi_i)}\in\mathcal{X}_{\pa_\GG(\pi_i)}\}$. 
Recall that $\vx_\emptyset$ is a subcontext of any context, and let $I_k^\mathcal{T}\subset \mathbf{s}$ denote the corresponding intervention target in the CStree representation of $\MM(\GG)$. 
Then by~\eqref{eqn:I-DAGfactorization}, we have that $I^\mathcal{T}_{k,\vx_{\emptyset}} = \{\vx_{\pa_\GG(\pi_i)} : \pi_i \in I_k \mbox{ and } \vx_{\pa_\GG(\pi_i)}\in \mathcal{X}_{\pa_\GG(\pi_i)}\}$.
Hence, $\mathbf{s}_i \cap I^\mathcal{T}_{k,\vx_{\emptyset}} = \mathbf{s}_i$, showing that $I_k^\mathcal{T}$ is complete.
\end{remark}

In the DAG setting, $\GG_{\vx_M}^\ci = \GG_{\vx_\emptyset}^\ci = \GG^\ci$ is the only minimal context $\ci$-DAG.
So Remark~\ref{rem: DAG interventions are complete} shows that complete intervention targets are a context-specific generalization of intervention targets in DAG models for which the arrow $\omega_k \rightarrow \pi_i$ in the DAG $\GG_{\vx_M}^\ci$ does not hide any context-specific invariances.
In particular, by requiring an intervention to be complete with respect to a certain minimal context $\vx_M$, we intervene as we would in normal DAG models (e.g., targeting all conditional factors for a single node), but only in the local DAG model $P(\vX_{[p]\setminus M} \mid \vX_M = \vx_M)$ for $\vx_M$ in Theorem~\ref{thm: markov to context graphs}~(3). 
Hence, complete interventions with respect to $\vx_M$ should be representable with the $\ci$-DAG $\GG_{\vx_M}^\ci$, in the same fashion as $\GG^\ci$ for non-context-specific interventions. 
However, when this minimal context is empty, a complete intervention will target all stages in level $\mathbf{s}_i$, thereby obscuring any context-specific targeting in nonempty $\vx_M$. 
To avoid this obfuscation, we limit when interventions are complete with respect to the empty context $\vx_\emptyset$. 

\begin{definition}
    \label{def: CS-complete}
    An intervention target $I\subseteq \mathbf{s}$ is \emph{context-specific complete}, or \emph{CS-complete} for short, if $I$ is complete with respect to all minimal contexts in $\mathcal{C}_\mathcal{T}\setminus \vx_\emptyset$, 
    and whenever $I$ contains  $\vx_S$ whose only subcontext in $\mathcal{C}_\mathcal{T}$ is $\vx_\emptyset$ then $\mathbf{s}_i\in I$.
\end{definition}

\begin{example}
    \label{ex: CS complete}
    Consider the soft interventions from the chicken pox model in Example~\ref{ex:CSI causal model}. 
    These, respectively, are the mechanism changes $P^{I_1}(X_2 \mid X_1 = 0)$ and $P^{I_2}(X_3 \mid X_2 = 1)$. 
    The intervention $I_1 = \{\mathcal{S}_{\pi,2}(x_1 = 0)\}$ is not CS-complete since $\mathcal{S}_{\pi,2}(x_1 = 0)$ is defined by $x_1 = 0$ whose only minimal subcontext is $\vx_\emptyset$, but we have not intervened on all stages in level $2$. 
    To make $I_1$ CS-complete, we can add in the other stage $\mathcal{S}_{\pi,2}(x_1 = 1) = \{x_1 = 1\}$. 
    Note that taking $I_1 = \{\mathcal{S}_{\pi,2}(x_1 = 0), \mathcal{S}_{\pi,2}(x_1 = 1)\}$ relaxes our assumptions on the mechanisms that may be perturbed.  
    However, if only one mechanism is perturbed by the experiment, this will become apparent in parameter fitting. 

    On the other hand, the intervention $I_2 = \{\mathcal{S}_{\pi,3}(x_2 = 1)\}$ perturbing the mechanism $P^{I_2}(X_3 \mid X_2 = 1)$ is indeed CS-complete.  This is because the stage-defining context $x_2 = 1$ contains the minimal context $x_2 = 1$ and no other.  Since there is no other stage $\mathcal{S}_{\pi,3}(\vx_{S^\prime}^\prime)$ for which $\vx_{S^\prime}^\prime$ has $x_2 = 1$ as a subcontext, $I_2$ is CS-complete. Thus the sequence of intervention targets
    \[
    \ci = \{I_0 = \emptyset, I_1 =\{\mathcal{S}_{\pi,2}(x_1 = 0), \mathcal{S}_{\pi,2}(x_1 = 1)\}, I_2 = \{\mathcal{S}_{\pi,3}(x_2 = 1)\}\}
    \]
    is CS-complete.
\end{example}

Note that interventions in DAG models are CS-complete. 
For CS-complete intervention targets we obtain the following.

\begin{theorem}
    \label{prop: interventional global markov}
    Let $\mathcal{T}$ be a CStree and $\ci = (I_0 = \emptyset, I_1,\ldots, I_K)$ a sequence of CS-complete intervention targets. 
    Then $(\vX^{I_k})_{k = 0}^K\in \MM(\mathcal{T},\ci)$ if and only if $(\vX^{I_k})_{k = 0}^K$ satisfies the context-specific $\ci$-Markov property with respect to $\mathcal{T}^\ci$; i.e., $\MM(\mathcal{T},\ci) = \MM(\GG_{\mathcal{T}}^\ci)$.
\end{theorem}

Theorem~\ref{prop: interventional global markov} extends Theorem~\ref{thm: markov to context graphs} to CS-complete, general, context-specific interventions. 
It also generalizes \citep[Proposition~3.8]{YKU18} to context-specific models.

\subsection{Markov equivalence of interventional CStrees}
\label{subsec: statistical equivalence of interventional CStrees}
We may now use Theorem~\ref{prop: interventional global markov} to give a combinatorial characterization of when two interventional CStrees $\mathcal{T}^\ci$ and $\widetilde{\mathcal{T}}^{\widetilde{\ci}}$ under CS-complete interventions are \emph{Markov equivalent}; i.e., satisfy $\MM(\mathcal{T},\ci) = \MM(\widetilde{\mathcal{T}},\widetilde{\ci})$.

\begin{remark}
    \label{rem: distinct interventions}
Note that, unlike DAGs, the intervention targets $\ci$ and $\widetilde{\ci}$ may be distinct.  
This is because they are context-specific, targeting the stages defining their respective CStrees, and hence sensitive to the model factorization~\eqref{eqn:CStreefactorization}. 
In particular, the stages in $I_k\in \ci$ need not be contained in any $\widetilde{I}_k\in\widetilde{\ci}$. 
(For DAGs, an intervention target $I_k$ only considers the content on the left-hand-side of the conditioning bar in $P(x_i \mid \vx_{\pa_\GG(i)})$, whereas in context-specific interventions the targets also consider the context on the right-hand-side.)
However, when the models are equivalent, they will be in bijection in the following sense.
\end{remark}

\begin{definition}
\label{def: compatible targets}
Let $\mathcal{T}$ and  $\widetilde{\mathcal{T}}$ be two CStrees with the same set of minimal contexts $\mathcal{C}$.  
If $\mathcal{T}$ has collection of targets $\ci$ and $\widetilde{\mathcal{T}}$ has collection of targets $\widetilde{\ci}$, we say that $\ci$ and $\widetilde{\ci}$ are \emph{compatible} if there exists a bijection $\Phi: \ci\longrightarrow\widetilde{\ci}$ such that for all $\vx_M\in\mathcal{C}$
\[
\omega_I\rightarrow k\in\GG_{\vx_M}^\ci
\qquad 
\Longleftrightarrow 
\qquad
\omega_{\Phi(I)}\rightarrow k\in\GG_{\vx_M}^{\widetilde{\ci}}.
\]
\end{definition}

Provided with a compatible pair of intervention targets, the notion of two $\ci$-DAGs having the same skeleton and v-structures extends in the expected way to minimal context $\ci$-DAGs. 
This is made concrete with the following definition. 

\begin{definition}
\label{def: skeleta and v-structures}
Let $\mathcal{T},\widetilde{\mathcal{T}}$ be two CStrees with targets $\ci,\widetilde{\ci}$, respectively, and the same set of minimal contexts $\mathcal{C}$.   
We say the sequences of minimal context graphs $\GG_{\mathcal{T}}^\ci$ and $\GG_{\widetilde{\mathcal{T}}}^{\widetilde{\ci}}$ \emph{have the same skeleton and v-structures} if:
\begin{enumerate}
    \item $\ci$ and $\widetilde{\ci}$ are compatible,
    \item $\GG_{\vx_M}\in\GG_{\mathcal{T}}$ and $\widetilde{\GG}_{\vx_M}\in\GG_{\widetilde{\mathcal{T}}}$ have the same skeleton for all $\vx_M\in\mathcal{C}$,
    \item $\GG_{\vx_M}\in\GG_{\mathcal{T}}$ and $\widetilde{\GG}_{\vx_M}\in\GG_{\widetilde{\mathcal{T}}}$ have the same v-structures for all $\vx_M\in\mathcal{C}$,
    \item $w_I\rightarrow k \leftarrow j$ is a v-structure in $\GG_{\vx_M}^\ci$ if and only if $w_\Phi(I)\rightarrow k \leftarrow j$ is a v-structure in $\widetilde{\GG}_{\vx_M}^{\widetilde{\ci}}$ for all $\vx_M\in\mathcal{C}$.
\end{enumerate}
\end{definition}

We then obtain the following theorem, which is a common generalization of Corollary~\ref{cor: VP generalization} and Theorem~\ref{thm: Yang characterization}.

\begin{theorem}
\label{thm: interventional equivalence characterization}
Let $\mathcal{T}^\ci$ and $\widetilde{\mathcal{T}}^{\widetilde{\ci}}$ be interventional CStrees with $\ci$ and $\widetilde{\ci}$ CS-complete and $\emptyset\in\ci\cap\widetilde{\ci}$.  
Then $\mathcal{T}^\ci$ and $\widetilde{\mathcal{T}}^{\widetilde{\ci}}$ are Markov equivalent if and only if $\GG_\mathcal{T}^\ci$ and $\GG_{\widetilde{\mathcal{T}}}^{\widetilde{\ci}}$ have the same skeleton and v-structures.
\end{theorem}

\begin{example}
    \label{ex: I-CStree equivalence}
    Recall from Example~\ref{ex: equivalent CStrees} that the two CStrees $\mathcal{T}_A$ and $\mathcal{T}_B$ depicted in Figure~\ref{fig: equivalent CStrees} are Markov equivalent to the CStree $\mathcal{T}$ from Figure~\ref{fig:cstree} representing the chicken pox model from Example~\ref{ex:CSI causal model}. 
    In Example~\ref{ex: CS complete}, we gave a sequence $\ci$ of CS-complete intervention targets that encode the context-specific mechanism changes from Example~\ref{ex:CSI causal model}. 
    The resulting interventional CStree $\mathcal{T}^\ci$ for the chicken pox model has minimal context $\ci$-DAGs depicted in Figure~\ref{fig:minimal context I-graphs}. 
    The corresponding interventional CStrees $\mathcal{T}_A^\ci$ and $\mathcal{T}_B^\ci$ are quickly checked to have the same minimal context $\ci$-DAGs as $\mathcal{T}^\ci$ but with the edge directions between $1,2$ and $1,3$ reversed, respectively. 
    In particular, all three models have distinct v-structures and therefore are not $\ci$-Markov equivalent.
    Hence, choosing the model that best fits the observational and experimental data will allow us to distinguish the true causal structure while also accommodating for the context-specific nature of both the model itself and the considered soft interventions. 
\end{example}

\begin{remark}
    \label{rem: I-LDAG}
    Since the minimal context graph representation of a CStree may be large, provided the model admits several minimal contexts, the main reason for working with the representation $\GG_\mathcal{T}^\ci$ of an interventional CStree $\mathcal{T}^\ci$ would be to verify model equivalence. 
    A more compact representation of the interventional model may be obtained by augmenting the LDAG $(\GG,\mathcal{L})$ representation of $\mathcal{T}$. 
    Specifically, supposing that $\GG$ has node set $[p]$ and edge set $E$, we define the DAG $\GG^\ci = ([p]\cup W_\ci, E \cup E_\ci)$, where
    \[
    W_\ci = \{\omega_k : k \in [K]\} \qquad \mbox{and} \qquad
    E_{\ci} = \{\omega_k \rightarrow \pi_i : k\in [K] \mbox{ and there exists } \vx_S \in \mathbf{s}_i \cap I_k\}.
    \]
    (Note that when $\mathcal{T}$ is the CStree representation of a DAG model $\GG$, this is simply the $\ci$-DAG $\GG^\ci$ of \citet{YKU18}.)
    Define then the set of edge labels $\mathcal{L}^\ci = \mathcal{L}\cup \bigcup_{k\in [K]}L_{k,\pi_i}$
    where
    \[
    L_{k,\pi_i} = \mathcal{X}\setminus \{\vx: \vx \mbox{ has a subcontext in }\cup_{\vx_S\in I_k}\mathcal{S}_{\pi,i}(\vx_S)\}.
    \]
    Since the interventional edges in the LDAG $(\GG^\ci,\mathcal{L}^\ci)$ vanish whenever an outcome lies in an intervened stage, one can verify that any $(\vX^0,\ldots, \vX^K)\in\MM(\mathcal{T},\ci)$ also lies in $\MM(\GG^\ci,\mathcal{L}^\ci)$. 
    Hence, the LDAG $(\GG^\ci,\mathcal{L}^\ci)$ is a valid representation of the model that is conveniently more compact than the minimal context graphs $\GG_\mathcal{T}^\ci$.
    
    Note also that a characterization of context-specific, general interventional models in the language the LDAGs $(\GG^\ci,\mathcal{L}^\ci)$ that generalizes the Markov equivalence characterization of LDAGs described in Remark~\ref{rem: LDAG markov property} would reduce to considering equivalence of context graphs and thus recover our result in Theorem~\ref{thm: interventional equivalence characterization}.
    Thus, Theorem~\ref{thm: interventional equivalence characterization} may be viewed as a generalization of \citep[Theorem~4]{PNKC15} to context-specific, general interventions in the LDAGs with a factorization that easily extends the soft intervention factorization in~\eqref{eqn:I-DAGfactorization}; i.e., the family of CStrees.
\end{remark}

\begin{example}
    \label{ex: I-LDAG}
    Consider the interventional CStree model $\MM(\mathcal{T}, \ci)$ where $\mathcal{T}$ is the CStree in Figure~\ref{fig:CStreeExample} for the chicken pox model from Example~\ref{ex:CSI causal model} and $\ci$ is the CS-complete sequence of interventions from Example~\ref{ex: CS complete}. 
    The minimal context graph representation of this model is depicted in Figure~\ref{fig:minimal context I-graphs}. 
    The more compact LDAG representation for $\MM(\mathcal{T}, \ci)$ is shown in Figure~\ref{fig:I-LDAG}. 
    Note that, while this representation is more concise, it may require a bit more parsing to identify the context-specific mechanisms by comparing edge labels. 
\end{example}

\begin{figure}[t]
    \centering
    \begin{tikzpicture}[thick,scale=0.3]
		\node[circle, draw, fill=black!0, inner sep=1pt, minimum width=1pt] (H1) at (3.25,8) {\large$1$};
		\node[circle, draw, fill=black!0, inner sep=1pt, minimum width=1pt] (B1) at (-2.25,4) {\large$2$};
		\node[circle, draw, fill=black!0, inner sep=1pt, minimum width=1pt] (G1) at (8.25,4) {\large$3$};
		\node[circle, draw, fill=black!0, inner sep=1pt, minimum width=1pt] (B2) at (3.25,0) {\large$4$};

            \node[circle, fill=black!0, inner sep=1pt, minimum width=1pt] (i1) at (-12,4) {$\omega_1$} ; 
            \node[circle, fill=black!0, inner sep=1pt, minimum width=1pt] (i2) at (18,4) {$\omega_2$} ; 

		\draw[->]   (H1) -- (B1) ;
		\draw[->]   (H1) -- node[midway,sloped,above]{\scriptsize ${\{1\}}$}(G1) ;
		\draw[->]   (H1) -- node[align=center,below, rotate=-90]{{\scriptsize $\{(0,1),\, (\ast,0)\}$}} (B2) ;
		\draw[->]   (B1) -- node[align=center,below, rotate=-40]{{\scriptsize $\{(\ast,0)\}$}} (B2) ;
		\draw[->]   (G1) -- (B2) ;
		\draw[->]   (B1) -- (G1) ;

            \draw[->]   (i1) -- (B1) ;
            \draw[->]   (i2) -- node[midway,sloped,above]{\scriptsize $\mathcal{X}\setminus{\{(\ast,1, \ast, \ast)\}}$}(G1) ;
	\end{tikzpicture}
    
\caption{The $\ci$-LDAG representation of the interventional CStree model $\MM(\mathcal{T},\ci)$ in Example~\ref{ex: I-CStree equivalence}.}
\label{fig:I-LDAG}
\end{figure}

\section{Real Data Example}
\label{sec:real examples}
We give a small example of interventional CStree models in a real data scenario. 
Note that the model estimation methods used in this section are brute force, as we leave all questions of model selection (e.g., structure learning) and inference to future work, as discussed in Section~\ref{sec:discussion}.
The analyses conducted here are available at \url{https://github.com/soluslab/CStrees}.

\subsection{Context fear conditioning data set description}
\label{subsec: data description}
The data set available at the UCI Machine Learning Repository \citep{DG19} 
records expression levels of $77$ different proteins/protein modifications measured in the cerebral cortex of mice.  
Each mouse is either a control or a Ts65Dn trisomic Down Syndrome mouse. 
Each mouse was either injected with saline or treated with the drug memantine, which is believed to affect associative learning in mice.  
The mice were then trained in context fear conditioning (CFC), a task used to assess associative learning \citep{RKS98}.  
The standard CFC protocol divides mice into two groups: the \emph{context-shock} (CS) group, which are placed into a novel cage, allowed to explore, and then receive a brief electric shock, and the \emph{shock-context} (SC) group, which is placed in the novel cage, immediately given the electric shock, and thereafter allowed to explore.  
The expression levels of $77$ different proteins were measured from eight different classes of mice, defined by whether the mouse is control (c) or trisomic (t), received memantine (m) or saline (s), and whether it was in a CS or SC group for the learning task.  
The eight classes are denoted as c-CS-s, c-CS-m, c-SC-s, c-SC-m, t-CS-s, t-CS-m, t-SC-s, t-SC-m.  
There are $9,10,9,10,7,9,9$, and $9$ mice in each class, respectively.  
Fifteen measurements of each protein were registered per mouse, yielding a total of $1080$ measurements per protein. 
Each measurement is regarded as an independent sample.  

\subsection{Model selection process}
\label{subsec: model selection}
Since the number of CStrees on $p$ nodes grows much faster than the number of DAGs on $p$ nodes (see Table~\ref{fig: counting CStrees} in Appendix~\ref{appsec: enumeration}), we limit our small analysis to a single observational and a single interventional group. 
We treat the $135$ measurements taken from the group c-SC-s as observational data and the $150$ measurements taken from the group c-SC-m as interventional data, taking treatment with memantine as our intervention. 
We consider the expression levels of four proteins, each of which is believed to discriminate between between the classes c-SC-s and c-SC-m (see \citep[Table 3, Column 2]{HGK15}).  
As CStree models are for categorical data, we discretized the data set using the quantile method.  
The result is four binary random variables, one for each protein considered, with outcomes ``high (expression level)'' and ``low (expression level).''

Since the target of the intervention is unknown, to build an interventional CStree model, we need to select an appropriate CStree out of all CStrees on four binary variables with a single intervention target that could be targeting any subset of the variables in any context. 
There are $59,136$ CStrees on four binary random variables.  
As the intervention targets are latent, we need to consider all possible interventions in a given observational CStree.  
For a CStree on four binary variables, the number of such interventional models to be scored can be as large as $2^{15}$.  
To avoid excessive runtimes, we first learn an optimal equivalence class of CStrees with respect to the Bayesian Information Criterion (BIC), and then score all possible interventional CStrees that arise by targeting any subset of the stages in any element of this equivalence class.  

As the development of structure learning methods falls outside the scope of this paper, we use a naive method to learn the BIC-optimal equivalence class, which we refer to as BHC-CS. 
Given a random sample $\mathbb{D}$, BHC-CS learns a CStree on $p$ variables, with a causal ordering specified a priori, by starting with the complete dependence model (i.e., all stages are singletons) and considering for $i \in\{2,\ldots, p\}$ all possible pairwise mergings of stages in level $i$ picking the BIC-optimal merging at each iteration.  Here, each merging is done to (minimally) ensure the resulting tree is also a CStree.
When there is no longer any merging in level $i$ that increases BIC, the algorithm moves to level $i+1$ and repeats the process. 
BHC-CS is thus a version of the \emph{backwards hill-climbing} algorithm for learning staged trees \citep{CLRV20} that learns only CStrees.
BHC-CS saves the optimal model learned for each causal ordering and then returns the best scoring model overall. 
A formula for the BIC of a CStree is presented in Appendix~\ref{appsec: BIC}.

\subsubsection{pCAMKII, pPKCG, NR1 and pS6}
\label{subsubsec: protein comparison 1}
A representative of the BIC-optimal equivalence class for only observational data on the proteins pCAMKII, pPKCG, NR1 and pS6 is given in Figure~\ref{fig: context graphs for optimal observational tree 1}. 
Its staged tree representation is given Figure~\ref{fig: BIC optimal observational tree 1} in Appendix~\ref{appsec: additional results}.

\begin{figure}[t!]
	\centering
 \begin{tabular}{| c | c | c | c |}\hline
\begin{tikzpicture}[thick, scale=0.23]
 	 \node[circle, fill=black!00, inner sep=1pt, minimum width=1pt] (1) at (25,-5) {\tiny pCAMKII};
 	 \node[circle, fill=black!00, inner sep=1pt, minimum width=1pt] (2) at (32,-5) {\tiny pS6};
 	 \node[circle, fill=black!00, inner sep=1pt, minimum width=1pt] (3) at (36,2) {\tiny NR1};
 	 \node[circle, fill=black!00, inner sep=1pt, minimum width=1pt] (4) at (29,2) {\tiny pPKCG};

          \draw[->]   (1) -- (3) ;
 	 \draw[->]   (1) -- (4) ;
 	 \draw[->]   (2) -- (3) ;
 	 \draw[->]   (2) -- (4) ;
 	 \draw[<-]   (3) -- (4) ;

        \node at (23,4) {\footnotesize $\GG_{\vx_\emptyset}$} ;
   
\end{tikzpicture}
 &
\begin{tikzpicture}[thick, scale=0.23]
        \node[circle, fill=black!00, inner sep=1pt, minimum width=1pt] (21) at (46,-5) {\tiny pS6};
 	 \node[circle, fill=black!00, inner sep=1pt, minimum width=1pt] (31) at (49,2) {\tiny NR1};
 	 \node[circle, fill=black!00, inner sep=1pt, minimum width=1pt] (41) at (42,1) {\tiny pPKCG};

   
        \draw[->]   (21) -- (31) ;
 	\draw[->]   (41) -- (31) ;
  
        \node at (42,5.5) {\footnotesize $\GG_{\textrm{pCAMKII = high}}$} ;
        
\end{tikzpicture}
 &
\begin{tikzpicture}[thick, scale=0.23]
        \node[circle, fill=black!0, inner sep=1pt, minimum width=1pt] (22) at (60,-5) {\tiny pS6};
 	 \node[circle, fill=black!0, inner sep=1pt, minimum width=1pt] (32) at (62,2) {\tiny NR1};
 	 \node[circle, fill=black!0, inner sep=1pt, minimum width=1pt] (42) at (55,1) {\tiny pPKCG};
   
        \draw[->]   (22) -- (42) ;
        
        \node at (54,5.5) {\footnotesize $\GG_{\textrm{pCAMKII = low}}$} ;
        
\end{tikzpicture}
 &
\begin{tikzpicture}[thick, scale=0.23]
        \node[circle, fill=black!00, inner sep=1pt, minimum width=1pt] (13) at (68,-5) {\tiny pCAMKII};
 	 \node[circle, fill=black!00, inner sep=1pt, minimum width=1pt] (23) at (75,-5) {\tiny pS6};
 	 \node[circle, fill=black!00, inner sep=1pt, minimum width=1pt] (43) at (72,2) {\tiny NR1}; 	
   
        \draw[->]   (13) -- (23) ;
 	\draw[->]   (13) -- (43) ;
   
        \node at (68.5,4) {\footnotesize $\GG_{\textrm{pPKCG = low}}$} ;
        
\end{tikzpicture}
 \\\hline
     
 \end{tabular}
	\vspace{-0.2cm}
	\caption{An optimal observational CStree for pCAMKII, pPKCG, NR1 and pS6.}
	\label{fig: context graphs for optimal observational tree 1} 
\end{figure}
According to Corollary~\ref{cor: VP generalization}, the CStree in Figure~\ref{fig: BIC optimal observational tree 1} is in an equivalence class of size two, where the other element is given by swapping pCAMKII and pS6 in the causal ordering.  
This corresponds to reversing the arrow between these two nodes in $\GG_{\textrm{pPKCG = low}}$.  
Notice also that the arrow pPKCG$\to$NR1 is covered in $\GG_{\vx_\emptyset}$, and hence reversing this arrow would result in a Markov equivalent DAG.  
However, according to Corollary~\ref{cor: VP generalization}, this arrow is fixed among all elements of the equivalence class due to the v-structure in the context graph $\GG_{\textrm{pCAMKII=high}}$.  

When the interventional data is considered, the resulting BIC-optimal interventional CStree is given in Figure~\ref{fig: BIC optimal interventional context graphs 1}.
Its interventional staged tree representation is depicted in Figure~\ref{fig: BIC optimal interventional tree 1} in Appendix~\ref{appsec: additional results}.
The intervention is graphically represented by the node $\omega_{\textrm{mem}}$.
Precisely one stage is targeted for intervention and it is in the context pCAMKII = high.  

While this intervention introduces new v-structures, none of the new v-structures fix edges that were not already fixed in the observational context graphs.  
Hence, a targeted intervention at pCAMKII or pS6 in the context that pPKCG = low is needed to distinguish the true causal structure among the proteins. 

\begin{figure}[t!]
	\centering

\begin{tabular}{ | c | c | c | c |}\hline
\begin{tikzpicture}[thick,scale=0.23]
 	 \node[circle, fill=black!00, inner sep=1pt, minimum width=1pt] (i2) at (24.5,-1) {\footnotesize$\omega_{\textrm{mem}}$};
 	 
 	 \node[circle, fill=black!00, inner sep=1pt, minimum width=1pt] (1) at (25,-5) {\tiny pCAMKII};
 	 \node[circle, fill=black!00, inner sep=1pt, minimum width=1pt] (2) at (32,-5) {\tiny pS6};
 	 \node[circle, fill=black!00, inner sep=1pt, minimum width=1pt] (3) at (36,2) {\tiny NR1};
 	 \node[circle, fill=black!00, inner sep=1pt, minimum width=1pt] (4) at (29,2) {\tiny pPKCG};

 	 \draw[->]   (1) -- (3) ;
 	 \draw[->]   (1) -- (4) ;
 	 \draw[->]   (2) -- (3) ;
 	 \draw[->]   (2) -- (4) ;
 	 \draw[<-]   (3) -- (4) ;
 	 \draw[->]   (i2) -- (4) ;

	 \node at (23,4) {\footnotesize $\GG_{\vx_\emptyset}$} ;
  
\end{tikzpicture}
&
\begin{tikzpicture}[thick,scale=0.23]
 	 \node[circle, fill=black!00, inner sep=1pt, minimum width=1pt] (i21) at (40,-5) {\footnotesize$\omega_{\textrm{mem}}$};

 	 \node[circle, fill=black!00, inner sep=1pt, minimum width=1pt] (21) at (46,-5) {\tiny pS6};
 	 \node[circle, fill=black!00, inner sep=1pt, minimum width=1pt] (31) at (49,2) {\tiny NR1};
 	 \node[circle, fill=black!00, inner sep=1pt, minimum width=1pt] (41) at (42,1) {\tiny pPKCG};

        \draw[->]   (21) -- (31) ;
 	\draw[->]   (41) -- (31) ;
 	\draw[->]   (i21) -- (41) ;
  
        \node at (43,5) {\footnotesize $\GG_{\textrm{pCAMKII = high}}$} ;
        
\end{tikzpicture}
&
\begin{tikzpicture}[thick,scale=0.23]
 	 \node[circle, fill=black!00, inner sep=1pt, minimum width=1pt] (i22) at (54.5,-5) {\footnotesize$\omega_{\textrm{mem}}$};

	 \node[circle, fill=black!0, inner sep=1pt, minimum width=1pt] (22) at (60,-5) {\tiny pS6};
 	 \node[circle, fill=black!0, inner sep=1pt, minimum width=1pt] (32) at (62,2) {\tiny NR1};
 	 \node[circle, fill=black!0, inner sep=1pt, minimum width=1pt] (42) at (55,1) {\tiny pPKCG};

        \draw[->]   (22) -- (42) ;
        
        \node at (56,5) {\footnotesize $\GG_{\textrm{pCAMKII = low}}$} ;
        
\end{tikzpicture}
&
\begin{tikzpicture}[thick,scale=0.23]
 	  \node[circle, fill=black!00, inner sep=1pt, minimum width=1pt] (i23) at (66.5,-1) {\footnotesize$\omega_{\textrm{mem}}$};
 	 
 	 \node[circle, fill=black!00, inner sep=1pt, minimum width=1pt] (13) at (68,-5) {\tiny pCAMKII};
 	 \node[circle, fill=black!00, inner sep=1pt, minimum width=1pt] (23) at (75,-5) {\tiny pS6};
 	 \node[circle, fill=black!00, inner sep=1pt, minimum width=1pt] (43) at (72,2) {\tiny NR1}; 	

        \draw[->]   (13) -- (23) ;
 	\draw[->]   (13) -- (43) ;
   
        \node at (68.5,4) {\footnotesize $\GG_{\textrm{pPKCG = low}}$} ;
        
\end{tikzpicture}
\\\hline
\end{tabular}

	\vspace{-0.2cm}
	\caption{An optimal interventional CStree within the MEC of Figure~\ref{fig: context graphs for optimal observational tree 1}.}
	\label{fig: BIC optimal interventional context graphs 1} 
\end{figure}  

\subsubsection{pPKCG, pNUMB, pNR1 and pCAMKII}
\label{subsubsec: protein comparison 2}
To illustrate the refinement of equivalence classes via intervention, we can consider another set of four proteins: pPKCG, pNUMB, pNR1 and pCAMKII.  
A representative of the BIC-optimal equivalence class from observational data only is depicted in Figure~\ref{fig: context graphs for optimal observational tree}.
Via Corollary~\ref{cor: VP generalization}, one can deduce that the equivalence class contains three additional CStrees, which are shown in Figure~\ref{fig: other three trees} in Appendix~\ref{appsec: additional results}.  
When the interventional data is included, the BIC scoring criterion cannot definitively distinguish between the two interventional CStrees in Figure~\ref{fig: optimal interventional context graphs} given the relatively small sample size. 
By Theorem~\ref{thm: interventional equivalence characterization}, the CStree in Figure~\ref{fig: optimal interventional context graphs a} has an equivalence class of size one, whereas there are three trees in the equivalence class represented by Figure~\ref{fig: optimal interventional context graphs b}. 
A simple bootstrap on the interventional data, producing $1000$ replicates, slightly favors the model in Figure~\ref{fig: optimal interventional context graphs a} with an equivalence class of size $1$.

\begin{figure}[t!]
	\centering
 \begin{tabular}{| c | c |}\hline
\begin{tikzpicture}[thick,scale=0.23]
 	 \node[circle, fill=black!00, inner sep=1pt, minimum width=1pt] (1) at (25,-5) {\tiny pPKCG};
 	 \node[circle, fill=black!00, inner sep=1pt, minimum width=1pt] (2) at (35,-5) {\tiny pNUMB};
 	 \node[circle, fill=black!00, inner sep=1pt, minimum width=1pt] (3) at (39,2) {\tiny pNR1};
 	 \node[circle, fill=black!00, inner sep=1pt, minimum width=1pt] (4) at (29,2) {\tiny pCAMKII};

 	 \draw[->]   (1) -- (2) ;
 	 \draw[->]   (1) -- (4) ;
 	 \draw[->]   (2) -- (3) ;
 	 \draw[->]   (3) -- (4) ;

	 \node at (23,4) {\footnotesize $\GG_{\vx_\emptyset}$} ;
  
\end{tikzpicture}
&
\begin{tikzpicture}[thick,scale=0.23]
 	 \node[circle, fill=black!00, inner sep=1pt, minimum width=1pt] (21) at (51,-5) {\tiny pNUMB};
 	 \node[circle, fill=black!00, inner sep=1pt, minimum width=1pt] (31) at (57,2) {\tiny pNR1};
 	 \node[circle, fill=black!00, inner sep=1pt, minimum width=1pt] (41) at (47,0) {\tiny pCAMKII};

 	 \draw[->]   (21) -- (31) ;

	   \node at (46,4) {\footnotesize $\GG_{\textrm{pPKCG = high}}$} ;
  
\end{tikzpicture}
\\\hline
 \end{tabular}
 
	\vspace{-0.2cm}
	\caption{An optimal observational CStree for pPKCG, pNUMB, pNR1, pCAMKII.}
	\label{fig: context graphs for optimal observational tree}
\end{figure}

\begin{figure}
\centering
    	{\begin{subfigure}[t]{0.45\textwidth}
	     \centering 
\begin{tabular}{| c | c |}\hline
\begin{tikzpicture}[thick,scale=0.23]
 	 \node[circle, fill=black!00, inner sep=1pt, minimum width=1pt] (1) at (29,-5) {\tiny pPKCG};
 	 \node[circle, fill=black!00, inner sep=1pt, minimum width=1pt] (2) at (36,-5) {\tiny pNUMB};
 	 \node[circle, fill=black!00, inner sep=1pt, minimum width=1pt] (3) at (37,2) {\tiny pNR1};
 	 \node[circle, fill=black!00, inner sep=1pt, minimum width=1pt] (4) at (30,2) {\tiny pCAMKII};
 	 \node[circle, fill=black!00, inner sep=1pt, minimum width=1pt] (i2) at (32.75,-0.5) {\footnotesize$\omega_{\textrm{mem}}$};

 	 \draw[->]   (1) -- (2) ;
 	 \draw[->]   (1) -- (4) ;
 	 \draw[->]   (2) -- (3) ;
 	 \draw[->]   (3) -- (4) ;
 	 \draw[->]   (i2) -- (2) ;

	 \node at (28,4.75) {\footnotesize $\GG_{\vx_\emptyset}$} ;
\end{tikzpicture}
&
\begin{tikzpicture}[thick,scale=0.23]
 	 \node[circle, fill=black!00, inner sep=1pt, minimum width=1pt] (21) at (47,-6) {\tiny pNUMB};
 	 \node[circle, fill=black!00, inner sep=1pt, minimum width=1pt] (31) at (51.5,1) {\tiny pNR1};
 	 \node[circle, fill=black!00, inner sep=1pt, minimum width=1pt] (41) at (44.5,0.5) {\tiny pCAMKII};
 	 \node[circle, fill=black!00, inner sep=1pt, minimum width=1pt] (i21) at (41.5,-2) {\footnotesize$\omega_{\textrm{mem}}$};

 	 \draw[->]   (21) -- (31) ;
 	 \draw[->]   (i21) -- (21) ;

        \node at (43.5,4) {\footnotesize $\GG_{\textrm{pPKCG = high}}$} ;
        
\end{tikzpicture}
\\\hline
\end{tabular}
	\caption{\,}
	\label{fig: optimal interventional context graphs a}
        \end{subfigure} }
        \hfill
 {	\begin{subfigure}[t]{0.45\textwidth}
\centering 
\begin{tabular}{| c | c |}\hline
\begin{tikzpicture}[thick,scale=0.23]
 	 \node[circle, fill=black!00, inner sep=1pt, minimum width=1pt] (i2) at (34,-0.5) {\footnotesize$\omega_{\textrm{mem}}$};
 	 
 	 \node[circle, fill=black!00, inner sep=1pt, minimum width=1pt] (1) at (29,-5) {\tiny pPKCG};
 	 \node[circle, fill=black!00, inner sep=1pt, minimum width=1pt] (2) at (36,-5) {\tiny pNUMB};
 	 \node[circle, fill=black!00, inner sep=1pt, minimum width=1pt] (3) at (37,2) {\tiny pNR1};
 	 \node[circle, fill=black!00, inner sep=1pt, minimum width=1pt] (4) at (30,2) {\tiny pCAMKII};

 	 \draw[<-]   (1) -- (2) ;
 	 \draw[->]   (1) -- (4) ;
 	 \draw[<-]   (2) -- (3) ;
 	 \draw[->]   (3) -- (4) ;
 	 \draw[->]   (i2) -- (1) ;

	 \node at (28,4.75) {\footnotesize $\GG_{\vx_\emptyset}$} ;
  
\end{tikzpicture}
&
\begin{tikzpicture}[thick,scale=0.23]
 	 \node[circle, fill=black!00, inner sep=1pt, minimum width=1pt] (i21) at (41.5,-1.5) {\footnotesize$\omega_{\textrm{mem}}$};

 	 \node[circle, fill=black!00, inner sep=1pt, minimum width=1pt] (21) at (42,-6) {\tiny pNUMB};
 	 \node[circle, fill=black!00, inner sep=1pt, minimum width=1pt] (31) at (48.5,1) {\tiny pNR1};
 	 \node[circle, fill=black!00, inner sep=1pt, minimum width=1pt] (41) at (41.75,1.5) {\tiny pCAMKII};

 	 \draw[<-]   (21) -- (31) ;

	   \node at (43.5,4) {\footnotesize $\GG_{\textrm{pPKCG = high}}$} ;
  
\end{tikzpicture}
\\\hline
\end{tabular}
 	
 	\caption{\,}
	\label{fig: optimal interventional context graphs b}
 	\end{subfigure}}
 	\caption{Optimal interventional CStrees within the MEC of Figure~\ref{fig: context graphs for optimal observational tree}.}
 	\label{fig: optimal interventional context graphs}
\end{figure}

\section{Discussion}
\label{sec:discussion}
We introduced the family of context-specific conditional independence models called CStrees, which are defined according to a factorization criterion generalizing the factorization definition of a DAG model. 
This factorization definition allowed for a straightforward extension of the general interventional DAG model $\MM(\GG, \ci)$ to a context-specific setting that further accommodates context-specific, general interventions. 
We obtained a graphical characterization of model equivalence for CStree models that extends to a characterization of interventional CStree models when the intervention targets are CS-complete. 
A first natural endeavour for future work is to develop structure learning algorithms for CStree models. 
This was already initiated in the observational data regime by \cite{rios2024scalable} in a recent paper following the initial release of this article. 
It would be of valuse to extend their work to learning interventional CStree models. 

Other considerations include generalizing properties of DAGs useful in inference to the context-specific setting via CStrees, such as characterizations of decomposable models relevant in clique-tree inference algorithms. 
First steps in this direction have been taken in a recent paper by \citet{alexandr2024decomposable} that followed the initial release of this article. 

The CS-complete interventions are, in a sense made precise in Section~\ref{sec:interventionalcstrees}, the natural extension of general interventions in DAG models to a context-specific setting. 
However, it may also be of interest to investigate graphical representations of incomplete, context-specific interventions in, for instance DAG models.
In such models, the underlying model would be a conditional independence model, and hence only have the minimal context DAG $\GG_{\vx_\emptyset}$, but the interventions may be context-specific requiring additional combinatorial structure to be concisely represented.

\subsection*{Acknowledgements}
The authors would also like to thank Danai Deligeorgaki, Christiane G\"orgen, Manuele Leonelli, and Gherardo Varando for helpful discussions.
Eliana Duarte was supported by the Deutsche Forschungsgemeinschaft DFG under grant 314838170, GRK 2297 MathCoRe,  by the FCT grant 2020.01933.CEECIND, and partially supported by CMUP under the FCT grant UIDB/00144/2020.
L.~Solus was partially supported the Wallenberg Autonomous Systems and Software Program (WASP) funded by the Knut and Alice Wallenberg Foundation, the G\"oran Gustafsson Prize for Young Researchers, a project grant from KTH Digital Futures and a Starting Grant from The Swedish Research Council.

\bibliographystyle{plainnat}
\bibliography{2024bibliography}

\clearpage
\appendix

\section{Constructing the LDAG representation of a CStree}
\label{appsec: LDAG construction}

In this section, we give an explicit description as to how one constructs an LDAG $(\GG, \mathcal{L})$ for a CStree $\mathcal{T} = (\pi, \mathbf{s})$ such that $\MM(\mathcal{T})=\MM(\GG, \mathcal{L})$. 
We refer the reader to Section~\ref{sec:relatedwork} for the definition of an LDAG model and to Subsection~\ref{subsec:cstrees} for the definition of a CStree model. 

By Definition~\ref{def:CStree}, the CStree model $\MM(\mathcal{T})$ consists of all distributions $\vX$ satisfying the CSI relations $X_{\pi_i} \independent \vX_{[\pi_1:\pi_{i-1}]\setminus S} \mid \vX_S = \vx_S$ corresponding to the stages $\mathcal{S}_{\pi, i}(\vx_S)\in \mathbf{s}$. 
In particular, for all $i\in[p]$, any $\vX\in\MM(\mathcal{T})$ satisfies all CSI relations in the set 
\[
\mathcal{D}_{\pi, i} = \{ X_{\pi_i} \independent \vX_{[\pi_1:\pi_{i-1}]\setminus S} \mid \vX_S = \vx_S : \mathcal{S}_{\pi, i}(\vx_S)\in \mathbf{s}_i\}. 
\]
Let $P_{\pi_i}$ denote the union of all sets $S\subseteq[\pi_1:\pi_{i-1}]$ such that $X_{\pi_i} \independent \vX_{[\pi_1:\pi_{i-1}]\setminus S} \mid \vX_S = \vx_S\in \mathcal{D}_{\pi,i}$.
Let $\GG = ([p], E)$ denote that DAG in which $\pa_\GG(\pi_i) = P_{\pi_i}$ for all $i\in[p]$. 
It follows that $\vX$ is Markov to $\GG$, as $X_{\pi_i}$ depends only on the variables in $P_i$ through the relations in $\mathcal{D}_{\pi,i}$. 

It remains to construct the set of labels $\mathcal{L}$. 
Recall that the label $L_{\pi_j,\pi_i}$ of an edge $\pi_j \rightarrow \pi_i$ in $\GG$ is the subset of outcomes $\vx_{\pa_\GG(\pi_i)\setminus \pi_j}\in\mathcal{X}_{\pa_\GG(\pi_i)\setminus \pi_j}$ for which $X_{\pi_i}\independent X_{\pi_j} \mid \vX_{\pa_\GG(\pi_i)\setminus \pi_j} = \vx_{\pa_\GG(\pi_i)\setminus \pi_j}$. 
For each relation $X_{\pi_i} \independent \vX_{[\pi_1:\pi_{i-1}]\setminus S} \mid \vX_S = \vx_S\in \mathcal{D}_{\pi,i}$, we know that $S\subseteq \pa_\GG(\pi_i)$. 
Hence, we may apply weak union (see Subsection~\ref{subsec: context-specific conditional independence}) to recover the CSI relation
\begin{equation*}
\label{eqn: stage relation}
X_{\pi_i} \independent \vX_{[\pi_1:\pi_{i-1}]\setminus (\pa_\GG(\pi_i)\setminus\pi_j)} \mid \vX_{\pa_\GG(\pi_i)\setminus (S\cup \pi_j)}, \vX_S = \vx_S.
\end{equation*}
Applying the decomposition axiom then shows that $\vX$ satisfies the pairwise relation
\begin{equation*}
\label{eqn: decomposed stage relation}
X_{\pi_i} \independent X_{\pi_j} \mid \vX_{\pa_\GG(\pi_i)\setminus (S\cup\pi_j)}, \vX_S = \vx_S.
\end{equation*}
Following specialization, and recalling that $S\subseteq \pa_\GG(\pi_i)$, we see that $\vX$ satisfies all pairwise relations 
\begin{equation}
   X_{\pi_i} \independent X_{\pi_j} \mid \vX_{\pa_\GG(\pi_i)\setminus (S\cup\pi_j)} = \vx_{\pa_\GG(\pi_i)\setminus (S\cup\pi_j)}, \vX_S = \vx_S. 
\end{equation}
Hence, we may take
\[
L_{\pi_j, \pi_i} = \bigcup_{X_{\pi_i} \independent \vX_{[\pi_1:\pi_{i-1}]\setminus S} \mid \vX_S = \vx_S \in \mathcal{D}_{\pi,i}}\mathcal{X}_{\pa_\GG(\pi_i)\setminus (S\cup\pi_j)} \times \vX_S = \vx_S.
\]
It follows that $\vX\in\MM(\GG, \mathcal{L})$. 
Hence, $\MM(\mathcal{T})\subseteq \MM(\GG,\mathcal{L})$ for this LDAG. 

To see the reverse inclusion, start with the LDAG $(\GG, \mathcal{L})$ with label set $\mathcal{L}$ defined as above. 
It suffices to show that $\vX\in\MM(\GG, \mathcal{L})$ satisfies the relations $X_{\pi_i} \independent \vX_{[\pi_1:\pi_{i-1}]\setminus S} \mid \vX_S = \vx_S\in \mathcal{D}_{\pi,i}$ for all $i$. 
Fixing one such relation for some $i$ and context $\vx_S$, note that the label set $\mathcal{L}$ implies that $\vX\in\MM(\GG,\mathcal{L})$ satisfies the relations
\[
X_{\pi_i} \independent X_{\pi_j} \mid \vX_{\pa_\GG(\pi_i)\setminus (S\cup\pi_j)}, \vX_S = \vx_S
\,\,\,\,
\mbox{and}
\,\,\,\,
X_{\pi_i} \independent X_{\pi_k} \mid \vX_{\pa_\GG(\pi_i)\setminus (S\cup\pi_k)}, \vX_S = \vx_S
\]
for any two parents $\pi_j,\pi_k\in\pa_\GG(\pi_i)$. 
Hence, repeated application of the intersection axiom implies that $\vX$ satisfies
\begin{equation}
    \label{eqn: contraction 1}
    \vX_{\pi_i} \independent \vX_{\pa_\GG(\pi_i)\setminus S} \mid \vX_S = \vx_S.
\end{equation}

Since $\vX\in \MM(\GG)$ by definition of $\MM(\GG,\mathcal{L})$ we further have that $\vX$ satisfies the CI relation $\vX_{\pi_i}\independent \vX_{[\pi_1:\pi_{i-1}]\setminus\pa_\GG(\pi_i)} \mid \vX_{\pa_\GG(\pi_i)}$. 
Recalling that $S\subseteq \pa_\GG(\pi_i)$ and applying specialization, we obtain that $\vX$ satisfies the relation
\begin{equation}
\label{eqn: contraction 2}
\vX_{\pi_i} \independent \vX_{[\pi_1:\pi_{i-1}]\setminus\pa_\GG(\pi_i)} \mid \vX_{\pa_\GG(\pi_i)\setminus S}, \vX_S = \vx_S.
\end{equation}
Applying contraction to~\eqref{eqn: contraction 1} and~\eqref{eqn: contraction 2}, we find that $\vX$ satisfies 
\[
\vX_{\pi_i} \independent \vX_{[\pi_1: \pi_{i-1}]\setminus S} \mid \vX_S = \vx_S.
\]
Since the choice of $i$ and relation in $\mathcal{D}_{\pi, i}$ were arbitrary, we see that $\vX\in \MM(\GG, \mathcal{L})$ satisfies all CSI relations defining the model $\MM(\mathcal{T})$. Thus, $\MM(\mathcal{T}) = \MM(\GG, \mathcal{L})$, and the proof is complete.

\section{Proofs for Section~\ref{sec:cstrees}}
\label{appsec: observational proofs}

\begin{proof}[Proof of Theorem~\ref{thm: containment}]
The containment $\mathbb{D}\subseteq \mathbb{C}$ is presented in Remark~\ref{rem: DAGs are CStrees}. 
To see that this containment is strict we may construct a CStree that is not a DAG. 
Let $\mathcal{T} = (\pi, \mathbf{s})$ over three binary variables $X_1,X_2,X_3$, where $\pi = 123$ and 
\[
\mathbf{s} = \bigcup_{x_1\in\{0,1\}}\{\mathcal{S}_{\pi,2}(x_1)\}\cup 
\{\mathcal{S}_{\pi, 3}(x_1 = 0)\}\cup
\bigcup_{x_1x_2\in\{(1,0),(1,1)\}}\{\mathcal{S}_{\pi,3}(x_1x_2)\}.
\]
In this case, level $3$ of $\mathcal{T}$ is defined by three conditional distributions:
\[
X_3 | X_1 = 0 \sim \bernoulli(\theta_1), \, \, \, X_3 \mid X_{1,2} = (1,0) \sim \bernoulli(\theta_2), \mbox{ and} \,\, \,  X_3 \mid X_{1,2} = (1,1) \sim \bernoulli(\theta_3)
\]
The first conditional distribution implies $X_3 \independent X_2 \mid X_1 = 0$, so $\MM(\mathcal{T})$ is not the dependence model on $(X_1,X_2,X_3)$. 
On the other hand, any generic choice of $0 < \theta_1,\theta_2, \theta_3< 1$ results in $X_3 \not\independent X_2 \mid X_1 = 1$. 
Hence, the model does not satisfy $X_3 \independent X_2 \mid X_1$, and therefore is not a conditional independence model. 
Thus, it cannot be a DAG model, proving that $\mathbb{D}\subsetneq \mathbb{C}$. 

The inclusion $\mathbb{C}\subseteq \mathbb{L}$ is proven in Section~\ref{appsec: LDAG construction}. 
We observe that this inclusion is strict by way of example. 
Let $\GG = ([3], E)$ where 
\[
E = \{1\rightarrow 2, 1\rightarrow 3, 2\rightarrow3\},
\]
let 
$
L_{1,3} = \{0\}\subset\mathcal{X}_2,
$
and let 
$
L_{2,3} = \{1\}\subset\mathcal{X}_1.
$
It follows that $\vX\in\MM(\GG, \mathcal{L})$ satisfies the two CSI relations
\[
X_3 \independent X_1 \mid X_2 = 0\qquad \mbox{and} \qquad
X_3 \independent X_2 \mid X_1 = 1. 
\]
The former relation implies 
\[
P(X_3 \mid X_1 = 0, X_2 = 0 ) = P( X_3 \mid X_1 = 1, X_2 = 0),
\]
and the latter implies
\[
P(X_3 \mid X_1 = 1, X_2 = 0) = P(X_3 \mid X_1 = 1, X_2 = 1). 
\]
Hence, all three of the above conditional probabilities are equal.
It follows that $\MM(\GG, \mathcal{L})$ admits a staged tree representation with causal order $\pi = 123$ and a single non-singleton stage 
\[
\{(0,0), (1,0), (1,1))\}\subset\mathcal{X}_{\{1,2\}}.
\]
However, this stage does not admit a stage-defining context, and therefore $\MM(\GG, \mathcal{L})$ cannot be a CStree model.

To see the inclusion $\mathbb{L}\subseteq\mathbb{S}$, we first recall the definition of a (stratified) staged tree model.  
A \emph{(stratified) staged tree} is a pair $\mathcal{T} = (\pi, \mathbf{s})$ where $\pi$ is a variable ordering and $\mathbf{s}$ is a collection of sets. 
Given $\vX = (X_1,\ldots, X_p)$ a joint categorical distribution and an ordering $\pi = \pi_1\cdots \pi_p$, we construct a rooted tree $\mathcal{T}$ with node set $\{r\}\cup \bigcup_{i \in\{2,\ldots, p+1\}}\mathcal{X}_{\pi_1:\pi_{i-1}}$ and edge set containing for all $i$, $\vx_{\pi_1:\pi_{i-1}} \rightarrow \vx_{\pi_1:\pi_{i-1}}x_{\pi_i}$ for all $\vx_{\pi_1:\pi_{i-1}}\in\mathcal{X}_{\pi_1:\pi_{i-1}}$ and all $x_{\pi_i}\in\mathcal{X}_{\pi_i}$.  
We also include the edges $r\rightarrow x_{\pi_1}$ for all $x_{\pi_1}\in\mathcal{X}_{\pi_i}$. 
This tree is the same as constructed for CStrees in Subsection~\ref{subsec:cstrees}, so we may use the same terminology. 
For each $i\in[p]$, we then partition the nodes $\mathcal{X}_{\pi_1:\pi_{i-1}}$ in level $i$ into disjoint sets called stages. 
Nodes in the same stage are colored the same.  
When two nodes $\vx_{\pi_1:\pi_{i-1}}, \vx_{\pi_1:\pi_{i-1}}^\prime$ are colored the same it encodes the conditional invariance $P(X_{\pi_i} \mid \vX_{\pi_1:\pi_{i-1}} = \vx_{\pi_1:\pi_{i-1}}) = P(X_{\pi_i} \mid \vX_{\pi_1:\pi_{i-1}} = \vx_{\pi_1:\pi_{i-1}}^\prime)$. 
The staged tree model $\MM(\mathcal{T})$ consists of all distributions satisfying these pairwise conditional invariances. 

Note now that an LDAG model $\MM(\GG,\mathcal{L})$ consists of all joint distributions satisfying a collection of CSI relations $X_{\pi_i} \independent X_{\pi_j} \mid \vX_{\pa_\GG(\pi_i)\setminus \{\pi_j\}} = \vx_{\pa_\GG(\pi_i)\setminus \{\pi_j\}}$ for $\vx_{\pa_\GG(\pi_i)\setminus\pi_j}\in L_{\pi_j,\pi_i}$ and the collection of CI relations $X_{\pi_i} \independent \vX_{[\pi_1:\pi_{i-1}]\setminus\pa_\GG(\pi_i)} \mid \vX_{\pa_\GG(\pi_i)}$ for all $i\in[p]$.
We claim that this set of relations corresponds to a set of pairwise conditional invariances $P(X_{\pi_i} \mid \vX_{\pi_1:\pi_{i-1}} = \vx_{\pi_1:\pi_{i-1}}) = P(X_{\pi_i} \mid \vX_{\pi_1:\pi_{i-1}} = \vx_{\pi_1:\pi_{i-1}}^\prime)$ defining a staged tree model. 
To see this note that, by definition, the CI relation $X_{\pi_i} \independent \vX_{[\pi_1:\pi_{i-1}]\setminus\pa_\GG(\pi_i)} \mid \vX_{\pa_\GG(\pi_i)}$ corresponds to a collection of pairwise conditional invariances
\[
P(X_{\pi_i} \mid \vx_{[\pi_1:\pi_{i-1}]\setminus\pa_\GG(\pi_i)}, \vx_{\pa_\GG(\pi_i)}) = P(X_{\pi_i} \mid \vx_{[\pi_1:\pi_{i-1}]\setminus\pa_\GG(\pi_i)}^\prime, \vx_{\pa_\GG(\pi_i)})
\]
for all $\vx_{[\pi_1:\pi_{i-1}]\setminus\pa_\GG(\pi_i)}, \vx_{[\pi_1:\pi_{i-1}]\setminus\pa_\GG(\pi_i)}^\prime\in \mathcal{X}_{[\pi_1:\pi_{i-1}]\setminus\pa_\GG(\pi_i)}^\prime$ and each $\vx_{\pa_\GG(\pi_i)}\in\mathcal{X}_{\pa_\GG(\pi_i)}$.
Incorporating a CSI relation $X_{\pi_i} \independent X_{\pi_j} \mid \vX_{\pa_\GG(\pi_i)\setminus \{\pi_j\}} = \vx_{\pa_\GG(\pi_i)\setminus \{\pi_j\}}$ into this invariance simply shifts $x_{\pi_j}$ from the set of fixed elements to the set of varying elements:
\[
P(X_{\pi_i} \mid \vx_{[\pi_1:\pi_{i-1}]\setminus(\pa_\GG(\pi_i)\cup\pi_j)}, \vx_{\pa_\GG(\pi_i)\setminus\pi_j}) = P(X_{\pi_i} \mid \vx_{[\pi_1:\pi_{i-1}]\setminus(\pa_\GG(\pi_i)\cup\pi_j)}^\prime, \vx_{\pa_\GG(\pi_i)\setminus\pi_j}).
\]
Since all of the considered joint outcomes $\vx_{[\pi_1:\pi_{i-1}]\setminus\pa_\GG(\pi_i)}\vx_{\pa_\GG(\pi_i)}$, $\vx_{[\pi_1:\pi_{i-1}]\setminus\pa_\GG(\pi_i)}^\prime\vx_{\pa_\GG(\pi_i)}$ or $\vx_{[\pi_1:\pi_{i-1}]\setminus(\pa_\GG(\pi_i)\cup\pi_j)}\vx_{\pa_\GG(\pi_i)\setminus\pi_j}$, $\vx_{[\pi_1:\pi_{i-1}]\setminus(\pa_\GG(\pi_i)\cup\pi_j)}^\prime\vx_{\pa_\GG(\pi_i)\setminus\pi_j}$ is an element of $\mathcal{X}_{\pi_1:\pi_{i-1}}$, it follows that the relations defining the LDAG model $\MM(\GG,\mathcal{L})$ are simply a collection of conditional invariances that also define a staged tree model. 
Hence, for any LDAG $(\GG,\mathcal{L})$ there is a (stratified) staged tree $\mathcal{T}$ such that $\MM(\GG,\mathcal{L}) = \MM(\mathcal{T})$; i.e., $\mathbb{L}\subseteq\mathbb{S}$.

To see that this inclusion is strict, consider the staged tree with causal order $\pi = 123$, with a single non-singleton stage 
\[
\mathcal{S} = \{(0,0), (1,1)\}\subseteq\mathcal{X}_{\{1,2\}},
\]
which corresponding to the conditional invariance
\begin{equation}
\label{eqn: conditional invariance}
P(X_3 \mid X_1 = 0, X_2 = 0) = P(X_3 \mid X_1 = 1, X_2 = 1).
\end{equation}
Recall then that an LDAG encodes CSI relations $\vX_A \independent \vX_B \mid \vX_C, \vX_S = \vx_S$, which correspond to conditional invariances of a particular form; namely, 
\begin{equation}
    \label{eqn: CSI conditional invariance}
    P(\vX_A \mid \vx_B, \vx_C, \vx_S) = P(\vX_A \mid \vx_B^\prime, \vx_C, \vx_S)
\end{equation}
for each $\vx_C \in\mathcal{X}_C$ and all $\vx_B,\vx_B^\prime\in\mathcal{X}_B$. 
It is easy to see that the relation~\eqref{eqn: conditional invariance} is the only relation satisfied by all distributions in the model $\MM(\mathcal{T})$ (this can also be checked computationally via computer algebra software). 
Furthermore, there is no CSI relation that allows us to express the relation~\eqref{eqn: conditional invariance} in the form~\eqref{eqn: CSI conditional invariance}.
Thus, the model $\MM(\mathcal{T})$ cannot be expressed as an LDAG model, and we conclude that $\mathbb{L}\subsetneq \mathbb{S}$. 
\end{proof}

\begin{proof}[Proof of Lemma~\ref{lem: subcontexts}]
Suppose that $\vX_M = \xx_M\notin \mathcal{C}_\mathcal{T}$. 
Then, by definition of a minimal context, there must exist $\emptyset\neq T\subseteq M$ such that
\[
\vX_A \independent \vX_B \mid \vX_{C\cup T}, \vX_{M\setminus T} = \vx_{M\setminus T}\in\J(\mathcal{T}).
\]
Picking $T$ to be any maximal subset of $M$ with respect to this property yields $\vX_{M\setminus T} = \vx_{M\setminus T}\in\mathcal{C}_\mathcal{T}$, and hence $\vX_M = \vx_M$ has the minimal context $\vX_{M\setminus T} = \vx_{M\setminus T}$ as a subcontext.  

Applying this observation to the stage-defining relations 
\[
X_{\pi_i} \independent \vX_{[\pi_1:\pi_{i-1}]\setminus S} \mid \vX_S = \vx_S
\]
shows that any stage-defining context contains a minimal context.
\end{proof}

\begin{proof}[Proof of Theorem~\ref{thm: markov to context graphs}]
\label{proof: markov to context graphs}
We first show that (1) and (2) are equivalent. 
Suppose that $\vX\in\MM(\mathcal{T})$.  
Then $\vX$ entails all CSI relations 
\begin{equation}
\label{eqn: must entail}
X_{\pi_k} \independent \vX_{[\pi_1:\pi_{k-1}]\setminus S} \mid \vX_S = \vx_S\in\mathcal{D}_{\pi,k}
\end{equation}
for all $k\in[p]$. 
Hence, $\vX$ entails all relataions in the set $\mathcal{D}_\mathcal{T}$ defining $\MM(\mathcal{T})$.
Hence, $\vX$ satisfies all CSI relations in $\J(\mathcal{T})$, and therefore satisfies all CSI relations in $\J_{\vx_M}$ for all $\vx_M\in\mathcal{C}_\mathcal{T}$. 
Hence, $\vX$ satisfies the global Markov property with respect to $\mathcal{T}$. 

Conversely, suppose that $\vX\in\MM(\GG_\mathcal{T})$.  
Then for all $\vx_M\in\mathcal{C}_\mathcal{T}$, $\vX$ entails $\vX_A\independent \vX_B\mid \vX_C, \vX_M = \vx_M$ whenever $A$ and $B$ are d-separated given $C$ in the minimal I-MAP $\GG_{\vx_M}$ of $\J_{\vx_M}$.  
To see that $\vX$ factorizes according to $\mathcal{T}$, it suffices to show that, for all $k\in[p]$, $\vX$ entails each relation in~\eqref{eqn: must entail}.
By definition of $\J(\mathcal{T})$, each CSI relation \eqref{eqn: must entail} is in
$\J(\mathcal{T})$.  
So by Lemma~\ref{lem: subcontexts}, either $\vx_{S}\in\mathcal{C}_\mathcal{T}$
or there exists $M\subset S$ such that the restriction $\xx_M$ of $\vx_S$ is in $\mathcal{C}_\mathcal{T}$ 
and \eqref{eqn: must entail} is implied by specialization of
\begin{equation}
    \label{eqn: unspecial}
    X_{\pi_k} \independent \vX_{[\pi_1:\pi_{k-1}]\setminus S}\mid X_{S\setminus M},X_M=\vx_M\in\J_{\vx_M}.
\end{equation}
It remains to see that this relation is realized as a d-separation in $\GG_{\vx_M}$. 
By weak union,~\eqref{eqn: unspecial} implies that 
\begin{equation*}
    X_{\pi_k} \independent X_{\pi_j}\mid \vX_{[\pi_1:\pi_{k-1}]\setminus (S\cup\{\pi_j\})},\vX_{S\setminus M},X_M=\vx_M\in\J_{\vx_M}
\end{equation*}
for all $\pi_j\in[\pi_1:\pi_{k-1}]\setminus S$, and hence, the minimal I-MAP $\GG_{\vx_M}$ does not contain the edges $\pi_j\rightarrow \pi_k$ for all $\pi_j\in[\pi_1:\pi_{k-1}]\setminus S$.  
In particular, $\pi_k$ and $\pi_j$ are d-separated in $\GG_{\vx_M}$ given $[\pi_1:\pi_{k-1}]\setminus (S\cup\{\pi_j\}) \cup S\setminus M$ for all $\pi_j\in[\pi_1:\pi_{k-1}]\setminus S$.
Since d-separation in DAGs satisfies the intersection axiom \citep{SL14}, iterated application of this axiom shows that $\pi_k$ and $[\pi_1:\pi_{k-1}]\setminus S$ are d-separated given $S\setminus M$ in $\GG_{\vx_M}$.
Applying specialization, it follows that $\vX$ entails the CSI relations in \eqref{eqn: must entail}, and therefore lies in the model $\MM(\mathcal{T})$. 

We now show that (2) and (3) are equivalent. 
Let $\vx_M\in \mathcal{C}_\mathcal{T}$, and set $g(\vX_{[p]\setminus M}) := f(\vX_{[p]\setminus M} \mid \vX_M = \vx_M)$.  
Since $\vX$ is Markov to $\GG_\mathcal{T}$, whenever $A$ and $B$ are $d$-separated in $\GG_{\vx_M}$ given $C$, we have that $\vX$ entails $\vX_A \independent \vX_B \mid \vX_C, \vX_M = \xx_M$; or equivalently, 
\[
\frac{f(\vx_A, \vx_B, \vx_C, \vx_M)}{f(\vx_B, \vx_C, \vx_M)} = \frac{f(\vx_A,\vx_C, \vx_M)}{f(\vx_C, \vx_M)},
\]
for any $(\vx_A,\vx_B,\vx_C)\in\mathcal{X}_{A\cup B\cup C}$.
Since for any $\vx_{[p]\setminus M}\in\mathcal{X}_{[p]\setminus M}$,
\[
g(\vx_{[p]\setminus M}) = \frac{1}{f(\vx_M)}f(\vx_{[p]\setminus M}, \vx_M), 
\]
then 
\[
g(\vx_A,\vx_B,\vx_C) = \sum_{\yy\in\mathcal{X}_{[p]\setminus A\cup B \cup C\cup M}}\frac{f(\yy,\vx_A,\vx_B,\vx_C,\xx_M)}{f(\vx_M)} = \frac{1}{f(\vx_M)}f(\vx_A,\vx_B,\vx_C,\vx_M),
\]
and similarly for $g(\vx_B, \vx_C), g(\vx_A,\vx_C)$, and $g(\vx_C)$.  
Hence, whenever $A$ and $B$ are $d$-separated in $\GG_{\vx_M}$ given $C$, we have that 
\[
\frac{g(\vx_A, \vx_B, \vx_C)}{g(\vx_B, \vx_C)} = \frac{g(\vx_A,\vx_C)}{g(\vx_C)}.  
\]
Therefore, $g$ entails $\vX_A \independent \vX_B \mid \vX_C$ whenever $A$ and $B$ are $d$-separated given $C$ in $\GG_{\vx_M}$ for all $\vx_M\in\mathcal{C}_\mathcal{T}$ if and only if $\vX\in\MM(\GG_\mathcal{T})$.  
It follows that $\vX$ satisfies the global Markov property with respect to $\mathcal{T}$ if and only if for all $\vx_M\in\mathcal{C}_\mathcal{T}$,
\begin{equation*}
\begin{split}
f(\vX_{[p]\setminus M} \mid \vX_M = \vx_M) = g(\vX_{[p]\setminus M}) &= \prod_{k\in[p]\setminus M}g(X_k\mid \vX_{\pa_{\GG_{\vx_M}}(k)}),\\
&= \prod_{k\in[p]\setminus M}\frac{g(X_k, \vX_{\pa_{\GG_{\vx_M}}(k)})}{g(\vX_{\pa_{\GG_{\vx_M}}(k)})},\\
&= \prod_{k\in[p]\setminus M}\frac{f(X_k, \vX_{\pa_{\GG_{\vx_M}}(k)},\vX_M = \vx_M)}{f(\vX_{\pa_{\GG_{\vx_M}}(k)}, \vX_M = \vx_M)},\\
&= \prod_{k\in[p]\setminus M}f(X_k\mid \vX_{\pa_{\GG_{\vx_M}}(k)}, \vX_M = \vx_M),\\
\end{split}
\end{equation*}
which completes the proof.
\end{proof}

\begin{proof}[Proof of Lemma~\ref{lem: equal minimal contexts}]
\label{proof: equal minimal contexts}
Suppose, for the sake of contradiction, that there exists $\vx_M\in\mathcal{C}_\mathcal{T}$ such that $\vx_M\notin\mathcal{C}_{\mathcal{T}^\prime}$. 
Then, by definition of the set $\mathcal{C}_\mathcal{T}$ and Lemma~\ref{lem: subcontexts}, there must exist a CSI relation $\vX_A\independent \vX_B\mid \vX_C, \vX_M = \vx_M$ in $\mathcal{J}(\mathcal{T})$ that is not implied by specialization of any statement $\vX_A\independent \vX_B\mid \vX_{C\cup T},\vX_{M\setminus T} = \vx_{M\setminus T}$.  
Since $\vx_M\notin\mathcal{C}_{\mathcal{T}^\prime}$, it follows that either $\vX_A\independent \vX_B\mid \vX_C, \vX_M = \vx_M$ is not a CSI relation in $\mathcal{J}(\mathcal{T}^\prime)$ or there exists some subcontext $X_{M\setminus T} = \vx_{M\setminus T}\in\mathcal{C}_{\mathcal{T}^\prime}$ such that the statement $\vX_A\independent \vX_B\mid \vX_C, \vX_M = \vx_M$ is implied by the statement $\vX_A\independent \vX_B\mid \vX_{C\cup T}, \vX_{M\setminus T} = \vx_{M\setminus T}$ encoded by $\GG_{\vx_{M\setminus T}}^\prime\in\GG_{\mathcal{T}^\prime}$. 
It follows from Theorem~\ref{thm: markov to context graphs} that the statement $\vX_A\independent \vX_B\mid \vX_{C\cup T}, \vX_{M\setminus T} = \vx_{M\setminus T}$ is in  $\mathcal{J}(\mathcal{T}^\prime)$.  
However, such a statement cannot be in $\mathcal{J}(\mathcal{T})$, as this would imply that the statement $\vX_A\independent \vX_B\mid \vX_{C},X_{M} = \vx_{M}$ is obtained by specialization from a statement of the form $\vX_A\independent \vX_B\mid \vX_{C\cup T}, \vX_{M\setminus T} = \vx_{M\setminus T}$  in $\mathcal{J}(\mathcal{T})$. 
This latter fact would contradict our initial assumption that $\vx_{M}$ is a minimal context.  
Hence, we may assume that there is a minimal context $\vx_M$, such that the CSI relation $\vX_A\independent \vX_B\mid \vX_C, \vX_M = \vx_M$
is in $\mathcal{J}(\mathcal{T})$ but not in $\mathcal{J}(\mathcal{T}')$. 

We now note that, by \cite{DG20}, the model $\MM(\mathcal{T})$ is equal to an irreducible algebraic variety intersected with the probability simplex. 
That is, $\MM(\mathcal{T})=V(P_{\mathcal{T}})\cap \Delta_{|\mathcal{X}|-1}^{\circ}$
where $P_{\mathcal{T}}$ is a prime ideal in a polynomial ring and $V(P_{\mathcal{T}})$ is the set of all points in $\C^{|\mathcal{X}|}$ that vanish on the polynomials in $P_{\mathcal{T}}$. 
The same holds for $\mathcal{T}'$, $\MM(\mathcal{T}^\prime)=V(P_{\mathcal{T}^\prime})\cap \Delta_{|\mathcal{X}|-1}^{\circ}$. 
Since $\MM(\mathcal{T})=\MM(\mathcal{T}^\prime)$, it follows that their closures with respect to the Zariski topology are equal, namely $V(P_{\mathcal{T}})=V(P_{\mathcal{T}'})$
and hence $P_{\mathcal{T}}=P_{\mathcal{T}^\prime}$. 
In particular, every equation that is satisfied by every distribution in $\MM(\mathcal{T})$ is also an equation
satisfied by every distribution in $\MM(\mathcal{T}')$. 
Since $\vX_A\independent \vX_B\mid \vX_C, \vX_M = \vx_M$ is in $\mathcal{J}(\mathcal{T})$, then every distribution in $\MM(\mathcal{T})$ satisfies $\vX_A\independent \vX_B\mid \vX_C, \vX_M = \vx_M$. 
By \citep[Proposition 4.1.6]{S19} restricted to the context-specific setting, 
there is a set of polynomials associated to the statement $\vX_A\independent \vX_B\mid \vX_C, \vX_M = \vx_M$, which we denote by $I_{A\independent B \mid C, \vX_M=\vx_M}$,  that vanish at 
every distribution in $\MM(\mathcal{T})$.
In particular, $I_{A\independent B \mid C, \vX_M=\vx_M}\subset P_{\mathcal{T}}=P_{\mathcal{T}'}$. 
Hence every  polynomial in $I_{A\independent B \mid C, \vX_M=\vx_M}$ vanishes at every distribution in $\MM(\mathcal{T}^\prime)$,
which means that every distribution in $\MM(\mathcal{T}^\prime)$
satisfies the statement $\vX_A\independent \vX_B\mid \vX_C, \vX_M = \vx_M$. But this implies $\vX_A\independent \vX_B\mid \vX_C, \vX_M = \vx_M \in J(\mathcal{T}^\prime)$, a contradiction. 
Hence, $\MM(\mathcal{T})$ and $\MM(\mathcal{T}')$ have the same set of minimal contexts.
\end{proof}

\begin{proof}[Proof of Theorem~\ref{thm: first characterization}]
\label{proof: first characterization}
Suppose $\mathcal{T}$ and $\mathcal{T}^\prime$ are Markov equivalent.
By Lemma~\ref{lem: equal minimal contexts}, it follows that $\mathcal{C}:=\mathcal{C}_\mathcal{T} = \mathcal{C}_{\mathcal{T}^\prime}$.
So we need to show that for all $\vx_M\in\mathcal{C}$, for any disjoint subsets $A,B,C\subset[p]\setminus M$ with $A,B\neq\emptyset$, that $A$ and $B$ are d-separated given $C$ in $\GG_{\vx_M}$ if and only if $A$ and $B$ are d-separated given $C$ in $\GG_{\vx_M}^\prime$.
\bigskip

For the sake of contradiction, suppose that $A$ and $B$ are d-separated given $C$ in $\GG_{\vx_M}$ but $d$-connected given $C$ in $\GG_{\vx_M}^\prime$.  
Let $\pi = \langle i_0,\ldots, i_M\rangle$ be a d-connecting path between $i_0\in A$ and $i_M\in B$ given $C$ in $\GG_{\vx_M}^\prime$.  
Let $\overline{\GG}\subset\GG_{\vx_M}^\prime$ be the subgraph of $\GG_{\vx_M}^\prime$ consisting of all nodes and edges on $\pi$ together with all nodes and edges on a directed path from any node $i_j$ of $\pi$ that is the center of a collider subpath of $\pi$ to a node in $S$.  
Suppose also that any remaining nodes of $S$ not captured in the above paths are included in $\overline{G}$ as isolated nodes.   
Let $V$ denote the set of nodes of $\overline{G}$.  
By \cite[Lemma 12]{M13}, there exists a discrete distribution $\vX_V$ that is Markov to $\overline{\GG}$ for which $X_{i_0} \not\independent X_{i_M} \mid X_S$ holds in $\vX_V$.  
As $\overline{\GG}$ is a subDAG of $\GG_{X_C = x_C}^\prime$, it follows that the subword of the causal ordering of $\mathcal{T}^\prime$ on the elements of $V$ is a linear extension of $\overline{\GG}$.  
Hence, we can factor $\vX_V$ according to a subtree of $\mathcal{T}^\prime$ in the following way: 

Let $P_V(x) = \prod_{i\in V}P_V(x_i \mid x_{\pa_{\overline{\GG}}(i)})$ be the probability mass function for $\vX_V$.  
For each $i\in V$, consider its associated level (without loss of generality, level $i$) in the tree $\mathcal{T}^\prime$.  
Let $\xx = (x_1,\ldots,x_p)\in\mathcal{X}$ be any outcome of $(X_1,\ldots,X_p)$ that includes the context $\vx_M$.  
For every $i\in V$, it follows that the root-to-leaf path in $\mathcal{T}^\prime$ corresponding to the outcome $\xx$ passes through exactly one stage in level $i$.  
We assign the parameters on the edges emanating from nodes in this the stage value $P_V(x_i \mid \xx_{\pa_{\overline{\GG}}(i)})$, for all $x_i\in\mathcal{X}_{\{i\}}$.  
We then take a generic sequence of numbers $\alpha_1^{(i)},\ldots,\alpha_{|\RR_{\{i\}}|}^{(i)}\in(0,1)$ that sum to one and assign these values to all other parameters on any edge emanating from level $i$, one parameter to each edge of every floret (always assigned in the same order for every floret, say top-to-bottom).
Here, a \emph{floret} is the set of edges emanating out of a single node.
In particular, we let $\alpha_{x_i}^{(i)}$ be the element in this sequence that is always assigned to the edge emanating from the floret that corresponds to the outcome $x_i$ of $X_i$.  
We similarly assign parameters to all edges on levels of $\mathcal{T}^\prime$ corresponding to variables $X_i$ with $i\notin V$.
It follows that such a specification of parameters factors according to $\mathcal{T}^\prime$, and hence specifies a distribution $\vX\in\MM(\mathcal{T}^\prime)$ with mass function $P$.  
(This is because the assignment of parameters we have made corresponds to a staging of a tree with the same causal ordering as $\mathcal{T}^\prime$ whose stages are a coarsening of those in $\mathcal{T}^\prime$.)
Moreover, for every $\xx_V\in\mathcal{X}_V$ we have that $P(\vx_V\mid \vx_C)=P_V(\vx_V)$. 
It follows that $X_{i_0}\not\independent X_{i_M}\mid \vX_C, \vX_M = \vx_M$ holds in $\vX$.  
As $\MM(\mathcal{T}) = \MM(\mathcal{T}^\prime)$, and $\vX$ factors according to $\mathcal{T}^\prime$, it must be that $\vX$ also factors according to $\mathcal{T}$.  
By Theorem~\ref{thm: markov to context graphs}, we know that $\vX\in\MM(\GG_\mathcal{T})$.  
Hence, as $\vx_M\in\mathcal{C}_\mathcal{T}$, and $A$ and $B$ are d-separated given $C$ in $\GG_{\vx_M}$, then it must be that $i_0$ and $i_M$ are also d-separated given $C$ in $\GG_{\vx_M}$.  
Hence, as $\vX\in \MM(\GG_\mathcal{T})$, it must be that $\vX$ entails $X_{i_0}\independent X_{i_M}\mid \vX_C, \vX_M = \vx_M$, which is a contradiction.  
Thus, we conclude that $A$ and $B$ are d-separated given $C$ in $\GG_{\vx_M}$ if and only if $A$ and $B$ are d-separated given $C$ in $\GG_{\vx_M}^\prime$.
Hence, $\GG_{\vx_M}$ and $\GG_{\vx_M}^\prime$ are Markov equivalent for all $\vx_M\in\mathcal{C}$.  

Suppose now that $\GG_{\vx_M}$ and $\GG_{\vx_M}^\prime$ are Markov equivalent for all $\vx_M\in\mathcal{C}$.  
We need to show that $\MM(\mathcal{T}) = \MM(\mathcal{T}^\prime)$.
If $\vX\in\MM(\mathcal{T})$, by Theorem~\ref{thm: markov to context graphs}, $\vX\in\MM(\GG_\mathcal{T})$ and hence entails $\vX_A\independent \vX_B \mid \vX_C,\vX_M = \vx_M$ whenever $A$ and $B$ are d-separated given $C$ in $\GG_{\vx_M}$.  
As $\GG_{\vx_M}$ and $\GG_{\vx_M}^\prime$ are Markov equivalent, it follows that $A$ and $B$ are d-separated given $C$ in $\GG_{\vx_M}$ if and only if $A$ and $B$ are d-separated given $C$ in $\GG_{\vx_M}^\prime$.
Hence, $\vX\in \MM(\GG_{\mathcal{T}^\prime})$.  
Again by Theorem~\ref{thm: markov to context graphs}, it follows that $\vX\in\MM(\mathcal{T}^\prime)$.  
By symmetry of this argument it follows that $\MM(\mathcal{T}) = \MM(\mathcal{T}^\prime)$, which completes the proof.
\end{proof}

\section{Proof for Section~\ref{sec:interventionalcstrees}}
\label{appsec: interventional proofs}

\begin{proof}[Proof of Lemma~\ref{lem: characterizing interventional settings}]
\label{proof:  characterizing interventional settings}
Let $(\vX^{I})_{I\in\ci}$ be a collection of distributions indexed by $\ci$ containing $I_0 = \emptyset$.
Suppose first that there exists $\vX^{0}\in\MM_\mathcal{T}$ such that for all $I\in \ci$ the distribution $\vX^{I}$ factorizes as in~\eqref{eqn:I-CStreefactorization} with respect to $\vX^{0}$. 
It follows that $\vX^{I}\in\MM_\mathcal{T}$ for all $I\in\ci$.  
So it remains to show that 
\[
P^{I}(x_i \mid \vx_S) = P^{J}(x_i \mid \vx_S)
\]
for all outcomes $\vx\in\mathcal{X}$ and all stages $\mathcal{S}_{\pi,i}(\vx_S)$ whenever $\vx_S\notin I\cup J$.
However, since $\vX^{I}$ factorizes as in~\eqref{eqn:I-CStreefactorization}, we have that 
\[
P^{I}(x_i \mid \vx_S) = P^{I_0}(x_i \mid \vx_S)
\]
for all outcomes $\vx\in\mathcal{X}$ and all stages $\mathcal{S}_{\pi,i}(\vx_S)$ whenever $\vx_S\notin I$.
Hence, for $I,J\in \ci$, it must be that
\[
P^{I}(x_i \mid \vx_S) = P^{I_0}(x_i \mid \vx_S) = P^{J}(x_i \mid \vx_S)
\]
for all outcomes $\vx\in\mathcal{X}$ and all stages $\mathcal{S}_{\pi,i}(\vx_S)$ whenever $\vx_S\notin I_k\cup I_{k^\prime}$.
Thus, $(\vX^{I})_{I\in\ci}\in \MM(\mathcal{T},\ci)$.

Conversely, Let $(\vX^0,\ldots, \vX^K)\in\MM(\mathcal{T},\ci)$. 
Since $I_0 = \emptyset$, the result follows immediately.     
\end{proof}

\begin{proof}[Proof of Theorem~\ref{prop: interventional global markov}]
\label{proof: interventional global markov}

Suppose first that $(\vX^{I})_{I\in\ci}$ satisfies the context-specific $\ci$-Markov property with respect to $\mathcal{T}^\ci$.  
By Lemma~\ref{lem: characterizing interventional settings}, it suffices to show that each $\vX^{I}$, for $I\in\ci$, factorizes as in~\eqref{eqn:I-CStreefactorization}.  
Without loss of generality, we assume that the causal order of $\mathcal{T}$ is $\pi = 12\cdots p$. 
Since $(\vX^{I})_{I\in\ci}$ satisfies the context-specific $\ci$-Markov property with respect to $\mathcal{T}^\ci$, then we know that each $\vX^{I}$ satisfies the global Markov property with respect to $\mathcal{T}$.
So by Theorem~\ref{thm: markov to context graphs}, each $\vX^{I}$ is in $\MM(\mathcal{T})$.
Hence, 
\[
P^I(\vx) = \prod_{i=1}^pP^I(x_i \mid \vx_{\pa_{\mathcal{T}}(\vx_{\pi_1:\pi_{i-1}})}) \qquad \mbox{for all outcomes $\vx = (x_1,\ldots, x_p)$.}
\]
Fix $I\in \ci$, $i\in[p]$ and an outcome $\vx\in\mathcal{X}$.
Then either $\vx_{[i-1]}\in \mathcal{S}_{\pi,i}(\vx_S)$ for some $\vx_S\in I$, or it is not. 
To prove that $\vX^I$ factorizes as in~\eqref{eqn:I-CStreefactorization}, we need to show the invariance $P^I(x_i \mid \vx_{\pa_{\mathcal{T}}(\vx_{\pi_1:\pi_{i-1}})}) = P^{I_0}(x_i \mid \vx_{\pa_{\mathcal{T}}(\vx_{\pi_1:\pi_{i-1}})})$ holds for $\vx$ in the latter case.
Suppose then that $\vx_{[i-1]}\notin \mathcal{S}_{\pi,i}(\vx_S)$ for any $\vx_S\in I$.
Then $\vx_{[i-1]}$ does not contain any $\vx_S\in I$ as a subcontext.  
We further claim that if $\vx_{[i-1]}\notin \mathcal{S}_{\pi,i}(\vx_S)$ for any $\vx_S\in I$ then $w_{I}$ is d-separated from $i$ given $\pa_{\GG_{\vx_M}}(i)\cup W_\ci\setminus\{w_{I}\}$ in $\GG_{\vx_M}^\ci$ for all $\vx_M\in\mathcal{C}_\mathcal{T}$.  
To see this, we consider first the case of $\GG_{\vx_M}^\ci$ for $\vx_M \neq \vx_\emptyset$.

Note that every minimal context $\vx_M \neq \vx_\emptyset$ is either a subcontext of a stage-defining context or not. 
If a minimal context $\vx_M$ is not a subcontext of any stage-defining context $\vx_S$, then by definition of $\GG_{\vx_M}^\ci$, there will be no edge $\omega_I\rightarrow i$ in $\GG_{\vx_M}^\ci$. 
Since $\vx_{[i-1]}$ does not contain any $\vx_S\in I$, then any minimal context $\vx_M$ that is a subcontext of $\vx_S$ could contain an arrow $\omega_I\rightarrow i$ only if there is another outcome $\yy$ with $\yy_{[i-1]}$ having $\vx_{S^\prime}^\prime\in I$ as a subcontext.
However, since $I$ is CS-complete and $\vx_M\neq \vx_\emptyset$, no $\yy_{[i-1]}$ with this property can exist. 
Hence, if $\vx_{[i-1]}\notin \mathcal{S}_{\pi,i}(\vx_S)$ for any $\vx_S\in I$
then there is no minimal context graph $\GG_{\vx_M}^\ci$ with $\vx_M \neq \vx_\emptyset$ containing the edge $\omega_{I}\rightarrow i$. 
Hence, we know that the vertex $w_{I}$ is d-separated from $i$ given $\pa_{\GG_{\vx_M}}(i)\cup W_\ci\setminus\{w_{I}\}$ in $\GG_{\vx_M}^\ci$ for all $\vx_M\in\mathcal{C}_\mathcal{T}\setminus \{\vx_\emptyset\}$.

It remains to prove the claim when $\vx_M = \vx_\emptyset$.
Notice first that since the edge $\omega_I \rightarrow i$ does not appear in any $\GG_{\vx_M}^\ci$ for $\vx_M \neq \vx_{\emptyset}$, then if this edge occurs in $\GG_{\vx_\emptyset}^\ci$, the targeted stage $\vx_S\in I$ only has the minimal context $\vx_\emptyset$ as a subcontext.
However, since $I$ is CS-complete, this would imply that $\mathbf{s}_i\subseteq I$, contradicting the assumption that $\vx_{[i-1]}\notin \mathcal{S}_{\pi,i}(\vx_S)$ for any $\vx_S\in I$.
Thus, if $\vx_{[i-1]}\notin \mathcal{S}_{\pi,i}(\vx_S)$ for any $\vx_S\in I$ then $w_{I}$ is d-separated from $i$ given $\pa_{\GG_{\vx_M}}(i)\cup W_\ci\setminus\{w_{I}\}$ in $\GG_{\vx_M}^\ci$ for all $\vx_M\in\mathcal{C}_\mathcal{T}$.

Now, to extract the required conditional invariances with the help of this observed d-separation in all $\GG_{\vx_M}^\ci$, let $\vx_M\in\mathcal{C}_\mathcal{T}$ be a minimal context that is a subcontext of $\vx_{[i-1]}$.  Then, 
\begin{equation}
\label{eqn: proven invariance}
\begin{split}
P^{I}(x_i \mid \vx_{[i-1]}) &= P^{I}(x_i \mid \vx_{\pa_{\GG_{\vx_M}}(i)},\vx_M), \\
&= P^{I_0}(x_i \mid \vx_{\pa_{\GG_{\vx_M}}(i)},\vx_M), \\
&= P^{I_0}(x_i \mid \vx_{[i-1]}).
\end{split}
\end{equation}
Here, the first equality follows from the fact that $\vX^{I}$ satisfies the global Markov with respect to $\mathcal{T}$ (specifically, $X_i$ will be independent of every variable $X_j$ with $j \in [i-1]\setminus (M\cup \pa_{\GG_{\vx_M}}(i))$ given its parents $\pa_{\GG_{\vx_M}}(i)$ in the minimal I-MAP $\GG_{\vx_M}$).  
The second equality follows from Definition~\ref{def: context-specific I-Markov property}~(2) and the observed d-separation.
The third inequality, analogous to the first, follows from the fact that $\vX^{I_0}$ satisfies the global Markov property with respect to $\mathcal{T}$. 


Finally, since $\vX^{I_0}$ and $\vX^{I}$ both satisfy the global Markov property with respect to $\mathcal{T}$, it follows from Theorem~\ref{thm: markov to context graphs} that these distributions are in $\MM(\mathcal{T})$ and thus factorize as in~\eqref{eqn:CStreefactorization}.  
Thus, we have the equalities
\[
P^{I}(x_i \mid \vx_{[i-1]}) = P^{I}(x_i \mid \vx_{\pa_{\mathcal{T}}(\vx_{[i-1]})})
\qquad 
\mbox{and}
\qquad
P^{I_0}(x_i \mid \vx_{[i-1]}) = P^{I_0}(x_i \mid \vx_{\pa_{\mathcal{T}}(\vx_{[i-1]})}).
\]
Combining these equalities with the equality in~\eqref{eqn: proven invariance}, we obtain that $\vX^{I}$ factorizes as in~\eqref{eqn:I-CStreefactorization}.
Hence, by Lemma~\ref{lem: characterizing interventional settings}, we obtain that $(\vX^{I})_{I\in\ci}\in\MM(\mathcal{T},\ci)$. 

Conversely, suppose that $(\vX^{I})_{I\in\ci}\in\MM_\ci(\mathcal{T})$. 
It follows that $\vX^{I}\in\MM_\mathcal{T}$ for each $I\in\ci$.  
By Theorem~\ref{thm: markov to context graphs}, we then have that $\vX^{I}$ is Markov to $\GG_\mathcal{T}$.  
Hence, for all $\vx_M\in\mathcal{C}_\mathcal{T}$, the distribution $\vX^{I}$ entails $\vX_A\independent \vX_B \mid \vX_C, \vX_M = \vx_M$ whenever $A$ and $B$ are $d$-separated given $C$ in $\GG_{\vx_M}$.  

To see that condition (2) of Definition~\ref{def: context-specific I-Markov property} holds, fix $I\in \ci$ and let $A,C\subset[p]\setminus M$ be disjoint subsets for which $C\cup W_{\ci}\setminus\{\omega_I\}$ $d$-separates $A$ and $\omega_I$ in $\GG_{\vx_M}^\ci$. 
Let $V$ denote the ancestral closure of $A\cup C$ in $\GG_{\vx_M}$.  
Let $B^\prime \subset V$ denote the set of all nodes in $V$ that are $d$-connected to $\omega_I$ given $C\cup W_{\ci}\setminus\{\omega_I\}$ in $\GG_{\vx_M}^\ci$, and set $A^\prime := V\setminus B^\prime\cup C$. 
By applying Theorem~\ref{thm: markov to context graphs}~(3) and Lemma~\ref{lem: characterizing interventional settings}, the remainder of the proof follows exactly as in the proof of \citep[Proposition 3.8]{YKU18}.
This is because, condition~(2) of Definition~\ref{def: context-specific I-Markov property} is an invariance condition on conditional distributions within the distributions $P^{I_0}(\vX_{[p]\setminus M} \mid \vX_M = \vx_M)$ and $P^{I_k}(\vX_{[p]\setminus M} \mid \vX_M = \vx_M)$. 
Theorem~\ref{thm: markov to context graphs}~(3) states that considering such conditional distributions for $\vX^{I_0},\vX^{I_k}\in \MM(\mathcal{T})$ amounts to working with the distributions $\vX_{[p]\setminus M}^{I_0} \mid \vX_M = \vx_M$ and $\vX_{[p]\setminus M}^{I_k} \mid \vX_M = \vx_M$, which are both in $\MM(\GG_{\vx_M})$. 
Hence, we have reduced the problem to proving certain invariances hold in a DAG model, which is exactly treated by the proof of \citep[Proposition 3.8]{YKU18}. 
\end{proof}

\begin{proof}[Proof of Theorem~\ref{thm: interventional equivalence characterization}]
\label{proof: interventional equivalence characterization}
Suppose first that $\GG_\mathcal{T}^\ci$ and $\GG_{\widetilde{\mathcal{T}}}^{\widetilde{\ci}}$ have the same skeleton and v-structures. 
Then, $\ci$ and $\widetilde{\ci}$ are compatible and
there exists a bijection $\Phi:\ci\to\widetilde{\ci}$ for which they have the same set of $d$-separations.
Hence, by Theorem~\ref{prop: interventional global markov}, $\MM(\mathcal{T},\ci) = \MM(\widetilde{\mathcal{T}},\widetilde{\ci})$

Inversely, suppose that $\GG_\mathcal{T}^\ci$ and $\GG_{\widetilde{\mathcal{T}}}^{\widetilde{\ci}}$ do not have the same skeleton and v-structures.  
Then either
\begin{enumerate}
	\item $\mathcal{C}_\mathcal{T}\neq \mathcal{C}_{\widetilde{\mathcal{T}}}$,
	\item $\mathcal{C}_\mathcal{T} = \mathcal{C}_{\widetilde{\mathcal{T}}}$, but $\ci$ and $\widetilde{\ci}$ are not compatible, 
	\item $\mathcal{C}_\mathcal{T} = \mathcal{C}_{\widetilde{\mathcal{T}}}$ and $\ci$ and $\widetilde{\ci}$ are compatible, but there is some $\vx_M\in\mathcal{C}_\mathcal{T}$ such that $\GG_{\vx_M}\in\GG_\mathcal{T}$  and $\widetilde{\GG}_{\vx_M}\in\GG_{\widetilde{\mathcal{T}}}$ do not have the same skeleton and v-structures, or
	\item $\mathcal{C}_\mathcal{T} = \mathcal{C}_{\widetilde{\mathcal{T}}}$, $\ci$ and $\widetilde{\ci}$ are compatible via a bijection $\Phi: \ci\longrightarrow \widetilde{\ci}$, but there exists $\vx_M\in\mathcal{C}_\mathcal{T}$ and a node $\omega_{I^\ast}$ in $\GG_{\vx_M}^\ci$ for which there is some $j\in[p]\setminus M$ such that $\omega_{I^\ast}\rightarrow j$ is part of a v-structure in $\GG_{\vx_M}^\ci$ but $\omega_{\Phi(I)}\rightarrow j$ is not part of a v-structure in $\widetilde{\GG}_{\vx_M}^{\widetilde{\ci}}$. 
\end{enumerate}

For case (1), note that, by Theorem~\ref{thm: first characterization}, since $\mathcal{C}_{\mathcal{T}}\neq \mathcal{C}_{\widetilde{\mathcal{T}}}$,  there exists a distribution $\vX\in \MM(\mathcal{T})$ such that $\vX\notin \MM(\widetilde{\mathcal{T}})$. 
Then the interventional setting $(\vX^{I})_{I\in\ci}$ where $\vX^{I}:=\vX$ for all $I\in \ci$ is an element of $\MM(\mathcal{T}, \ci)$ but not of $\MM(\widetilde{\mathcal{T}},\widetilde{\ci})$. 

For case (2), suppose that $\mathcal{C}_\mathcal{T} = \mathcal{C}_{\widetilde{\mathcal{T}}}$ but $\ci$ and $\widetilde{\ci}$ are not compatible. 
Suppose first that $|\ci|\neq|\widetilde{\ci}|$.  
Without loss of generality, we assume $|\ci|<|\widetilde{\ci}|$. 
Then, given any $\vX^{\emptyset}\in\MM(\widetilde{\mathcal{T}})$, we know $(\vX^{I})_{I\in\ci}\in\MM(\widetilde{\mathcal{T}},\widetilde{\ci})$, where $\vX^{I} := \vX^{\emptyset}$ for all $I\in\widetilde{\ci}$.  
However, no sequence of distributions of this length can possibly be in $\MM(\mathcal{T},\ci)$.  
Hence, $\MM(\mathcal{T},\ci)\neq\MM(\widetilde{\mathcal{T}},\widetilde{\ci})$. 

On the other hand, suppose that $|\ci| = |\widetilde{\ci}|$.  
Since $\ci$ and $\widetilde{\ci}$ are not compatible, then there is no relabeling of $\widetilde{\ci}$ such that each $\omega_I$ has the same children in $\widetilde{\GG}_{\vx_M}$ as it does in $\GG_{\vx_M}$ for all $\vx_M\in\mathcal{C}_\mathcal{T}$.  
Hence, without loss of generality, given any relabeling of $\widetilde{\ci}$ according to a bijection $\Phi: \ci\rightarrow \widetilde{\ci}$, there is some $\vx_M\in\mathcal{C}_\mathcal{T}$ for which there is an $I^\ast\in\ci$ and $k\in[p]\setminus M$ such that $\omega_{I^\ast}\rightarrow k$ is an edge of $\GG_{\vx_M}^\ci$ but $\omega_{\Phi(I^\ast)}\rightarrow k$ is not an edge of $\widetilde{\GG}_{\vx_M}^{\widetilde{\ci}}$. 
Thus, for any relabeling of $\widetilde{\ci}$ via a bijection $\Phi: \ci\longrightarrow\widetilde{\ci}$ there is a context $\vx_M\in\mathcal{C}_\mathcal{T}$ for which there is $I^\ast\in\ci$ and $k\in[p]\setminus M$ such that $\omega_{I^\ast}$ is d-connected to $k$ given $S:=\pa_{\widetilde{\GG}_{\vx_C}^{\widetilde{\ci}}}(k)$ in $\GG_{\vx_M}^\ci$ and d-separated from $k$ given $S$ in $\widetilde{\GG}_{\vx_M}^{\widetilde{\ci}}$.  
From \citet[Section 3.1]{DS20} applied to $\widetilde{\GG}_{\vx_M}^{\widetilde{\ci}}$, the d-separation statement that holds in $\widetilde{\GG}_{\vx_M}^{\widetilde{\ci}}$ translates into a set of polynomials 
$\mathrm{Inv}_{\widetilde{\GG}_{\vx_M}^{\widetilde{\ci}}}$ that vanish when evaluated at the points in 
$\MM(\widetilde{\mathcal{T}},\widetilde{\ci})$. 
Using these polynomials, we show that 
$\MM(\mathcal{T},\ci) \neq \MM(\widetilde{\mathcal{T}},\widetilde{\ci})$. 

Suppose by way of contradiction  that $\MM(\mathcal{T},\ci) = \MM(\widetilde{\mathcal{T}},\widetilde{\ci})$.
By \citet[Theorem 4.3]{DS20}, $\MM(\mathcal{T},\ci)$ is an irreducible variety intersected with a product of open
probability simplices, $\MM(\mathcal{T},\ci)=V(P_{(\mathcal{T},\ci)})\cap \Delta_{|\mathcal{X}|-1}^{\circ, (0)}\times \cdots \times \Delta_{|\mathcal{X}|-1}^{\circ,(K)} $, where $P_{(\mathcal{T},\ci)}$ is  a prime ideal in a polynomial ring and
$V(P_{(\mathcal{T},\ci)})$  is the set of all points in $\C^{(K+1)|\mathcal{X}|}$ that evaluate to zero at the elements
of $P_{(\mathcal{T},\ci)}$. 
The same holds for $\MM(\widetilde{\mathcal{T}},\widetilde{\ci})$, 
$\MM(\widetilde{\mathcal{T}},\widetilde{\ci})=V(P_{(\widetilde{\mathcal{T}},\widetilde{\ci})})\cap \Delta_{|\mathcal{X}|-1}^{\circ, (0)}\times \cdots \times \Delta_{|\mathcal{X}|-1}^{\circ,(K)} $. 
By a similar argument as in the proof of Lemma~\ref{lem: equal minimal contexts}, $P_{(\mathcal{T},\ci)}=P_{(\widetilde{\mathcal{T}},\widetilde{\ci})}$.
In particular, every equation satisfied by an interventional setting in $\MM(\mathcal{T},\ci)$ must be satisfied by every interventional setting in $\MM(\widetilde{\mathcal{T}},\widetilde{\ci})$.
Thus every polynomial in $\mathrm{Inv}_{\widetilde{\GG}_{\vx_M}^{\widetilde{\ci}}}$ vanishes
at every element in $\MM(\mathcal{T},\ci)$. 
This means that  $\omega_{I^\ast}$ is d-separated from $k$ given $S$ in $\GG_{\vx_M}^\ci$, a contradiction. 
Hence $\MM(\mathcal{T},\ci) \neq \MM(\widetilde{\mathcal{T}},\widetilde{\ci})$.

For case (3), suppose that $\mathcal{C}_\mathcal{T} = \mathcal{C}_{\widetilde{\mathcal{T}}}$ and $\ci$ and $\widetilde{\ci}$ are compatible.  
Hence, we may relabel $\widetilde{\ci}$ so that all nodes in $\GG_{\widetilde{\mathcal{T}}}^{\widetilde{\ci}}$ and $\GG_\mathcal{T}^\ci$ have the same labels and, after this relabeling, all nodes $\omega_I$ have the same children in $\GG_{\vx_M}^\ci$ and $\widetilde{\GG}_{\vx_M}^{\widetilde{\ci}}$ for all $\vx_M$.  
Suppose now that there exists $\vx_M\in\mathcal{C}_\mathcal{T}$ such that $\GG_{\vx_M}\in\GG_\mathcal{T}$ and $\widetilde{\GG}_{\vx_M}\in\GG_{\widetilde{\mathcal{T}}}$ do not have the same skeleton and v-structures.  
In this case, it follows from Theorem~\ref{thm: first characterization} that $\mathcal{T}$ and $\widetilde{\mathcal{T}}$ are not Markov equivalent.  
That is, $\MM(\mathcal{T})\neq\MM(\widetilde{\mathcal{T}})$.  
So, without loss of generality, there exists $\vX\in\MM(\mathcal{T})$ such that $\vX\notin\MM(\widetilde{\mathcal{T}})$.  
By setting $\vX^{I}:=\vX$ for all $I\in\ci$, we produce $(\vX^{I})_{I\in\ci}$ that is in $\MM(\mathcal{T},\ci)$ but not in $\MM(\widetilde{\mathcal{T}},\widetilde{\ci})$.  
Hence, $\MM(\mathcal{T},\ci)\neq\MM(\widetilde{\mathcal{T}},\widetilde{\ci})$. 

Finally, in case (4), we assume that $\mathcal{C}_\mathcal{T} = \mathcal{C}_{\widetilde{\mathcal{T}}}$ and $\ci$ and $\widetilde{\ci}$ are compatible.  
As in case (3), these assumptions ensure that we may relabel $\widetilde{\ci}$ so that all nodes in $\GG_{\widetilde{\mathcal{T}}}^{\widetilde{\ci}}$ and $\GG_\mathcal{T}^\ci$ have the same labels and, after this relabeling, all nodes $\omega_I$ have the same children in $\GG_{\vx_M}^\ci$ and $\widetilde{\GG}_{\vx_M}^{\widetilde{\ci}}$ for all $\vx_M$. 
It follows that, after this relabeling, $\GG_{\vx_M}^\ci$ and $\widetilde{\GG}_{\vx_M}^{\widetilde{\ci}}$ have the same skeleton for all $\vx_M\in\mathcal{C}_\mathcal{T}$.  
Hence, if we assume now that there is some $\vx_M\in\mathcal{C}_\mathcal{T}$ for which there is a node $\omega_{I^\ast}$ in $\GG_{\vx_M}^\ci$ and $k\in[p]\setminus M$ for which $\omega_{I^\ast}\rightarrow j$ is part of a v-structure in $\GG_{\vx_M}^\ci$ but $\omega_{I^\ast}\rightarrow j$ is not part of a v-structure in $\widetilde{\GG}_{\vx_M}^{\widetilde{\ci}}$, then it must be that there exists $k\in[p]\setminus M\cup\{j\}$ such that $\omega_{I^\ast}\rightarrow j\leftarrow k$ is a v-structure in $\GG_{\vx_M}^\ci$ but $\omega_{I^\ast}\rightarrow j\rightarrow k$ in $\widetilde{\GG}_{\vx_M}^{\widetilde{\ci}}$.  
Given such a scenario, let $S:=\pa_{\GG_{\vx_M}^\ci}(k)$.  
It follows that $w_{I^\ast}$ is d-separated from $k$ given $S$ in $\GG_{\vx_M}^\ci$ but d-connected given $S$ in $\widetilde{\GG}_{\vx_M}^{\widetilde{\ci}}$.
Similar to the argument given in case (2), the d-separation statement that holds in $\GG_{\vx_M}^{\ci}$ translates into a set of polynomials 
$\mathrm{Inv}_{\GG_{\vx_M}^{\ci}}$ that vanish when evaluated at the points in 
$\MM(\mathcal{T},\ci)$. 
Supposing for the sake of contradiction that $\MM(\mathcal{T},\ci) = \MM(\widetilde{\mathcal{T}},\widetilde{\ci})$, the same argument shows that the polynomials in $\mathrm{Inv}_{\GG_{\vx_M}^{\ci}}$ must also vanish on all points in the model $\MM(\widetilde{\mathcal{T}},\widetilde{\ci})$, which would contradict $\omega_{I^\ast}$ and $k$ being d-connected given $S$ in $\widetilde{\GG}_{\vx_M}^{\widetilde{\ci}}$.  
Hence, we conclude that $\MM(\mathcal{T},\ci) \neq \MM(\widetilde{\mathcal{T}},\widetilde{\ci})$, which completes the proof.
\end{proof}

\section{Maximum likelihood estimation}
\label{appsec: BIC}

We now present a closed-form formula for the MLE of a CStree.  We derive the formula
from results about the MLE of a staged tree model in \citep{DMS20}. In particular, the MLE of a 
staged tree model is an invariant of its statistical equivalence class.

We consider data $\mathbb{D}$ summarized as a $d_1\times d_2 \times \cdots \times d_p$ contingency
table $u$. The entry $u_{\xx}$ of $u$ is the number of occurrences
of the outcome $\xx \in \mathcal{X}$
in $\mathbb{D}$.
Given  $C\subset [p]$ we consider the marginal table $u_{C}$. The entry 
$u_{\xx,C}$ in the table $u_{C}$ is obtained by fixing the indices of the states in $C$ and summing
over all other indices not in $C$.
That is, 
\[
u_{\xx,C} = \sum_{\yy\in\mathcal{X}_{[p]\setminus C}}u_{\xx_{C},\yy}.
\]

\begin{proposition}\label{prop:mle}
Let $\mathcal{T}$ and $u$ a random sample drawn from $(X_1,\ldots, X_p)$. The MLE in $\MM(\mathcal{T},\theta)$ for
the table $u$ is
\begin{equation*} \label{eq:mle}
\hat{p}_{\xx} = \prod_{
\substack{k=1, \\ \mathcal{S}_{\pi,k}(\vx_S)\in \mathbf{s}_k}}^p \frac{u_{\xx, S\cup \{k\} }}{u_{\xx, S}}.
\end{equation*}
If $\mathcal{T}$ represents a DAG model, then $S=\pa_{\GG}(k)$.
\end{proposition}
\begin{proof}
Following the formula from \cite[Proposition 11]{DMS20}, we find the maximum likelihood
estimates for the parameters of the model. Let $y_{\theta(e)}$ be a parameter associated to the edge $e= \xx_{[k-1]}\to \xx_{[k]}$ in $\mathcal{T}$.
The quotient $u_{\xx,[k]}/u_{\xx,[k-1]}$ is the empirical estimate for the transition probability from
$\xx_{[k-1]}$ to $\xx_{[k]}$. By \cite[Remark 12]{DMS20} to
obtain the maximum likelihood estimate for $y_{\theta(e)}$
we consider all fractions $u_{\xx',[k]}/u_{\xx',[k-1]}$ such that $\theta(\xx'_{[k-1]}\to \xx'_{[k]})
= \theta(e)$ and aggregate them by adding their numerators and denominators separately. The complete set of those fractions is indexed by the elements in the stage 
$\mathcal{S}_{\pi,k}(\vx_S)\in\mathbf{s}_k$
that contains the node $\xx_{[k-1]}$. Since $\mathcal{T}$ is a CStree, any two  vertices $\xx',\xx''$
in the same stage $\mathcal{S}_{\pi,k}(\vx_S)$ are in the same level and satisfy $\xx'_{S}= \xx''_{S}$. 
Therefore,
\begin{equation*}
\hat{y}_{\theta(e)} = \frac{\sum_{\xx'_{[k]}\in \mathcal{S}_{\pi,k}(\vx_S)} u_{\xx',[k]}}{\sum_{\xx'_{[k]}\in \mathcal{S}_{\pi,k}(\vx_S)} u_{\xx',[k-1]}}
= \frac{u_{\xx,S\cup \{k\} }}{u_{\xx, S}}.
\end{equation*}
Finally, $\hat{p}_{\xx}=\prod_{k=1}^{p}\hat{y}_{\theta(\xx_{[k-1]}\to\xx_{[k]})}$ and using the formula
for $\hat{y}_{\theta(\xx_{[k-1]}\to\xx_{[k]})}$ yields the desired equation (\ref{eq:mle}).
The last assertion in the proposition follows from \cite[Example 3.2]{DS20}.
\end{proof}

Given data $\mathbb{D}$ drawn from a joint distribution over variables $(X_1,\ldots,X_p)$ and a DAG $\GG = ([p],E)$, the \emph{Bayesian Information Criterion} (BIC) is defined as
\[
\mathcal{S}(\GG,\mathbb{D}) = \log p(\mathbb{D} \mid \hat{\theta}, \GG) - \frac{d}{2}\log(n),
\]
where $\hat{\theta}$ denotes the maximum likelihood values for the DAG model parameters, $d$ denotes the number of free parameters in the model, and $n$ denotes the sample size.
In a similar fashion, the BIC of a CStree $\mathcal{T}$ is defined as 
\[
\mathcal{S}(\mathcal{T},\mathbb{D}) = \log p(\mathbb{D} \mid \hat{\theta}, \mathcal{T}) - \frac{d}{2}\log(n),
\]
where the number of free parameters is $\sum_{k\in[p]}|\mathbf{s}_k|(|\mathcal{X}_{\{k\}}|- 1)$.
For example, when all variables are binary, $d$ is the number of stages in $\mathcal{T}$. 
By Corollary~\ref{cor: VP generalization}, the number of free parameters $d$ is the same for any two statistically equivalent staged trees, as the stages in each tree are determined by the edges in their associated context graphs.   
Hence, the BIC is \emph{score equivalent} for CStrees, meaning that $\mathcal{S}(\mathcal{T},\mathbb{D}) = \mathcal{S}(\mathcal{T}^\prime,\mathbb{D})$ whenever $\mathcal{T}$ and $\mathcal{T}^\prime$ are statistically equivalent.  
By Theorem~\ref{thm: markov to context graphs}, CStrees are examples of discrete DAG models with explicit local constraints, and hence are curved exponential models \citep[Theorem 4]{GHKM01}.  
This observation also follows from a recent result of \citet{GLM20} who showed that all staged tree models are curved exponential models.
Thus, it follows from a result in \citep{H88}, that the BIC is \emph{consistent} which means it satisfies the conditions in the next definition:
\begin{definition}
\label{def: consistent}
Let $\mathbb{D}$ be a set of $n$ independent and identically distributed samples drawn from some distribution $\vX = (X_1,\ldots, X_p)$.  
A scoring criterion $\mathcal{S}$ is \emph{consistent} if, as $n\longrightarrow\infty$, the following holds for any two models $\MM,\MM^\prime$:
\begin{enumerate}
	\item $\mathcal{S}(\MM,\mathbb{D})>\mathcal{S}(\MM^\prime,\mathbb{D})$ whenever $\vX\in\MM$ but $\vX\notin\MM^\prime$, and 
	\item $\mathcal{S}(\MM,\mathbb{D})>\mathcal{S}(\MM^\prime,\mathbb{D})$ whenever $\vX\in\MM\cap\MM^\prime$ but $\MM$ has fewer parameters than $\MM^\prime$.  
\end{enumerate}
\end{definition}

\section{Enumeration of CStrees}
\label{appsec: enumeration}
Similar to DAGs, the number of CStrees grows super-exponentially in the number of variables $p$.  
The number of CStrees for $1,2,3,4$ and $5$ binary variables is depicted side-by-side with the corresponding number of DAGs and (compatibly labeled, stratified) staged trees in Figure~\ref{fig: counting CStrees}.  
\begin{figure}
	\centering
	{
\begin{tabular}{r | r | r |r|}
$p$ &   DAGs    &   CStrees  & Compatibly Labeled Staged Trees\\\hline
1   &   1                   &    1   & 1\\
2   &   3                   &    4   & 4\\
3   &   25                  &    96   & 180\\
4   &   543                 &    59136 &  2980800\\
5   &   29281               &    26466908160&  156196038558888000  \\
6   &   3781503             &    $1.1326\times 10^{22}$  & $1.20019\times 10^{44}$\\
7   &   1138779265          &     ?  & $1.44616\times 10^{110}$\\
8   &   783702329343        &     ? & $1.29814\times 10^{269}$\\
\end{tabular}
	}

	\caption{Number of DAGs, CStrees and compatibly labeled staged trees on $p$  binary variables.}
	\label{fig: counting CStrees}
\end{figure}
We see that the number of CStrees for representing $p$ binary variables is already on the order of $10$'s of billions for $p = 5$, whereas the number of DAGs for representing $p$ variables reaches this order of magnitude around $p = 7$.  
On the other hand, the number of general (compatibly labeled, stratified) staged trees on $p = 5$ binary variables is already on the order of $100$'s of millions of billions.
It is well-known \citep{CS14} and easy to verify that the number of (compatibly labeled, stratified) staged trees on $p$ binary variables is 
\[
p!\prod_{k = 1}^{p-1} B_{2^k}
\]
where $B_p$ is the \emph{$p^{th}$ Bell number} \citep[A000110]{OEIS}.
A similar formula holds for the number of CStrees: 

Let $[0,1]^p$ denote the $p$-dimensional cube that is given by the convex hull of all $(0,1)$-vectors in $\R^p$.  
We define the \emph{$(p+1)^{st}$ cubical Bell number} to be the number of ways to partition the vertices of $[0,1]^p$ into non-overlapping faces of the cube.  
We have the following proposition and relate the numbers $B^{(c)}_{p}$ to the classical Bell numbers $B_p$. 
\begin{proposition}
\label{prop: counting CStrees}
The number of CStrees on $p$ binary variables is
$
p!\prod_{k = 1}^p B_k^{(c)}.
$
\end{proposition}

\begin{proof}
The \emph{$p^{th}$ Bell number} $B_p$ counts the number of partitions of the $p$-set $\{1,\ldots,p\}$.  
Equivalently, it is number of ways to divide the vertices of a $(p-1)$-dimensional simplex into non-overlapping faces.  
Since CStrees are stratified, in order to enumerate them, one need only determine the number of possible ways to partition the nodes in level $k$ into stages that satisfy the condition given in Definition~\ref{def:CStree}, take the product of these numbers for $k = 1,\ldots, p-1$, and then multiply by $p!$ to account for the different possible causal orderings of the variables. 
Hence, to get a formula for the number of CStrees on $p$ binary variables we need to determine the number of possible ways to partition the nodes in level $k$ into stages that satisfy the condition given in Definition~\ref{def:CStree}.  
Since each variable is binary, there are exactly $2^k$ vertices in level $k$ of the CStree and each vertex corresponds to a unique vertex of the $k$-dimensional cube $[0,1]^k$.  
More generally a $(k - t)$-dimensional face of the cube $[0,1]^k$ is specified by fixing $t$ coordinates to be equal to either $0$ or $1$. 
Hence, a $(k-t)$-dimensional face corresponds to a stage $\mathcal{S}_{\pi,k}(\vx_S)$ with $|S| = t$. 
Since the stages cannot overlap, we have that a staging $\mathbf{s}_k$ of level $k$ corresponds to a collection of non-overlapping faces of $[0,1]^k$; i.e., the $(k+1)^{st}$ cubical Bell number $B_{k+1}^{(c)}$.  
\end{proof}

The cubical Bell numbers are known only for small values \cite[A018926]{OEIS}.  
For $p = 1,\ldots,6$ they are $1, 2, 8, 154, 89512, 71319425714$.
The number of CStrees in Figure~\ref{fig: counting CStrees} are computed via these values and Proposition~\ref{prop: counting CStrees}.  
Given Proposition~\ref{prop: counting CStrees}, it would be of interest to have a closed-form formula for $B_p^{(c)}$.

\clearpage

\section{Additional Figures for the Real Data Examples}
\label{appsec: additional results}

\begin{figure}[h]
	\centering
\begin{tikzpicture}[thick,scale=0.2]
	
	 \node[circle, draw, fill=black!0, inner sep=2pt, minimum width=2pt] (w3) at (8,15)  {};
 	 \node[circle, draw, fill=black!0, inner sep=2pt, minimum width=2pt] (w4) at (8,13) {};
 	 \node[circle, draw, fill=black!0, inner sep=2pt, minimum width=2pt] (w5) at (8,11) {};
 	 \node[circle, draw, fill=black!0, inner sep=2pt, minimum width=2pt] (w6) at (8,9) {};
 	 \node[circle, draw, fill=black!0, inner sep=2pt, minimum width=2pt] (v3) at (8,7)  {};
 	 \node[circle, draw, fill=black!0, inner sep=2pt, minimum width=2pt] (v4) at (8,5) {};
 	 \node[circle, draw, fill=black!0, inner sep=2pt, minimum width=2pt] (v5) at (8,3) {};
 	 \node[circle, draw, fill=black!0, inner sep=2pt, minimum width=2pt] (v6) at (8,1) {};
 	 \node[circle, draw, fill=black!0, inner sep=2pt, minimum width=2pt] (w3i) at (8,-1)  {};
 	 \node[circle, draw, fill=black!0, inner sep=2pt, minimum width=2pt] (w4i) at (8,-3) {};
 	 \node[circle, draw, fill=black!0, inner sep=2pt, minimum width=2pt] (w5i) at (8,-5) {};
 	 \node[circle, draw, fill=black!0, inner sep=2pt, minimum width=2pt] (w6i) at (8,-7) {};
	 
 	 \node[circle, draw, fill=black!0, inner sep=2pt, minimum width=2pt] (v3i) at (8,-9)  {};
 	 \node[circle, draw, fill=black!0, inner sep=2pt, minimum width=2pt] (v4i) at (8,-11) {};
 	 \node[circle, draw, fill=black!0, inner sep=2pt, minimum width=2pt] (v5i) at (8,-13) {};
 	 \node[circle, draw, fill=black!0, inner sep=2pt, minimum width=2pt] (v6i) at (8,-15) {}; 	 
	 
	 \node[circle, draw, fill=yellow!0, inner sep=2pt, minimum width=2pt] (w1) at (0,14) {};
 	 \node[circle, draw, fill=orange!70, inner sep=2pt, minimum width=2pt] (w2) at (0,10) {}; 

 	 \node[circle, draw, fill=yellow!00, inner sep=2pt, minimum width=2pt] (v1) at (0,6) {};
 	 \node[circle, draw, fill=orange!70, inner sep=2pt, minimum width=2pt] (v2) at (0,2) {};	 
 	 \node[circle, draw, fill=violet!70, inner sep=2pt, minimum width=2pt] (w1i) at (0,-2) {};
 	 \node[circle, draw, fill=violet!70, inner sep=2pt, minimum width=2pt] (w2i) at (0,-6) {};

 	 \node[circle, draw, fill=violet!70, inner sep=2pt, minimum width=2pt] (v1i) at (0,-10) {};
 	 \node[circle, draw, fill=violet!70, inner sep=2pt, minimum width=2pt] (v2i) at (0,-14) {};

 	 \node[circle, draw, fill=cyan!90, inner sep=2pt, minimum width=2pt] (w) at (-8,12) {};
 	 \node[circle, draw, fill=cyan!90, inner sep=2pt, minimum width=2pt] (v) at (-8,4) {};
 	 \node[circle, draw, fill=cyan!0, inner sep=2pt, minimum width=2pt] (wi) at (-8,-4) {};
 	 \node[circle, draw, fill=red!0, inner sep=2pt, minimum width=2pt] (vi) at (-8,-12) {};

 	 \node[circle, draw, fill=lime!70, inner sep=2pt, minimum width=2pt] (r) at (-16,8) {};
         \node[circle, draw, fill=lime!70, inner sep=2pt, minimum width=2pt] (ri) at (-16,-8) {};

 	 \node[circle, draw, fill=black!0, inner sep=2pt, minimum width=2pt] (I) at (-22,0) {};

 	 \draw[->]   (I) --    (r) ;
 	 \draw[->]   (I) --   (ri) ;

 	 \draw[->]   (r) --   (w) ;
 	 \draw[->]   (r) --   (v) ;

 	 \draw[->]   (w) --  (w1) ;
 	 \draw[->]   (w) --  (w2) ;

 	 \draw[->]   (w1) --   (w3) ;
 	 \draw[->]   (w1) --   (w4) ;
 	 \draw[->]   (w2) --  (w5) ;
 	 \draw[->]   (w2) --  (w6) ;

 	 \draw[->]   (v) --  (v1) ;
 	 \draw[->]   (v) --  (v2) ;

 	 \draw[->]   (v1) --  (v3) ;
 	 \draw[->]   (v1) --  (v4) ;
 	 \draw[->]   (v2) --  (v5) ;
 	 \draw[->]   (v2) --  (v6) ;

 	 \draw[->]   (ri) --   (wi) ;
 	 \draw[->]   (ri) -- (vi) ;

 	 \draw[->]   (wi) --  (w1i) ;
 	 \draw[->]   (wi) --  (w2i) ;

 	 \draw[->]   (w1i) --  (w3i) ;
 	 \draw[->]   (w1i) -- (w4i) ;
 	 \draw[->]   (w2i) --  (w5i) ;
 	 \draw[->]   (w2i) --  (w6i) ;

 	 \draw[->]   (vi) --  (v1i) ;
 	 \draw[->]   (vi) --  (v2i) ;

 	 \draw[->]   (v1i) --  (v3i) ;
 	 \draw[->]   (v1i) -- (v4i) ;
 	 \draw[->]   (v2i) -- (v5i) ;
 	 \draw[->]   (v2i) --  (v6i) ;
	 
	 \node at (-20,-17) {\scriptsize pCAMKII} ;
	 \node at (-12,-17) {\scriptsize pS6} ;
	 \node at (-4,-17) {\scriptsize pPKCG} ; 
	 \node at (4,-17) {\scriptsize NR1} ; 

\end{tikzpicture}
	\vspace{-0.2cm}
	\caption{An element of the BIC-optimal equivalence class of CStrees for the observational data for the proteins pCAMKII, pPKCG, NR1 and pS6.}
	\label{fig: BIC optimal observational tree 1}
\end{figure}

	\begin{figure}[h!]
	\centering

\begin{tikzpicture}[thick,scale=0.3]
	
 	 \node[circle, draw, fill=black!00, inner sep=2pt, minimum width=2pt] (a1) at (4,0.5)  {};
 	 \node[circle, draw, fill=black!00, inner sep=2pt, minimum width=2pt] (a2) at (4,-0.5)  {};
 	 \node[circle, draw, fill=black!00, inner sep=2pt, minimum width=2pt] (a3) at (4,-1) {};
 	 \node[circle, draw, fill=black!00, inner sep=2pt, minimum width=2pt] (a4) at (4,-2) {};
 	 \node[circle, draw, fill=black!00, inner sep=2pt, minimum width=2pt] (a5) at (4,-2.5) {};
 	 \node[circle, draw, fill=black!00, inner sep=2pt, minimum width=2pt] (a6) at (4,-3.5) {};
 	 \node[circle, draw, fill=black!00, inner sep=2pt, minimum width=2pt] (a7) at (4,-4) {};
 	 \node[circle, draw, fill=black!00, inner sep=2pt, minimum width=2pt] (a8) at (4,-5) {};
 	 \node[circle, draw, fill=black!00, inner sep=2pt, minimum width=2pt] (a9) at (4,-5.5)  {};
 	 \node[circle, draw, fill=black!00, inner sep=2pt, minimum width=2pt] (a10) at (4,-6.5)  {};
 	 \node[circle, draw, fill=black!00, inner sep=2pt, minimum width=2pt] (a11) at (4,-7) {};
 	 \node[circle, draw, fill=black!00, inner sep=2pt, minimum width=2pt] (a12) at (4,-8) {};
 	 \node[circle, draw, fill=black!00, inner sep=2pt, minimum width=2pt] (a13) at (4,-8.5) {};
 	 \node[circle, draw, fill=black!00, inner sep=2pt, minimum width=2pt] (a14) at (4,-9.5) {};
 	 \node[circle, draw, fill=black!00, inner sep=2pt, minimum width=2pt] (a15) at (4,-10) {};
 	 \node[circle, draw, fill=black!00, inner sep=2pt, minimum width=2pt] (a16) at (4,-11) {};

 	 \node[circle, draw, fill=black!00, inner sep=2pt, minimum width=2pt] (b1) at (4,0.5-12)  {};
 	 \node[circle, draw, fill=black!00, inner sep=2pt, minimum width=2pt] (b2) at (4,-0.5-12)  {};
 	 \node[circle, draw, fill=black!00, inner sep=2pt, minimum width=2pt] (b3) at (4,-1-12) {};
 	 \node[circle, draw, fill=black!00, inner sep=2pt, minimum width=2pt] (b4) at (4,-2-12) {};
 	 \node[circle, draw, fill=black!00, inner sep=2pt, minimum width=2pt] (b5) at (4,-2.5-12) {};
 	 \node[circle, draw, fill=black!00, inner sep=2pt, minimum width=2pt] (b6) at (4,-3.5-12) {};
 	 \node[circle, draw, fill=black!00, inner sep=2pt, minimum width=2pt] (b7) at (4,-4-12) {};
 	 \node[circle, draw, fill=black!00, inner sep=2pt, minimum width=2pt] (b8) at (4,-5-12) {};
 	 \node[circle, draw, fill=black!00, inner sep=2pt, minimum width=2pt] (b9) at (4,-5.5-12)  {};
 	 \node[circle, draw, fill=black!00, inner sep=2pt, minimum width=2pt] (b10) at (4,-6.5-12)  {};
 	 \node[circle, draw, fill=black!00, inner sep=2pt, minimum width=2pt] (b11) at (4,-7-12) {};
 	 \node[circle, draw, fill=black!00, inner sep=2pt, minimum width=2pt] (b12) at (4,-8-12) {};
 	 \node[circle, draw, fill=black!00, inner sep=2pt, minimum width=2pt] (b13) at (4,-8.5-12) {};
 	 \node[circle, draw, fill=black!00, inner sep=2pt, minimum width=2pt] (b14) at (4,-9.5-12) {};
 	 \node[circle, draw, fill=black!00, inner sep=2pt, minimum width=2pt] (b15) at (4,-10-12) {};
 	 \node[circle, draw, fill=black!00, inner sep=2pt, minimum width=2pt] (b16) at (4,-11-12) {};

 	 \node[circle, draw, fill=pink!70, inner sep=2pt, minimum width=2pt] (w3) at (0,0)  {};
 	 \node[circle, draw, fill=orange!80, inner sep=2pt, minimum width=2pt] (w4) at (0,-1.5) {};
 	 \node[circle, draw, fill=yellow!60, inner sep=2pt, minimum width=2pt] (w5) at (0,-3) {};
 	 \node[circle, draw, fill=orange!80, inner sep=2pt, minimum width=2pt] (w6) at (0,-4.5) {};
	 
 	 \node[circle, draw, fill=violet!70, inner sep=2pt, minimum width=2pt] (v3) at (0,-6)  {};
 	 \node[circle, draw, fill=violet!70, inner sep=2pt, minimum width=2pt] (v4) at (0,-7.5) {};
 	 \node[circle, draw, fill=violet!70, inner sep=2pt, minimum width=2pt] (v5) at (0,-9) {};
 	 \node[circle, draw, fill=violet!70, inner sep=2pt, minimum width=2pt] (v6) at (0,-10.5) {};

	 \node[circle, draw, fill=cyan!90, inner sep=2pt, minimum width=2pt] (w1) at (-8,-.75) {};
 	 \node[circle, draw, fill=cyan!90, inner sep=2pt, minimum width=2pt] (w2) at (-8,-3.75) {}; 

 	 \node[circle, draw, fill=red!70, inner sep=2pt, minimum width=2pt] (v1) at (-8,-6.75) {};
 	 \node[circle, draw, fill=olive!70, inner sep=2pt, minimum width=2pt] (v2) at (-8,-9.75) {};

 	 \node[circle, draw, fill=lime!70, inner sep=2pt, minimum width=2pt] (w) at (-16,-2.25) {};

 	 \node[circle, draw, fill=lime!70, inner sep=2pt, minimum width=2pt] (v) at (-16,-8.25) {};

 	 \node[circle, draw, fill=teal!70, inner sep=2pt, minimum width=2pt] (r) at (-20,-5.25) {};

 	 \node[circle, draw, fill=pink!70, inner sep=2pt, minimum width=2pt] (w3i) at (0,-12)  {};
 	 \node[circle, draw, fill=orange!80, inner sep=2pt, minimum width=2pt] (w4i) at (0,-13.5) {};
 	 \node[circle, draw, fill=yellow!70, inner sep=2pt, minimum width=2pt] (w5i) at (0,-15) {};
 	 \node[circle, draw, fill=orange!80, inner sep=2pt, minimum width=2pt] (w6i) at (0,-16.5) {};
	 
 	 \node[circle, draw, fill=violet!70, inner sep=2pt, minimum width=2pt] (v3i) at (0,-18)  {};
 	 \node[circle, draw, fill=violet!70, inner sep=2pt, minimum width=2pt] (v4i) at (0,-19.5) {};
 	 \node[circle, draw, fill=violet!70, inner sep=2pt, minimum width=2pt] (v5i) at (0,-21) {};
 	 \node[circle, draw, fill=violet!70, inner sep=2pt, minimum width=2pt] (v6i) at (0,-22.5) {};

	 \node[circle, draw, fill=green!40, inner sep=2pt, minimum width=2pt] (w1i) at (-8,-12.75) {};
 	 \node[circle, draw, fill=green!40, inner sep=2pt, minimum width=2pt] (w2i) at (-8,-15.75) {};

 	 \node[circle, draw, fill=red!70, inner sep=2pt, minimum width=2pt] (v1i) at (-8,-18.75) {};
 	 \node[circle, draw, fill=olive!70, inner sep=2pt, minimum width=2pt] (v2i) at (-8,-21.75) {};

 	 \node[circle, draw, fill=lime!70, inner sep=2pt, minimum width=2pt] (wi) at (-16,-14.25) {};

 	 \node[circle, draw, fill=lime!70, inner sep=2pt, minimum width=2pt] (vi) at (-16,-20.25) {};

 	 \node[circle, draw, fill=teal!70, inner sep=2pt, minimum width=2pt] (ri) at (-20,-17.25) {};

 	 \node[circle, draw, fill=black!100, inner sep=2pt, minimum width=2pt] (I) at (-26,-11.25) {};

 	 \draw[->]   (w3) --   (a1) ;
 	 \draw[->]   (w3) --   (a2) ;
 	 \draw[->]   (w4) --   (a3) ;
 	 \draw[->]   (w4) --   (a4) ;
 	 \draw[->]   (w5) --   (a5) ;
 	 \draw[->]   (w5) --   (a6) ;
 	 \draw[->]   (w6) --   (a7) ;
 	 \draw[->]   (w6) --   (a8) ;
 	 \draw[->]   (v3) --   (a9) ;
 	 \draw[->]   (v3) --   (a10) ;
 	 \draw[->]   (v4) --   (a11) ;
 	 \draw[->]   (v4) --   (a12) ;
 	 \draw[->]   (v5) --   (a13) ;
 	 \draw[->]   (v5) --   (a14) ;
 	 \draw[->]   (v6) --   (a15) ;
 	 \draw[->]   (v6) --   (a16) ;
 	 
 	 \draw[->]   (w3i) --   (b1) ;
 	 \draw[->]   (w3i) --   (b2) ;
 	 \draw[->]   (w4i) --   (b3) ;
 	 \draw[->]   (w4i) --   (b4) ;
 	 \draw[->]   (w5i) --   (b5) ;
 	 \draw[->]   (w5i) --   (b6) ;
 	 \draw[->]   (w6i) --   (b7) ;
 	 \draw[->]   (w6i) --   (b8) ;
 	 \draw[->]   (v3i) --   (b9) ;
 	 \draw[->]   (v3i) --   (b10) ;
 	 \draw[->]   (v4i) --   (b11) ;
 	 \draw[->]   (v4i) --   (b12) ;
 	 \draw[->]   (v5i) --   (b13) ;
 	 \draw[->]   (v5i) --   (b14) ;
 	 \draw[->]   (v6i) --   (b15) ;
 	 \draw[->]   (v6i) --   (b16) ;

 	 \draw[->]   (I) -- node[midway,sloped,above]{\tiny saline}    (r) ;
 	 \draw[->]   (I) -- node[midway,sloped,below]{\tiny memantine}  (ri) ;

 	 \draw[->]   (r) --   (w) ;
 	 \draw[->]   (r) --   (v) ;

 	 \draw[->]   (w) --  (w1) ;
 	 \draw[->]   (w) --  (w2) ;

 	 \draw[->]   (w1) --   (w3) ;
 	 \draw[->]   (w1) --   (w4) ;
 	 \draw[->]   (w2) --  (w5) ;
 	 \draw[->]   (w2) --  (w6) ;

 	 \draw[->]   (v) --  (v1) ;
 	 \draw[->]   (v) --  (v2) ;

 	 \draw[->]   (v1) --  (v3) ;
 	 \draw[->]   (v1) --  (v4) ;
 	 \draw[->]   (v2) --  (v5) ;
 	 \draw[->]   (v2) --  (v6) ;

 	 \draw[->]   (ri) --   (wi) ;
 	 \draw[->]   (ri) -- (vi) ;

 	 \draw[->]   (wi) --  (w1i) ;
 	 \draw[->]   (wi) --  (w2i) ;

 	 \draw[->]   (w1i) --  (w3i) ;
 	 \draw[->]   (w1i) -- (w4i) ;
 	 \draw[->]   (w2i) --  (w5i) ;
 	 \draw[->]   (w2i) --  (w6i) ;

 	 \draw[->]   (vi) --  (v1i) ;
 	 \draw[->]   (vi) --  (v2i) ;

 	 \draw[->]   (v1i) --  (v3i) ;
 	 \draw[->]   (v1i) -- (v4i) ;
 	 \draw[->]   (v2i) -- (v5i) ;
 	 \draw[->]   (v2i) --  (v6i) ;
 	 
    \node at (-24, -25) {$\ci$} ; 
    \node at (-18, -25) {pCAMKII} ; 
    \node at (-12, -25) {pS6} ; 
    \node at (-4, -25) {pPKCG} ; 
    \node at (2, -25) {NR1} ; 

\end{tikzpicture}
	\vspace{-0.2cm}
	\caption{The BIC-optimal interventional CStree on the variables pCAMKII, pPKCG, NR1 and pS6.}
	\label{fig: BIC optimal interventional tree 1}
\end{figure}

\begin{figure}[h!]
	\centering
\begin{tikzpicture}[thick,scale=0.2]
	
	 \node[circle, draw, fill=black!0, inner sep=2pt, minimum width=2pt] (w3) at (8,15)  {};
 	 \node[circle, draw, fill=black!0, inner sep=2pt, minimum width=2pt] (w4) at (8,13) {};
 	 \node[circle, draw, fill=black!0, inner sep=2pt, minimum width=2pt] (w5) at (8,11) {};
 	 \node[circle, draw, fill=black!0, inner sep=2pt, minimum width=2pt] (w6) at (8,9) {};
 	 \node[circle, draw, fill=black!0, inner sep=2pt, minimum width=2pt] (v3) at (8,7)  {};
 	 \node[circle, draw, fill=black!0, inner sep=2pt, minimum width=2pt] (v4) at (8,5) {};
 	 \node[circle, draw, fill=black!0, inner sep=2pt, minimum width=2pt] (v5) at (8,3) {};
 	 \node[circle, draw, fill=black!0, inner sep=2pt, minimum width=2pt] (v6) at (8,1) {};
 	 \node[circle, draw, fill=black!0, inner sep=2pt, minimum width=2pt] (w3i) at (8,-1)  {};
 	 \node[circle, draw, fill=black!0, inner sep=2pt, minimum width=2pt] (w4i) at (8,-3) {};
 	 \node[circle, draw, fill=black!0, inner sep=2pt, minimum width=2pt] (w5i) at (8,-5) {};
 	 \node[circle, draw, fill=black!0, inner sep=2pt, minimum width=2pt] (w6i) at (8,-7) {};
	 
 	 \node[circle, draw, fill=black!0, inner sep=2pt, minimum width=2pt] (v3i) at (8,-9)  {};
 	 \node[circle, draw, fill=black!0, inner sep=2pt, minimum width=2pt] (v4i) at (8,-11) {};
 	 \node[circle, draw, fill=black!0, inner sep=2pt, minimum width=2pt] (v5i) at (8,-13) {};
 	 \node[circle, draw, fill=black!0, inner sep=2pt, minimum width=2pt] (v6i) at (8,-15) {}; 	 
	 
	 \node[circle, draw, fill=yellow!70, inner sep=2pt, minimum width=2pt] (w1) at (0,14) {};
 	 \node[circle, draw, fill=yellow!70, inner sep=2pt, minimum width=2pt] (w2) at (0,10) {}; 

 	 \node[circle, draw, fill=yellow!70, inner sep=2pt, minimum width=2pt] (v1) at (0,6) {};
 	 \node[circle, draw, fill=yellow!70, inner sep=2pt, minimum width=2pt] (v2) at (0,2) {};	 
 	 \node[circle, draw, fill=orange!80, inner sep=2pt, minimum width=2pt] (w1i) at (0,-2) {};
 	 \node[circle, draw, fill=violet!70, inner sep=2pt, minimum width=2pt] (w2i) at (0,-6) {};

 	 \node[circle, draw, fill=orange!80, inner sep=2pt, minimum width=2pt] (v1i) at (0,-10) {};
 	 \node[circle, draw, fill=violet!70, inner sep=2pt, minimum width=2pt] (v2i) at (0,-14) {};

 	 \node[circle, draw, fill=cyan!90, inner sep=2pt, minimum width=2pt] (w) at (-8,12) {};
 	 \node[circle, draw, fill=red!80, inner sep=2pt, minimum width=2pt] (v) at (-8,4) {};
 	 \node[circle, draw, fill=cyan!90, inner sep=2pt, minimum width=2pt] (wi) at (-8,-4) {};
 	 \node[circle, draw, fill=red!80, inner sep=2pt, minimum width=2pt] (vi) at (-8,-12) {};

 	 \node[circle, draw, fill=lime!0, inner sep=2pt, minimum width=2pt] (r) at (-16,8) {};
         \node[circle, draw, fill=lime!0, inner sep=2pt, minimum width=2pt] (ri) at (-16,-8) {};

 	 \node[circle, draw, fill=black!0, inner sep=2pt, minimum width=2pt] (I) at (-22,0) {};

 	 \draw[->]   (I) --    (r) ;
 	 \draw[->]   (I) --   (ri) ;

 	 \draw[->]   (r) --   (w) ;
 	 \draw[->]   (r) --   (v) ;

 	 \draw[->]   (w) --  (w1) ;
 	 \draw[->]   (w) --  (w2) ;

 	 \draw[->]   (w1) --   (w3) ;
 	 \draw[->]   (w1) --   (w4) ;
 	 \draw[->]   (w2) --  (w5) ;
 	 \draw[->]   (w2) --  (w6) ;

 	 \draw[->]   (v) --  (v1) ;
 	 \draw[->]   (v) --  (v2) ;

 	 \draw[->]   (v1) --  (v3) ;
 	 \draw[->]   (v1) --  (v4) ;
 	 \draw[->]   (v2) --  (v5) ;
 	 \draw[->]   (v2) --  (v6) ;

 	 \draw[->]   (ri) --   (wi) ;
 	 \draw[->]   (ri) -- (vi) ;

 	 \draw[->]   (wi) --  (w1i) ;
 	 \draw[->]   (wi) --  (w2i) ;

 	 \draw[->]   (w1i) --  (w3i) ;
 	 \draw[->]   (w1i) -- (w4i) ;
 	 \draw[->]   (w2i) --  (w5i) ;
 	 \draw[->]   (w2i) --  (w6i) ;

 	 \draw[->]   (vi) --  (v1i) ;
 	 \draw[->]   (vi) --  (v2i) ;

 	 \draw[->]   (v1i) --  (v3i) ;
 	 \draw[->]   (v1i) -- (v4i) ;
 	 \draw[->]   (v2i) -- (v5i) ;
 	 \draw[->]   (v2i) --  (v6i) ;
	 
	 \node at (-20,-17) {\scriptsize pPKCG} ;
	 \node at (-12,-17) {\scriptsize pNUMB} ;
	 \node at (-4,-17) {\scriptsize pNR1} ; 
	 \node at (4,-17) {\scriptsize pCAMKII} ; 

\end{tikzpicture}
	\vspace{-0.2cm}
	\caption{An element of the BIC optimal equivalence class of CStrees for the observational mice data.}
	\label{fig: BIC optimal observational tree}
\end{figure}

\begin{figure}[h]
    \centering
    	\begin{subfigure}[t]{0.3\textwidth}
	     \centering
	     \begin{tikzpicture}[thick,scale=0.2]
	
	 \node[circle, draw, fill=black!0, inner sep=2pt, minimum width=2pt] (w3) at (5,15)  {};
 	 \node[circle, draw, fill=black!0, inner sep=2pt, minimum width=2pt] (w4) at (5,13) {};
 	 \node[circle, draw, fill=black!0, inner sep=2pt, minimum width=2pt] (w5) at (5,11) {};
 	 \node[circle, draw, fill=black!0, inner sep=2pt, minimum width=2pt] (w6) at (5,9) {};
 	 \node[circle, draw, fill=black!0, inner sep=2pt, minimum width=2pt] (v3) at (5,7)  {};
 	 \node[circle, draw, fill=black!0, inner sep=2pt, minimum width=2pt] (v4) at (5,5) {};
 	 \node[circle, draw, fill=black!0, inner sep=2pt, minimum width=2pt] (v5) at (5,3) {};
 	 \node[circle, draw, fill=black!0, inner sep=2pt, minimum width=2pt] (v6) at (5,1) {};
 	 \node[circle, draw, fill=black!0, inner sep=2pt, minimum width=2pt] (w3i) at (5,-1)  {};
 	 \node[circle, draw, fill=black!0, inner sep=2pt, minimum width=2pt] (w4i) at (5,-3) {};
 	 \node[circle, draw, fill=black!0, inner sep=2pt, minimum width=2pt] (w5i) at (5,-5) {};
 	 \node[circle, draw, fill=black!0, inner sep=2pt, minimum width=2pt] (w6i) at (5,-7) {};
	 
 	 \node[circle, draw, fill=black!0, inner sep=2pt, minimum width=2pt] (v3i) at (5,-9)  {};
 	 \node[circle, draw, fill=black!0, inner sep=2pt, minimum width=2pt] (v4i) at (5,-11) {};
 	 \node[circle, draw, fill=black!0, inner sep=2pt, minimum width=2pt] (v5i) at (5,-13) {};
 	 \node[circle, draw, fill=black!0, inner sep=2pt, minimum width=2pt] (v6i) at (5,-15) {}; 	 
	 
	 \node[circle, draw, fill=yellow!70, inner sep=2pt, minimum width=2pt] (w1) at (0,14) {};
 	 \node[circle, draw, fill=orange!80, inner sep=2pt, minimum width=2pt] (w2) at (0,10) {}; 

 	 \node[circle, draw, fill=yellow!70, inner sep=2pt, minimum width=2pt] (v1) at (0,6) {};
 	 \node[circle, draw, fill=orange!80, inner sep=2pt, minimum width=2pt] (v2) at (0,2) {};	 
 	 \node[circle, draw, fill=yellow!70, inner sep=2pt, minimum width=2pt] (w1i) at (0,-2) {};
 	 \node[circle, draw, fill=violet!70, inner sep=2pt, minimum width=2pt] (w2i) at (0,-6) {};

 	 \node[circle, draw, fill=yellow!70, inner sep=2pt, minimum width=2pt] (v1i) at (0,-10) {};
 	 \node[circle, draw, fill=violet!70, inner sep=2pt, minimum width=2pt] (v2i) at (0,-14) {};

 	 \node[circle, draw, fill=cyan!90, inner sep=2pt, minimum width=2pt] (w) at (-5,12) {};
 	 \node[circle, draw, fill=red!80, inner sep=2pt, minimum width=2pt] (v) at (-5,4) {};
 	 \node[circle, draw, fill=cyan!90, inner sep=2pt, minimum width=2pt] (wi) at (-5,-4) {};
 	 \node[circle, draw, fill=red!80, inner sep=2pt, minimum width=2pt] (vi) at (-5,-12) {};

 	 \node[circle, draw, fill=lime!0, inner sep=2pt, minimum width=2pt] (r) at (-10,8) {};
         \node[circle, draw, fill=lime!0, inner sep=2pt, minimum width=2pt] (ri) at (-10,-8) {};

 	 \node[circle, draw, fill=black!0, inner sep=2pt, minimum width=2pt] (I) at (-13,0) {};

 	 \draw[->]   (I) --    (r) ;
 	 \draw[->]   (I) --   (ri) ;

 	 \draw[->]   (r) --   (w) ;
 	 \draw[->]   (r) --   (v) ;

 	 \draw[->]   (w) --  (w1) ;
 	 \draw[->]   (w) --  (w2) ;

 	 \draw[->]   (w1) --   (w3) ;
 	 \draw[->]   (w1) --   (w4) ;
 	 \draw[->]   (w2) --  (w5) ;
 	 \draw[->]   (w2) --  (w6) ;

 	 \draw[->]   (v) --  (v1) ;
 	 \draw[->]   (v) --  (v2) ;

 	 \draw[->]   (v1) --  (v3) ;
 	 \draw[->]   (v1) --  (v4) ;
 	 \draw[->]   (v2) --  (v5) ;
 	 \draw[->]   (v2) --  (v6) ;

 	 \draw[->]   (ri) --   (wi) ;
 	 \draw[->]   (ri) -- (vi) ;

 	 \draw[->]   (wi) --  (w1i) ;
 	 \draw[->]   (wi) --  (w2i) ;

 	 \draw[->]   (w1i) --  (w3i) ;
 	 \draw[->]   (w1i) -- (w4i) ;
 	 \draw[->]   (w2i) --  (w5i) ;
 	 \draw[->]   (w2i) --  (w6i) ;

 	 \draw[->]   (vi) --  (v1i) ;
 	 \draw[->]   (vi) --  (v2i) ;

 	 \draw[->]   (v1i) --  (v3i) ;
 	 \draw[->]   (v1i) -- (v4i) ;
 	 \draw[->]   (v2i) -- (v5i) ;
 	 \draw[->]   (v2i) --  (v6i) ;
	 
	 \node at (-12,-17) {\tiny pNR1} ;
	 \node at (-7.75,-17) {\tiny pNUMB} ;
	 \node at (-3,-17) {\tiny pPKCG} ; 
	 \node at (2.5,-17) {\tiny pCAMKII} ; 

\end{tikzpicture}
\end{subfigure}
\begin{subfigure}[t]{0.3\textwidth}
	     \centering
\begin{tikzpicture}[thick,scale=0.2]
	
	 \node[circle, draw, fill=black!0, inner sep=2pt, minimum width=2pt] (w3) at (5,15)  {};
 	 \node[circle, draw, fill=black!0, inner sep=2pt, minimum width=2pt] (w4) at (5,13) {};
 	 \node[circle, draw, fill=black!0, inner sep=2pt, minimum width=2pt] (w5) at (5,11) {};
 	 \node[circle, draw, fill=black!0, inner sep=2pt, minimum width=2pt] (w6) at (5,9) {};
 	 \node[circle, draw, fill=black!0, inner sep=2pt, minimum width=2pt] (v3) at (5,7)  {};
 	 \node[circle, draw, fill=black!0, inner sep=2pt, minimum width=2pt] (v4) at (5,5) {};
 	 \node[circle, draw, fill=black!0, inner sep=2pt, minimum width=2pt] (v5) at (5,3) {};
 	 \node[circle, draw, fill=black!0, inner sep=2pt, minimum width=2pt] (v6) at (5,1) {};
 	 \node[circle, draw, fill=black!0, inner sep=2pt, minimum width=2pt] (w3i) at (5,-1)  {};
 	 \node[circle, draw, fill=black!0, inner sep=2pt, minimum width=2pt] (w4i) at (5,-3) {};
 	 \node[circle, draw, fill=black!0, inner sep=2pt, minimum width=2pt] (w5i) at (5,-5) {};
 	 \node[circle, draw, fill=black!0, inner sep=2pt, minimum width=2pt] (w6i) at (5,-7) {};
	 
 	 \node[circle, draw, fill=black!0, inner sep=2pt, minimum width=2pt] (v3i) at (5,-9)  {};
 	 \node[circle, draw, fill=black!0, inner sep=2pt, minimum width=2pt] (v4i) at (5,-11) {};
 	 \node[circle, draw, fill=black!0, inner sep=2pt, minimum width=2pt] (v5i) at (5,-13) {};
 	 \node[circle, draw, fill=black!0, inner sep=2pt, minimum width=2pt] (v6i) at (5,-15) {}; 	 
	 
	 \node[circle, draw, fill=yellow!70, inner sep=2pt, minimum width=2pt] (w1) at (0,14) {};
 	 \node[circle, draw, fill=yellow!70, inner sep=2pt, minimum width=2pt] (w2) at (0,10) {}; 

 	 \node[circle, draw, fill=orange!80, inner sep=2pt, minimum width=2pt] (v1) at (0,6) {};
 	 \node[circle, draw, fill=violet!70, inner sep=2pt, minimum width=2pt] (v2) at (0,2) {};	 
 	 \node[circle, draw, fill=yellow!70, inner sep=2pt, minimum width=2pt] (w1i) at (0,-2) {};
 	 \node[circle, draw, fill=yellow!70, inner sep=2pt, minimum width=2pt] (w2i) at (0,-6) {};

 	 \node[circle, draw, fill=orange!80, inner sep=2pt, minimum width=2pt] (v1i) at (0,-10) {};
 	 \node[circle, draw, fill=violet!70, inner sep=2pt, minimum width=2pt] (v2i) at (0,-14) {};

 	 \node[circle, draw, fill=cyan!90, inner sep=2pt, minimum width=2pt] (w) at (-5,12) {};
 	 \node[circle, draw, fill=cyan!90, inner sep=2pt, minimum width=2pt] (v) at (-5,4) {};
 	 \node[circle, draw, fill=red!80, inner sep=2pt, minimum width=2pt] (wi) at (-5,-4) {};
 	 \node[circle, draw, fill=red!80, inner sep=2pt, minimum width=2pt] (vi) at (-5,-12) {};

 	 \node[circle, draw, fill=lime!0, inner sep=2pt, minimum width=2pt] (r) at (-10,8) {};
         \node[circle, draw, fill=lime!0, inner sep=2pt, minimum width=2pt] (ri) at (-10,-8) {};

 	 \node[circle, draw, fill=black!0, inner sep=2pt, minimum width=2pt] (I) at (-13,0) {};

 	 \draw[->]   (I) --    (r) ;
 	 \draw[->]   (I) --   (ri) ;

 	 \draw[->]   (r) --   (w) ;
 	 \draw[->]   (r) --   (v) ;

 	 \draw[->]   (w) --  (w1) ;
 	 \draw[->]   (w) --  (w2) ;

 	 \draw[->]   (w1) --   (w3) ;
 	 \draw[->]   (w1) --   (w4) ;
 	 \draw[->]   (w2) --  (w5) ;
 	 \draw[->]   (w2) --  (w6) ;

 	 \draw[->]   (v) --  (v1) ;
 	 \draw[->]   (v) --  (v2) ;

 	 \draw[->]   (v1) --  (v3) ;
 	 \draw[->]   (v1) --  (v4) ;
 	 \draw[->]   (v2) --  (v5) ;
 	 \draw[->]   (v2) --  (v6) ;

 	 \draw[->]   (ri) --   (wi) ;
 	 \draw[->]   (ri) -- (vi) ;

 	 \draw[->]   (wi) --  (w1i) ;
 	 \draw[->]   (wi) --  (w2i) ;

 	 \draw[->]   (w1i) --  (w3i) ;
 	 \draw[->]   (w1i) -- (w4i) ;
 	 \draw[->]   (w2i) --  (w5i) ;
 	 \draw[->]   (w2i) --  (w6i) ;

 	 \draw[->]   (vi) --  (v1i) ;
 	 \draw[->]   (vi) --  (v2i) ;

 	 \draw[->]   (v1i) --  (v3i) ;
 	 \draw[->]   (v1i) -- (v4i) ;
 	 \draw[->]   (v2i) -- (v5i) ;
 	 \draw[->]   (v2i) --  (v6i) ;
	 
	 \node at (-11,-17) {\tiny pNUMB} ;
	 \node at (-6.25,-17) {\tiny pPKCG} ;
	 \node at (-2,-17) {\tiny pNR1} ; 
	 \node at (2.75,-17) {\tiny pCAMKII} ; 

\end{tikzpicture}
\end{subfigure}
\begin{subfigure}[t]{0.3\textwidth}
	     \centering
	     \begin{tikzpicture}[thick,scale=0.2]
	
	 \node[circle, draw, fill=black!0, inner sep=2pt, minimum width=2pt] (w3) at (5,15)  {};
 	 \node[circle, draw, fill=black!0, inner sep=2pt, minimum width=2pt] (w4) at (5,13) {};
 	 \node[circle, draw, fill=black!0, inner sep=2pt, minimum width=2pt] (w5) at (5,11) {};
 	 \node[circle, draw, fill=black!0, inner sep=2pt, minimum width=2pt] (w6) at (5,9) {};
 	 \node[circle, draw, fill=black!0, inner sep=2pt, minimum width=2pt] (v3) at (5,7)  {};
 	 \node[circle, draw, fill=black!0, inner sep=2pt, minimum width=2pt] (v4) at (5,5) {};
 	 \node[circle, draw, fill=black!0, inner sep=2pt, minimum width=2pt] (v5) at (5,3) {};
 	 \node[circle, draw, fill=black!0, inner sep=2pt, minimum width=2pt] (v6) at (5,1) {};
 	 \node[circle, draw, fill=black!0, inner sep=2pt, minimum width=2pt] (w3i) at (5,-1)  {};
 	 \node[circle, draw, fill=black!0, inner sep=2pt, minimum width=2pt] (w4i) at (5,-3) {};
 	 \node[circle, draw, fill=black!0, inner sep=2pt, minimum width=2pt] (w5i) at (5,-5) {};
 	 \node[circle, draw, fill=black!0, inner sep=2pt, minimum width=2pt] (w6i) at (5,-7) {};
	 
 	 \node[circle, draw, fill=black!0, inner sep=2pt, minimum width=2pt] (v3i) at (5,-9)  {};
 	 \node[circle, draw, fill=black!0, inner sep=2pt, minimum width=2pt] (v4i) at (5,-11) {};
 	 \node[circle, draw, fill=black!0, inner sep=2pt, minimum width=2pt] (v5i) at (5,-13) {};
 	 \node[circle, draw, fill=black!0, inner sep=2pt, minimum width=2pt] (v6i) at (5,-15) {}; 	 
	 
	 \node[circle, draw, fill=yellow!70, inner sep=2pt, minimum width=2pt] (w1) at (0,14) {};
 	 \node[circle, draw, fill=orange!80, inner sep=2pt, minimum width=2pt] (w2) at (0,10) {}; 

 	 \node[circle, draw, fill=yellow!70, inner sep=2pt, minimum width=2pt] (v1) at (0,6) {};
 	 \node[circle, draw, fill=violet!70, inner sep=2pt, minimum width=2pt] (v2) at (0,2) {};	 
 	 \node[circle, draw, fill=yellow!70, inner sep=2pt, minimum width=2pt] (w1i) at (0,-2) {};
 	 \node[circle, draw, fill=orange!80, inner sep=2pt, minimum width=2pt] (w2i) at (0,-6) {};

 	 \node[circle, draw, fill=yellow!70, inner sep=2pt, minimum width=2pt] (v1i) at (0,-10) {};
 	 \node[circle, draw, fill=violet!70, inner sep=2pt, minimum width=2pt] (v2i) at (0,-14) {};

 	 \node[circle, draw, fill=cyan!90, inner sep=2pt, minimum width=2pt] (w) at (-5,12) {};
 	 \node[circle, draw, fill=cyan!90, inner sep=2pt, minimum width=2pt] (v) at (-5,4) {};
 	 \node[circle, draw, fill=red!80, inner sep=2pt, minimum width=2pt] (wi) at (-5,-4) {};
 	 \node[circle, draw, fill=red!80, inner sep=2pt, minimum width=2pt] (vi) at (-5,-12) {};

 	 \node[circle, draw, fill=lime!0, inner sep=2pt, minimum width=2pt] (r) at (-10,8) {};
     \node[circle, draw, fill=lime!0, inner sep=2pt, minimum width=2pt] (ri) at (-10,-8) {};

 	 \node[circle, draw, fill=black!0, inner sep=2pt, minimum width=2pt] (I) at (-13,0) {};

 	 \draw[->]   (I) --    (r) ;
 	 \draw[->]   (I) --   (ri) ;

 	 \draw[->]   (r) --   (w) ;
 	 \draw[->]   (r) --   (v) ;

 	 \draw[->]   (w) --  (w1) ;
 	 \draw[->]   (w) --  (w2) ;

 	 \draw[->]   (w1) --   (w3) ;
 	 \draw[->]   (w1) --   (w4) ;
 	 \draw[->]   (w2) --  (w5) ;
 	 \draw[->]   (w2) --  (w6) ;

 	 \draw[->]   (v) --  (v1) ;
 	 \draw[->]   (v) --  (v2) ;

 	 \draw[->]   (v1) --  (v3) ;
 	 \draw[->]   (v1) --  (v4) ;
 	 \draw[->]   (v2) --  (v5) ;
 	 \draw[->]   (v2) --  (v6) ;

 	 \draw[->]   (ri) --   (wi) ;
 	 \draw[->]   (ri) -- (vi) ;

 	 \draw[->]   (wi) --  (w1i) ;
 	 \draw[->]   (wi) --  (w2i) ;

 	 \draw[->]   (w1i) --  (w3i) ;
 	 \draw[->]   (w1i) -- (w4i) ;
 	 \draw[->]   (w2i) --  (w5i) ;
 	 \draw[->]   (w2i) --  (w6i) ;

 	 \draw[->]   (vi) --  (v1i) ;
 	 \draw[->]   (vi) --  (v2i) ;

 	 \draw[->]   (v1i) --  (v3i) ;
 	 \draw[->]   (v1i) -- (v4i) ;
 	 \draw[->]   (v2i) -- (v5i) ;
 	 \draw[->]   (v2i) --  (v6i) ;
	 
	 \node at (-11.5,-17) {\tiny pNUMB} ;
	 \node at (-7.25,-17) {\tiny pNR1} ;
	 \node at (-3,-17) {\tiny pPKCG} ; 
	 \node at (2.5,-17) {\tiny pCAMKII} ; 

\end{tikzpicture}
\end{subfigure}
	     	\caption{	The three CStrees that are statistically equivalent to the tree depicted in Figure~\ref{fig: BIC optimal observational tree}.}
	\label{fig: other three trees}
\end{figure}


\end{document}